\documentclass[ejs]{imsart}

\doi{10.1214/154957804100000000}
\pubyear{0000}
\volume{0}
\firstpage{1}
\lastpage{1}
\arxiv{1904.03372}

\startlocaldefs
\usepackage{mathrsfs}
\usepackage[
colorlinks=true,%
breaklinks,
hyperindex,%
plainpages=false,%
bookmarksopen=false,%
linkcolor=blue,%
anchorcolor=blue,%
citecolor=blue,%
filecolor=black,%
menucolor=black,%
pagecolor=black,%
urlcolor=blue,%
pdfview=FitH,
pdfstartview=FitH,
hyperfigures,
bookmarksnumbered%
]{hyperref} 
\usepackage{amsmath,amssymb,amsthm,amsfonts,amsbsy,latexsym,dsfont,tikz}

\usepackage{enumerate}

\usepackage{graphicx}
\usepackage{graphics}
\usepackage{caption}
\usepackage{subcaption}
\usepackage{multirow}
\usepackage{amsmath}
\usepackage{hyperref}
\usepackage{url}
\usepackage{enumitem}
\usepackage{arydshln}
\usepackage{mathtools}
\usepackage{algorithm}
\usepackage{algorithmic}

\newcommand{\vzero}[0]{\mathbf 0}
\newcommand{\vone}[0]{\mathbf 1}

\newcommand{\Cov}[0]{\text{Cov}}

\newcommand{\calB}[0]{\mathcal{B}}
\newcommand{\calC}[0]{\mathcal{C}}
\newcommand{\calF}[0]{\mathcal{F}}

\newcommand{\calT}[0]{\mathcal{T}}

\newcommand{\E}[0]{\mathds{E}}

\newcommand{\R}[0]{\mathbb{R}}

\newcommand{\Id}[0]{\text{Id}}

\newcommand{\sign}[0]{\text{sign}}

\newcommand{\argmax}[0]{\text{argmax}}
\newcommand{\argmin}[0]{\text{argmin}}
\newcommand{\ub}[0]{\underline{b}}
\newcommand{\us}[0]{\underline{s}}

\newcommand{\Prob}[0]{\mathds{P}}

\renewcommand{\vec}[0]{\text{vec}}
\newcommand{\indep}{\rotatebox[origin=c]{90}{$\models$}}

\renewcommand{\le}{\leqslant}
\renewcommand{\ge}{\geqslant}

\theoremstyle{plain}
\newtheorem{thm}{Theorem}[section]
\newtheorem{lem}[thm]{Lemma}

\newtheorem{cor}[thm]{Corollary}

\theoremstyle{definition}

\theoremstyle{remark}
\newtheorem{rmk}{Remark}

\makeatletter
\DeclareRobustCommand*\cal{\@fontswitch\relax\mathcal}
\makeatother

\numberwithin{equation}{section}
\usepackage[section]{placeins}

\endlocaldefs

\begin{document}

\begin{frontmatter}

\title{A Robust Bootstrap Change Point Test for High-dimensional Location Parameter
}
\runtitle{Robust change point test}


\begin{aug}
	\author{Mengjia Yu\corref{}\ead[label=e1]{mengjia.yu.uiuc@gmail.com}}
	\and
	\author{Xiaohui Chen\ead[label=e2]{xhchen@illinois.edu}}
	
	\address{Department of Statistics, 
		University of Illinois at Urbana-Champaign, \\
		725 S. Wright Street, Champaign, IL 61820, USA \\ 
		\printead{e1,e2}}
	
	\runauthor{Yu and Chen}
	
\end{aug}

\begin{abstract}
	We consider the problem of change point detection for high-dimensional distributions in a location family when the dimension can be much larger than the sample size. In change point analysis, the widely used cumulative sum (CUSUM) statistics are sensitive to outliers and heavy-tailed distributions. In this paper, we propose a robust, tuning-free (i.e., fully data-dependent), and easy-to-implement change point test that enjoys strong theoretical guarantees. To achieve the robust purpose in a nonparametric setting, we formulate the change point detection in the multivariate $U$-statistics framework with anti-symmetric and nonlinear kernels. Specifically, the within-sample noise is canceled out by anti-symmetry of the kernel, while the signal distortion under certain nonlinear kernels can be controlled such that the between-sample change point signal is magnitude preserving. A (half) jackknife multiplier bootstrap (JMB) tailored to the change point detection setting is proposed to calibrate the distribution of our $\ell^{\infty}$-norm aggregated test statistic. Subject to mild moment conditions on kernels, we derive the uniform rates of convergence for the JMB to approximate the sampling distribution of the test statistic, and analyze its size and power properties. Extensions to multiple change point testing and estimation are discussed with illustration from numerical studies.
\end{abstract}



 \received{\smonth{1} \syear{0000}}

\tableofcontents

\end{frontmatter}


\section{Introduction}

Change point detection problems are commonly seen in many statistical and scientific areas including functional data analysis \cite{aue2009estimation,aston2012detecting}, time series inspection \cite{aue2009break,horvath1999testing,yau2016inference}, panel data study \cite{cho2016change,robbins2011mean,horvath2012change,bai2010common}, with applications to fields of biomedical engineering \cite{aston2012evaluating,zhong2016test}, genomics \cite{wang2011non}, financial revenue returns \cite{aston2018high,chofryzlewicz2015,bai2010common} among many others. Statistical testing and estimation of change points have long history with extensive literature \cite{dette2018estimating,aue2009break,holmes2013nonparametric,aston2018high,barigozzi2018simultaneous,lee2016lasso,lee2017oracle}. This paper studies the problem of change point detection for high-dimensional distributions (i.e., $p \gg n$) from a location family with shift parameter. 
Let $X_{i} \sim F_{i}, i =1,\dots,n$ be a sequence of independent random vectors taking values in $\R^{p}$. Our goal is to test whether or not there is a location shift in the distribution functions $F_{i}$. Precisely, let $\calF = \{F_{\theta}(x) = F(x-\theta) : \theta \in \R^{p}\}$ be a location family indexed by the shift parameter $\theta$, where $F = F_{0}$ is the standard distribution in $\calF$ ($F_0$ is arbitrary). We consider the following hypothesis testing problem: 
\begin{align*}
&& H_0: \ X_i \overset{i.i.d.}{\sim} F \mbox{ versus }  H_1: &\ X_1, \dots, X_m \overset{i.i.d.}{\sim} F \text{ and } X_{m+1}, \dots, X_n \overset{i.i.d.}{\sim} F_{\theta}, \\
&&&\text{ for some (unknown) $m \in \{1,\dots,n-1\}$ and $\theta\neq 0$}.
\end{align*}
An advantage of this model is the flexibility of $\calF$ whose mean parameter can be non-existing. Before highlighting the robustness from it, 
we shall first illustrate below the intuition of constructing a test statistic for separating $H_{0}$ and $H_{1}$. For brevity, we denote $G = F_{\theta}$ (i.e., $G(x) = F(x-\theta)$) for a fixed $\theta$, and $Y_{j} = X_{m+j}, j=1,\dots,n-m$. With this notation, we have $X_{1},\dots,X_{m}$ that are independent and identically distributed (i.i.d.) with distribution $F$ and $Y_{1},\dots,Y_{n-m}$ that are i.i.d.\ with distribution $G$ such that the change point detection problem boils down to the two-sample testing problem for the shift parameter $\theta$ with an unknown change point location $m$. Since $m$ is unknown, we may take all possible ordered pairs in the whole sample $X_{i},i=1,\dots,n$, such that the within-sample noise (i.e., in each $X$ and $Y$ samples, separately) cancels out and the between-sample signal is properly preserved under $H_{1}$. Note that our change point hypothesis on the location family $\calF$ is the same as the location-shift model:
\begin{equation}
\label{eqn:location_shift_model}
X_{i} = \theta \ \vone(i > m) + \xi_{i}, \ i=1,\dots,n, \text{ where } \xi_{i} \overset{i.i.d.}{\sim} F \text{ are random vectors in $\R^{p}$ .}
\end{equation}
Viewing $\theta$ as the mean-shift, a natural choice for detecting the existence of a change point shift is to consider the noise cancellations in the empirical mean differences:
\begin{equation}
\label{eqn:cp_linear_kernel}
U_{n} = \sum_{1 \le i < j \le n} (X_i-X_j). 
\end{equation}
Under $H_{0}$, we have $\E[U_{n}]=0$ so that there is no mean-shift signal contained in $U_{n}$ and the sampling behavior of $U_{n}$ is purely determined by the random noises $\xi_{1},\dots,\xi_{n}$. On the other hand, if $H_{1}$ is true, then $\E[U_{n}] = -m(n-m)\theta$. Thus, if the mean difference $\theta$ between the two samples is large enough to dominate the random behavior of $U_{n}$ (due to noise $\{\xi_i\}_{i=1}^n$) under $H_{0}$, then the statistic would be able to distinguish $H_{0}$ between $H_{1}$. 

In practice, a main concern of using $U_{n}$ in~\eqref{eqn:cp_linear_kernel} is its robustness. Specifically, the (empirical) mean functional is not robust in the sense that its influence function is unbounded. Further, in the high-dimensional setting, robustness is a challenging issue since information contained in the data is rather limited. To address this problem, we view the shift signal $\theta$ as a more general location parameter in the distribution family $\calF$ without referring to the means. This simple observation brings a major advantage that change point detection can be made possible even in cases where the means are undefined (such as the Cauchy distribution). To achieve the robustness purpose in a nonparametric setting, we consider a general nonlinear form of \eqref{eqn:cp_linear_kernel} in the $U$-statistics framework. Let $h: \R^p \times \R^p \rightarrow \R^d$ be an \textit{anti-symmetric} kernel, i.e., $h(x,y) = -h(y,x)$ for all $x, y \in \R^p$. We propose the statistic 
\begin{equation}
\label{eqn:Tn_scaled}
T_n = T_n(X_1^n) = {n}^{1/2}{n \choose 2}^{-1} \sum_{1 \le i < j \le n} h(X_i, X_j) 
\end{equation}
to test for $H_{0}$ against $H_{1}$. Clearly, $T_{n}$ is a (scaled) $U$-statistic of order two. The anti-symmetry of the kernel $h$ plays a key role in testing for the change point in terms of noise cancellations. To see this, under $H_0$ we have $\E[h(X_1,X_2)] = 0$ and $\E[T_{n}] = 0$. Observe that 
\begin{equation*}
T_n  = {2 \over {n}^{1/2}(n-1)} \left\{ \sum_{1\le i<j \le m} h(X_i,X_j) \ + \ \sum_{i=1}^m \sum_{j=1}^{n-m} h(X_i, Y_j) \ + \sum_{\mathclap{1 \le i<j \le n-m} }h(Y_i,Y_j) \right\}. 
\end{equation*}
Thus if $H_{1}$ is true, then $\E[T_{n}] \approx 2n^{-3/2} m(n-m)\theta_{h}$, where $\theta_{h} = \E[h(X_{1},Y_{1})]$ is the change point signal through the kernel $h$. If $\theta_{h}$ has a suitable lower bound, then we expect that $T_{n}$ can separate $H_{0}$ and $H_{1}$. For instance, consider the sign kernel $h(x,y) = \sign(x-y)$, where $\sign(x)$ is the component-wise sign operator of $x \in \R^p$ (i.e., for $j=1,\dots,p$, $\sign(x_{j}) = -1, 0, 1$ if $x_{j} < 0$, $x_{j} = 0$, $x_{j} > 0$, respectively). Then,  
\[
\theta_{h,j} = \E [\sign(X_{1,j}-Y_{1,j})] = 1- 2 \Prob (X_{1,j} \le Y_{1,j}) = 1 -2 \Prob (\Delta_{j} \le \theta_j), 
\]
where $\Delta_{j} = \xi_{1,j} - \xi_{m+1,j}$ is a random variable with symmetric distribution. In particular, if $F$ is the distribution in $\R^{p}$ with independent components such that each component admits a continuous probability density function $\phi_{j}, j =1,\dots,p$, then under local alternatives (i.e., $\theta \approx \vzero$) we have $\theta_{h,j} \approx -2 \ \phi_j^*(0) \ \theta_j$, where $\phi_{j}^*$ is the convolution of the densities of $\xi_{1,j}$ and $-\xi_{m+1,j}$. Hence, $\theta_{h}$ and $\theta$ have the same magnitude, implying that signal distortion under the sign kernel is only up to a multiplicative constant.  

The mean difference statistic $U_{n}$ in \eqref{eqn:cp_linear_kernel} is a special case of $T_{n}$ with the linear kernel $h(x_{1}, x_{2}) = x_{1} - x_{2}$ and $d = p$. The sign kernel $h(x,y) = \sign(x-y)$ considered above is another important anti-symmetric and bounded kernel, which is useful if the means are not robust or undefined. Specifically, for the sign kernel, component-wise median of $T_{n}$ corresponds to the Hodges-Lehmann estimator for the component-wise population median of the location difference before and after the change point \cite{HodgesLehmann1963_AMS}. In the univariate case $p = d = 1$, it is known that the Hodges-Lehmann estimator is a highly robust version of sample mean difference (with the linear kernel) against heavy-tailed distributions, and it has a much higher asymptotic relative efficiency $3 / \pi \approx 95\%$ (with respect to the mean) than the sample median at normality \cite{vogelwendler2017}. In addition, when the change point location $m$ is known, $T_{n}$ is also equivalent to the classical nonparametric Mann-Whitney test statistic (see e.g., Chapter 12 in \cite{vandeVaart1998_AsymStat}). 

Since $T_{n}$ is a $d$-dimensional random vector, we need to aggregate its components to make a decision rule for hypothesis testing.  We construct the critical regions based on the Kolmogorov-Smirnov (i.e., the $\ell^{\infty}$-norm) type aggregation of $T_{n}$, namely our change point test statistic is 
\begin{equation}
\label{eqn:T_n_bar}
\overline{T}_n :=  |T_n|_\infty = \max_{1 \le k \le d} |T_{nk}|.
\end{equation} 
Then $H_{0}$ is rejected if $\overline{T}_{n}$ is larger than a critical value such as the $(1-\alpha)$ quantile of $\overline{T}_{n}$. In Section~\ref{sec:bootstrap}, we will introduce a (Gaussian) multiplier bootstrap to calibrate the distribution of $\overline{T}_{n}$, and we will establish its non-asymptotic validity in the high-dimensional setting in Section~\ref{sec:main_results}. 

We point out that our test statistic has comparable computational and statistical properties to the widely used cumulative sum (CUSUM) procedures in literature. For a classical treatment of the CUSUM (and other change point) statistics, we refer to \cite{csorgohorvath1997} as a monograph on the change point analysis. The CUSUM statistics are defined as a sequence of (dependent) random vectors in $\R^{p}$ of the form 
\begin{equation}
\label{eqn:cusum_mean_Rp}
\qquad Z_n(s) = \left({s(n-s) \over n}\right)^{1/2} \left({1\over s} \sum_{i=1}^s X_i - {1 \over n-s} \sum_{i=s+1}^n X_i \right), \quad s=1,\dots,n-1. 
\end{equation}
It is obvious that the CUSUM statistics have a sequential nature in that the left and right sample averages are examined along all possible change point locations, which is necessary to estimate the location $m$. However, if the goal is only testing for the existence of a change point, this (local) sequential comparison strategy is not as efficient as a global test \eqref{eqn:Tn_scaled}, both computationally and statistically. Consider $d = p$, which is the case for the sign and linear kernels. For a general nonlinear kernel, computational cost is $O(n^{2}p)$ for $T_{n}$ (and also for $\overline{T}_{n}$). If the kernel is linear (i.e., $h(x,y)=x-y$), then the computational cost can be further reduced to $O(np)$ for $T_{n}$ effortlessly. Thus we call $T_{n}$ is the global \emph{one-pass} Mann-Whitney type test statistic. In contrast, the computational cost for $\{Z_{n}(s)\}_{s=1}^{n-1}$ is $O(n^{2}p)$ which can reduces to $O(np)$ \cite{killick2014changepoint} via dynamic programming. Statistically, it has been shown in \cite{yuchen2017finite,jirak2015} that a boundary removal procedure is needed for the (bootstrapped) CUSUM change point test to achieve the size validity since the distributions of $Z_{n}(s)$ are difficult to approximate at the boundary points. On the contrary, the test statistic $T_{n}$ proposed in this paper does not remove any boundary points because we are able to approximate the distribution of $T_{n}$ based on majority of the data points in the sample $X_{1},\dots,X_{n}$. Thus it is expected that $\overline{T}_{n}$ achieves faster rate of convergence in the error-in-size for the bootstrap calibration. See Remark~\ref{rmk:no_boundary_remove_H0} ahead for a detailed comparison.

\subsection{Literature review and our contribution}
Single change point inference has been extensively studied in literature such as \cite{csorgohorvath1997, gombay2002rates,horvath1993maximum} for univariate or fixed multivariate setting. 

	Using anti-symmetric kernels in $U$-statistics for location change can be traced back to \cite{Pettitt1979}, which considered a CUSUM-type sequence of two-sample Mann-Whitney statistics with the sign kernel and took the maximum absolute value along the sequence as the test statistic. Asymptotic properties of such statistic for univariate data have been studied in the settings of online and offline change point problems \cite{csorgHo1989invariance,gombay2001u,HawkinsDeng2010,kirch2019sequential}.
	To the best of our knowledge, the proposed global one-pass Mann-Whitney type change point detection procedure in~\eqref{eqn:Tn_scaled} based on a general anti-symmetric kernel without using a CUSUM-type sequence is new in literature, even in the one-dimensional case. 

Second, owing to increasing ability to handle large dimensional data, the focus migrates to a more challenging stage in high dimension that allows $p\rightarrow\infty$ faster than $n$.  Therefore, signal aggregation across dimension becomes influential in the designing of statistics and algorithm. 
For instance, \cite{jirak2015,yuchen2017finite,wangsamworth2017} dealt with sparse change (i.e. mean structure changes in a sparse subset of coordinates), while \cite{bai2010common,horvath2012change,enikeevaharchaoui2014} considered $\ell^2$-type aggregation for dense change. Taking both cases into account, \cite{enikeevaharchaoui2014} proposed a scan test statistic aiming at sparser change coupled with their linear statistic in inference. \cite{cho2016change} adopted additional weighted CUSUM-type factor along coordinate to make the double-CUSUM statistic more adaptive in detection. 
The detection rate are also investigated in terms of sparsity and signal magnitude as well as change point location \cite{enikeevaharchaoui2014,liu2019minimax,xie2013sequential}.
We show that our result achieves optimal minimax rate, cf. Remark~\ref{rmk:rate_optimal}.
For multiple change point detection which is more challenging and essential in applications, we will discuss a \textit{backward detection} (BD) algorithm without introducing external statistics.
We will also discuss an extension to dependent sequence in Remark \ref{rmk:extension_time_series}.

Among the change point literature, mean change are widely explored using CUSUM statistics \cite{jirak2015,yuchen2017finite,cho2016change,chofryzlewicz2015}, least-square type statistics \cite{bai2010common,bhattacharjee2019change}, $U$-statistics \cite{wang2019inference} and some other kernel based methods \cite{padilla2019optimal,chen2019inference,arlot2019kernel}. 
In practice, when error terms are heavy-tailed, Gaussianity assumption is beyond salvation and becomes too restrictive. This concern especially highlights the potential of robust nonparametric methodology (such as  nonlinear projection) to avoid direct measure on mean or higher moments in data distributions. Note that the $U$-statistic approach, including our method in this paper, is conducting ``global" characterization (either one-sample or two-sample) via kernels to have change point signals peak. Such kernel concept is different from kernel density estimator or kernel distance measure for individual observations. Specifically, \cite{padilla2019optimal} proposed CUSUM variant statistic based on kernel transferred data points; \cite{chen2019inference} smoothed left and right mean function using kernel density estimation; \cite{arlot2019kernel} applied kernel least-squares criterion to quantify segmentation candidate and estimate change point locations. Compared to aforementioned papers, our $U$-statistic approach starts from a pure testing point-of-view that does not rely on any tuning of bandwidth or threshold.

The rest of this paper proceeds as follows. The bootstrap calibration for the distribution of $\overline{T}_{n}$ is described in Section~\ref{sec:bootstrap}. Main results for size validity and power properties of the bootstrap test are derived in Section~\ref{sec:main_results}. Extensions to multiple change point scenario are elaborated in Section~\ref{sec:multiple_extension}. We report simulation study results in Section~\ref{sec:simulation} and real data examples in Section~\ref{sec:real_data}. All proofs with auxiliary lemmas are given in Appendix.

\subsection{Notation}
For $q > 0$ and a generic vector $x = (x_{1},\dots,x_{p})^{T} \in \R^p$, we denote $|x|_q = (\sum_{i=1}^p |x_i|^q)^{1/q}$ for the $\ell^q$-norm of $x$ and we write $|x| = |x|_2$. For a random variable $X$, denote $\|X\|_q = (\E|X|^q)^{1/q}$. For $\beta > 0$, let $\psi_\beta(x) = \exp(x^\beta) - 1$ be a function defined on $[0,\infty)$ and $L_{\psi_\beta}$ be the collection of all real-valued random variables $X$ such that $\E[\psi_\beta(|X| / C)] < \infty$ for some $C > 0$. For $X \in L_{\psi_\beta}$, define $\|X\|_{\psi_\beta} = \inf\{C > 0 : \E[\psi_\beta(|X| / C)]  \le 1 \}$. Then, for $\beta \in [1,\infty)$, $\|\cdot\|_{\psi_\beta}$ is an Orlicz norm and $(L_{\psi_\beta}, \|\cdot\|_{\psi_\beta})$ is a Banach space \cite{ledouxtalagrand1991}. For $\beta \in (0, 1)$, $\|\cdot\|_{\psi_\beta}$ is a quasi-norm, i.e., there exists a constant $C(\beta) > 0$ such that $\|X+Y\|_{\psi_\beta} \le C(\beta) (\|X\|_{\psi_\beta} + \|Y\|_{\psi_\beta})$ holds for all $X, Y \in L_{\psi_\beta}$ \cite{adamczak2008}. Let $\rho(X, Y) = \sup_{t \in \R} |\Prob(X \le t) - \Prob(Y \le t)|$ be the Kolmogorov distance between two random variables $X$ and $Y$. We shall use $C_1,C_2,\dots$ and $K_1,K_2,\dots$ to denote positive and finite constants that may have different values. The symbol $\gtrsim$ (or $\asymp, \lesssim$) denotes greater than (or equal to, smaller than) some rates with constants omitted and $\vee$ (or $\wedge$) means the maximum (or minimum) of terms. 

	Throughout the paper, we assume $n \ge 3$ and $d \ge 3$ (i.e., $\log n \ge 1$ and $\log d \ge 1$) to simplify some statements and all inference works for $d=1,2$.

\section{Bootstrap calibration}
\label{sec:bootstrap}
To approximate the distribution of $\overline{T}_n$, we propose the following bootstrap procedure. Let $e_1, \dots, e_n$ be i.i.d.\ $N(0,1)$ random variables that are independent of $X_1^n$. Define the bootstrapped $U$-statistic and test statistic as
\begin{equation}
\label{eqn:Tn_sharp}
T_n^\sharp = {n}^{1/2} {n \choose 2}^{-1} \sum_{i=1}^n \left\{ \sum_{j=i+1}^{n} h(X_i,X_j) \right\} e_i \quad \text{ and } \quad
\overline{T}_n^\sharp :=  |T_n^\sharp|_\infty = \max_{1 \le k \le d} |T_{nk}^\sharp|.
\end{equation}
We reject $H_0$ if $\overline{T}_n > q_{\overline{T}_n^\sharp \mid X_{1}^{n}} (1-\alpha)$, where 
\begin{equation*}
q_{\overline{T}_n^\sharp \mid X_{1}^{n}} (1-\alpha) = \inf \left\{t \in \R: \Prob (\overline{T}_n^\sharp \le t \mid X_1^n) \ge 1- \alpha \right\}
\end{equation*}
is the $(1-\alpha)$ quantile of the conditional distribution of $\overline{T}_n^\sharp$ given $X_1^n$. Before presenting the rigorous validity of our bootstrap test procedure in terms of the size and power in Section~\ref{sec:main_results}, we shall explain the reason why it can (asymptotically) separate $H_{0}$ against $H_{1}$. 

First, suppose $H_{0}$ is true, i.e., $X_{1},\dots,X_{n}$ are i.i.d.\ with distribution $F$. Let $g(x) = \E [ h(x, X_1)]$ and $f(x_1, x_2) = h(x_1, x_2) - g(x_1) + g(x_2)$. Due to the anti-symmetry of $h$, we have $f(x_{1},x_{2})= -f(x_2, x_1)$. Then the Hoeffding decomposition of $T_{n}$ is 
\begin{equation}
\label{eqn:hoeffding_decomp_one-sample}
T_{n} = \underbrace{{n}^{-1/2} \sum_{i=1}^{n} {2(n-2i+1) \over n-1} g(X_{i})}_{L_{n}} + \underbrace{{n}^{1/2} {n \choose 2}^{-1} \sum_{1 \le i < j \le n} f(X_{i}, X_{j})}_{R_{n}}. 
\end{equation}
Since $f$ is degenerate, the linear part $L_n$ is expected to be a leading term of $T_n$, and the distribution of $L_{n}$ (denote as ${\cal L}(L_{n})$) can be approximated by its Gaussian analog via matching the first and second moments \cite{cck2016a,chen2018gaussian}. Since $\E[L_{n}]=0$ and  
\[
\Cov(L_{n}) = {4 (n+1) \over 3 (n-1)} \Gamma \approx {4 \over 3} \Gamma \quad \mbox{with} \quad \Gamma = \Cov(g(X_{1})), 
\]
we expect that ${\cal L}(L_{n}) \approx {\cal L}(Z)$, where $Z \sim N(0, 4\Gamma/3)$, for a large sample size $n$. Once the Gaussian approximation result for $T_{n}$ by $Z$ is established, the rest of the work is to compare the distribution of $Z$ and the conditional distribution of $T_n^\sharp$ given $X_{1}^{n}$, both of which are mean-zero Gaussians. Since 
$
\Cov(T_n^\sharp \mid X_{1}^{n}) = {4 \over n (n-1)^2} \sum_{i=1}^{n} \sum_{j=i+1}^{n} \sum_{k=i+1}^{n} h(X_{i}, X_{j}) h(X_{i}, X_{k})^{T}, 
$
standard concentration inequalities for (one-sample) $U$-statistics in \cite{chen2018gaussian} yield that $\Cov(T_n^\sharp \mid X_{1}^{n}) \approx 4 \Gamma / 3$. Thus we expect that ${\cal L}(T_n^\sharp \mid X_{1}^{n}) \approx {\cal L}(Z) \approx {\cal L}(T_{n})$, from which the size validity of the bootstrapped change point test based on $\overline{T}_{n}^{\sharp}$ follows. 

Next, we suppose $H_{1}$ is true, i.e., $X_{1},\dots,X_{m}$ are i.i.d.\ with distribution $F$ and $Y_{1},\dots,Y_{n-m}$ are i.i.d.\ with distribution $G$ such that $G(x) = F(x-\theta)$ and $Y_{i}=X_{i+m}, i=1,\dots,n-m$. To study the power property, the main idea is to consider the two-sample Hoeffding decomposition of $T_{n}$ that is similar to \eqref{eqn:hoeffding_decomp_one-sample}. Suppose $h(x+c,y+c) = h(x,y)$ is \textit{shift-invariant} in terms of location parameter. Let $\theta_{h}=\E[h(X_{1},Y_{1})]$, 
\begin{equation*}
Gh(x) = \E [ h(x,Y_1) ] - \theta_h = g(x-\theta) -\theta_h, \quad Fh(y) = \E [ h(X_1,y) ] - \theta_h = -g(y) - \theta_h,
\end{equation*}
such that $\E[Gh(X_1)] = \E [Fh(Y_1)] = 0$.  Define 
\begin{equation*}
\breve{f}(x,y) = h(x,y) - Gh(x) - Fh(y) - \theta_h,
\end{equation*} 
which is degenerate such that $\E [\breve{f}(X_1,Y_1)] = \E [\breve{f}(X_1,y)] = \E [\breve{f}(x,Y_1)] = 0$.   
Under $H_1$, we may split the $U$-statistic sum as
\[
\sum_{1 \le i<j \le n} \! h(X_i, X_j) = \sum_{\substack{1 \le i<j \le m \\ m+1 \le i<j \le n}} \! h(X_i, X_j) +   \sum_{\substack{1 \le i \le m \\ 1 \le j \le n-m}} \! h(X_i, Y_j), 
\]
where the first sum on the r.h.s.\ of the above equation has mean zero (again, due to the anti-symmetry of $h$). Thus, to study the power of $\overline{T}_{n}$ (and its bootstrapped version $\overline{T}_{n}^{\sharp}$), it suffices to analyze the second sum on the r.h.s.\ of the last display above, which is a two-sample $U$-statistic $V_{n}$ that admits the following Hoeffding decomposition: 
{\small 
	\begin{align} 
	V_{n} &= \sum_{i=1}^{m} \sum_{j=1}^{n-m} h(X_i, Y_j) \nonumber\\
	\label{eqn:hoeffding_decomp_two-sample}
	&= m (n-m) \theta_h + (n-m) \sum_{i=1}^{m} Gh(X_i) + m \sum_{j=1}^{n-m} Fh(Y_j) +  \sum_{i=1}^{m} \sum_{j=1}^{n-m} \breve{f}(X_i, Y_j).
	\end{align}}
Since the last three sums on the r.h.s.\ of (\ref{eqn:hoeffding_decomp_two-sample}) have mean zero, the power of the proposed test is determined by the magnitude of $\theta_h$ and the sampling distributions of other terms involving no $\theta_h$. For the latter, all of those distributions can be well estimated and controlled as in $H_{0}$ since they do not contain the change point signal. Thus, if $|\theta_{h}|_{\infty}$ obeys a minimal signal size requirement, then the power of $\overline{T}_{n}^{\sharp}$ would tend to one. 

\begin{rmk}
	It is interesting to note that our bootstrapped $U$-statistic $T_{n}^{\sharp}$ in~\eqref{eqn:Tn_sharp} is closely related to the jackknife multiplier bootstrap (JMB) proposed in~\cite{chen2018gaussian} for high-dimensional $U$-statistics and in~\cite{chenkato2017a} for infinite-dimensional $U$-processes with symmetric kernels. In both settings, the (unobserved) H\'ajek projection process $g(\cdot)$ is estimated by the jackknife procedure and a multiplier bootstrap is applied to the jackknife estimated process. In our change point detection context, since the kernel is anti-symmetric, averaging the empirical H\'ajek process by jackknife would simply be an estimate of zero. Thus, we may only use half (e.g., a triangular array index subset $i < j$) of the JMB to estimate $g(\cdot)$. In view of this connection, we call our bootstrap method is a JMB tailored to change point detection. 
	\qed
\end{rmk}

\section{Theoretical properties}
\label{sec:main_results}
Let $X, X'$ be i.i.d.\ random vectors with distribution $F$. Recall that $g(x) = \E[h(x, X)]$ and $f(x_1,x_2) = h(x_1,x_2) - g(x_1) + g(x_2)$ in the Hoeffding decomposition \eqref{eqn:hoeffding_decomp_one-sample}. Then $\E[g(X)] = 0$ and $\E[f(x_1,X')] = \E[f(X,x_2)] = 0$ for all $x_{1}, x_{2} \in \R^{p}$ (i.e., $f$ is degenerate). Denote $\Gamma = \Cov(g(X)) = \E [g(X)^T g(X)]$. 
In this section, we will characterize theoretical properties through $d$ (the dimension of $h$) and $\theta_h$ (the expected mean change of $h(X, X+\theta)$) rather than $p$ (the original dimension of data) or $\theta$ (the original location shift parameter) since the whole procedure is constructed on top of $h(X,X')$.

\subsection{Size validity}
We first establish the validity of the bootstrap approximation to the distribution of $\overline{T}_n$ under $H_0$.  
Let $\ub > 0$ be a constant and $D_n \ge 1$ which is allowed to increase with $n$. We make the following assumptions. 
\begin{enumerate}[leftmargin=2cm]
	\item[(A1)] $\E g_j(X)^2 \ge \underline{b}^{2}$ for all $j = 1,\dots,d$.
	\item[(A2)] $\E |h_{j}(X, X')|^{2+k} \le D_n^k$ for all $j = 1,\dots,d$ and $k=1,2$.
	\item[(A3)] $\|h_j(X, X')\|_{\psi_1} \le D_n$ for all $j = 1,\dots,d$.
\end{enumerate}
Condition (A1) is a non-degeneracy requirement for the kernel $h$. Without (A1), bootstrap may approximate constant observation through a random process so that our method is not valid.  Conditions (A2) and (A3) impose moment conditions on the kernel $h$ coupled with the data distribution $F$.   For instance, when the kernel is bounded, we do not explicitly impose additional assumption on the data distribution $F$. Thus conditions (A2) and (A3) are more robust than the canonical linear kernel when the data distribution has polynomial tails. In our high-dimensional setting, we allow both $p$ and $d$ to increase with $n$.


\begin{thm}[Size validity of bootstrap test under $H_{0}$]
	\label{thm:gaussian_approx_rate}
	Suppose $H_0$ is true and (A1)-(A3) hold. Let $\gamma \in (0, e^{-1})$ such that $\log(1/\gamma) \le K \log(nd)$ for some constant $K>0$. Then there exists a constant $C := C(\ub, K)$ depending only on $\ub$ and $K$ such that 
	\begin{equation}
	\label{eqn:gaussian_approx_rate}
	\rho(\overline{T}_{n}, \overline{T}_{n}^{\sharp} \mid X_{1}^{n}) := \sup_{t \in \R} \left| \Prob(\overline{T}_n \le t) - \Prob(\overline{T}_n^\sharp \le t \mid X_{1}^{n}) \right| \le C  \varpi_n
	\end{equation}
	holds with probability at least $1-\gamma$, where 
	\begin{equation}
	\label{eqn:gaussian_approx_rate_2}
	\varpi_n = \left\{ { D_n^2 \log^7(n d) \over n } \right\}^{1/6}.
	\end{equation}
	Consequently, we have 
	\begin{equation}
	\label{eqn:error_in_size}
	\sup_{\alpha \in (0,1)} \left| \Prob(\overline{T}_{n} \le q_{\overline{T}_n^\sharp \mid X_{1}^{n}} (\alpha)) - \alpha \right| \le C \varpi_{n} + \gamma. 
	\end{equation}
	In particular, if $\log{d} \! = \! o(n^{1/7})$, then $\Prob(\overline{T}_{n} \! \le \! q_{\overline{T}_n^\sharp \! \mid X_{1}^{n}} \! (\alpha)) \! \to \! \alpha \text{ uniformly in } \alpha \! \in \! (\! 0,\! 1\! ) \! \text{ as } n \! \to\! \infty$. 
\end{thm}

Theorem~\ref{thm:gaussian_approx_rate} constructs non-asymptotic bootstrap validity in theory and guarantees that the $\alpha$-th quantile of bootstrapped statistic $\overline{T}_n^\sharp|X_1^n$ is always close to the $\alpha$-th quantile of test statistic $\overline{T}_{n}$. Moreover, the error bound is uniform over $\alpha \in (0,1)$. The technique for proving Theorem~\ref{thm:gaussian_approx_rate} extends the Gaussian approximation theory for $U$-statistics in \cite{chen2018gaussian}, which focuses on symmetric kernels.

\begin{rmk}[Comparisons with the CUSUM-based statistics]
	\label{rmk:no_boundary_remove_H0}
	\cite{jirak2015} and \cite{yuchen2017finite} propose CUSUM-based bootstrap tests that require the removal of boundary points for detecting change points in high-dimensional mean vectors. Specifically, for the CUSUM statistics \eqref{eqn:cusum_mean_Rp} considered in \cite{yuchen2017finite}, the test statistic is of the form $S_{n} = \max_{\underline{s} \le s \le n-\underline{s}} |Z_{n}(s)|_{\infty}$ for some boundary removal parameter $\underline{s} \in [1,n/2]$. Accordingly, the Gaussian multiplier bootstrap version of $Z_{n}(s)$ is defined as: 
	\[
	Z_{n}^{\sharp}(s) = \left({n-s \over ns}\right)^{1/2} \sum_{i=1}^{s} e_{i} (X_{i}-\overline{X}_{s}^{-}) - \left({s \over n(n-s)}\right)^{1/2} \sum_{i=s+1}^{n} e_{i} (X_{i} - \overline{X}_{s}^{+}),
	\]
	where $\overline{X}_{s}^{-} = s^{-1} \sum_{i=1}^{s} X_{i}$ and $\overline{X}_{s}^{+} = (n-s)^{-1} \sum_{i=s+1}^{n} X_{i}$ are the left and right sample averages at $s$, respectively. 
	$Z_{n}^{\sharp}(s)$ sequentially inspects the two-sample distributions before and after all possible change point locations in the interval $[\underline{s}, n-\underline{s}]$.
	Then for the special case of linear kernel $h(x,y)=x-y$ and distribution $F$ satisfying the conditions (A1), (A2), and (A3), the rate of convergence for $\overline{S}_{n}^{\sharp} := \max_{\underline{s} \le s \le n-\underline{s}} |Z_{n}^{\sharp}(s)|_{\infty}$ shown in \cite{yuchen2017finite} obeys 
	\[
	\rho(\overline{S}_{n}, \overline{S}_{n}^{\sharp} \mid X_{1}^{n}) \le C \left\{ { D_n^2 \log^7(n d) \over \underline{s} } \right\}^{1/6}
	\]
	with probability at least $1-\gamma$.
	Comparing the last display with the rate of convergence for $\rho(\overline{T}_{n}, \overline{T}_{n}^{\sharp} \mid X_{1}^{n})$ in~\eqref{eqn:gaussian_approx_rate} and~\eqref{eqn:gaussian_approx_rate_2}, we see that the JMB method proposed here has better statistical properties than the Gaussian multiplier bootstrap $\overline{T}_{n}^{\sharp}$ without removing any boundary points in computing $\overline{T}_{n}$ and $\overline{T}_{n}^{\sharp}$. Consequently this will reduce the error-in-size~\eqref{eqn:error_in_size} for our bootstrap calibration $\overline{T}_{n}^{\sharp}$. Empirical evidence for our algorithm with smaller error-in-size can be found in Section~\ref{sec:simulation}. The main reason for the improved rate is due to the fact that we can approximate the distribution of $\overline{T}_{n}$ based on the majority of the data points in the entire sample $X_{1},\dots,X_{n}$.  In addition, the proposed change point detector $\overline{T}_{n}$ and its JMB calibration $\overline{T}_{n}^{\sharp}$ can be viewed as a {\it nonlinear} and {\it one-pass} version of the CUSUM statistics. 
	\qed
\end{rmk}

	\begin{rmk}[Improved size validity of the bootstrap test] \label{rmk:improved_size}
		Proof of Theorem~\ref{thm:gaussian_approx_rate} is based on the Gaussian and bootstrap results for linear partial sums in high dimensions \cite{cck2016a} and the maximal inequality for degenerate $U$-statistics \cite{chenkato2017a}. Since the work of \cite{cck2016a}, there have been substantial progresses being made to improve the rate of convergence of Gaussian approximation for partial sums under various settings. For instance, \cite{cck2020} derived nearly optimal bound for the Gaussian approximation over hyper-rectangles. Tailored to our change point detection setting, if the correlation matrix of $L_n$ is {\it strongly} non-degenerate (i.e., the smallest eigenvalue of the correlation matrix of $L_n$ is strictly positive), then the rate of Gaussian approximation to $L_n$ can be sharpened to $n^{-1/2} (\log n) (\log d)^{3/2}$. Combining this with the maximal inequality for $R_n$, we can improve the overall bound for $\rho(\overline{T}_{n}, \overline{T}_{n}^{\sharp} \mid X_{1}^{n})$ to $n^{-1/4} (\log (nd))^{1/2} (\log n) (\log d)^{1/2}$.	 
		
		Let $\sigma_*$ be the square root of the smallest eigenvalue of the correlation matrix of $g(X)$. We assume that
		\begin{enumerate}[leftmargin=2cm,align=left]
			\item[(A2')] $\E |h_{j}(X, X')|^{4} \le D_n^2$ for all $j = 1,\dots,d$.
			\item[(A3')] $\|h_j(X, X')\|_{\psi_2} \le D_n$ for all $j = 1,\dots,d$.
		\end{enumerate}
		
		\begin{thm}[Improved size validity of the bootstrap test under $H_{0}$]
			\label{thm:gaussian_approx_rate_improve}
			Suppose $H_0$ is true, $\sigma_*^2>0$, and (A1), (A2') and (A3') hold. Let $\gamma \in (0, e^{-1})$ such that $\log(1/\gamma) \le K \log(nd)$ for some constant $K>0$. Then there exists a constant $C := C(\ub, \sigma_*, K)$ depending only on $\sigma_*, \ub$ and $K$ such that 
			\begin{equation}
			\label{eqn:gaussian_approx_rate_improve}
			\rho(\overline{T}_{n}, \overline{T}_{n}^{\sharp} \mid X_{1}^{n}) \le C  \varpi'_n
			\end{equation}
			holds with probability at least $1-\gamma$, where 
			\begin{equation}
			\label{eqn:gaussian_approx_rate_2_improve}
			\varpi'_n =  {D_n (\log (nd))^{1/2} (\log n) (\log d)^{1/2} \over n^{1/4}}.
			\end{equation}
		\end{thm}
		\qed
	\end{rmk}

\subsection{Power analysis}
Next, we analyze the power of the proposed testing under $H_1$ in terms of the change point signal $\theta_h = \E[h(X,X'+\theta)]$ and its location $m$.	
In our $U$-statistic framework, the test implicitly depends on $\theta$ through $\theta_h$, which the signal strength characterization will relate to. As we have discussed earlier, the signal magnitudes between $\theta$ and $\theta_{h}$ can be preserved for the robust sign kernel. Under $H_{1}$, we assume the following conditions. 
\begin{itemize}[leftmargin=2cm]
	\item[(B1)] $h$ is \emph{shift-invariant}: $h(x+c, y+c) = h(x,y)$.
	\item[(B2)]  $\E |h_{j}(X, X'+\theta) - \E[h_{j}(X, X'+\theta)] |^{2+\ell} \le D_n^\ell$ for all $j = 1, \cdots, d$ and $\ell =1,2$.
	\item[(B3)]  $|| h_{j}(X, X'+\theta) - \E[h_{j}(X, X'+\theta)] ||_{\psi_{1}} \le D_{n}$ for all $j = 1, \cdots, d$. 
\end{itemize}
Condition (B1) is a natural requirement since the within-sample noise cancellation by $h$ should be invariant under data translation in the location-shift model (\ref{eqn:location_shift_model}). Conditions (B2) and (B3) are in parallel with Condition (A2) and (A3) in the sense that they quantify the moment and tail behaviors of the centered version of the kernel $h$ (w.r.t.\ the distribution $F$). In particular, Conditions (B2) and (B3) separate the location-shift signal from the mean-zero noise, and if $\theta = 0$, Conditions (B2) and (B3) reduce to Conditions (A2) and (A3). Our next theorem characterizes the minimal signal strength for detecting the change point under the alternative hypothesis $H_{1}$. 

\begin{thm}[Power of bootstrap test under $H_1$]
	\label{thm:power_signal_rate}
	Suppose $H_1$ is true and (B1)-(B3) hold in addition to (A1)-(A3). Let $\zeta \in (0, e^{-1})$ such that $\log(1/\zeta) \le K \log(nd)$ for some constant $K>0$. Suppose $m \wedge (n-m) \ge K' \log^{5/2} (nd) $ for some large enough $K'>0$. If 
	\begin{equation}
	\label{eqn:signal_rate}
	m(n-m) |\theta_h|_\infty > K_0 D_n n^{3/2} \log^{1/2} ({nd / \alpha}) + C_1(\ub) n^{3/2} \log^{1/2} (\zeta^{-1}) \log^{1/2} (d),
	\end{equation}
	for some constants $K_0$ and $C_1(\ub)$, then 
	$
	\Prob ( \overline{T}_n > q_{\overline{T}_n^\sharp \mid X_{1}^{n}} (1-\alpha) ) \ge 1- \zeta - C_2(\ub) \varpi_n.
	$
\end{thm}

Theorem~\ref{thm:power_signal_rate} provides the lower bound of signal strength that is related to change point location $m$ and size level $\alpha$, as well as sample size $n$ and kernel dimension $d$.
Markedly, our theory derives the tail probability control on the maximum of two-sample order-two $U$-statistics.

\begin{rmk}[Interpretation of Theorem~\ref{thm:power_signal_rate}]
	\label{rmk:power_boundary}
	Note the first term on the r.h.s.\ of (\ref{eqn:signal_rate}) reflects the Type I error of the bootstrap test (coming from $\alpha$ and $\varpi_n$ in Theorem~\ref{thm:gaussian_approx_rate}), while the second term reflects the connection to the Type II error under $H_{1}$ through $\zeta$. If the location shift happens in the middle, i.e., $m \asymp n$, then $m(n-m) \asymp n^2$. In this case, the signal strength has to obey $|\theta_h|_\infty \gtrsim D_n n^{-1/2} \log^{1/2}(nd / \alpha)$, which matches the power result for the bootstrap test based on the CUSUM statistics in \cite{yuchen2017finite} (cf. Theorem 3.3 therein). If the location shift occurs at the boundary, for instance $m \wedge (n-m) \asymp n^{\beta}$ for $\beta < 1/2$, then the signal has to be $|\theta_h|_\infty \gtrsim n^{1/2-\beta}$, which diverges to infinity. Thus, under our framework, detection is possible for local alternative when the change point location satisfies $m \wedge (n-m) \gtrsim D_n n^{1/2} \log^{1/2} (nd)$.
	\qed
\end{rmk}

\begin{rmk}[Rate optimality for sparse alternative]
	\label{rmk:rate_optimal}
	In \cite[Theorem 1]{liu2019minimax}, the authors derived the minimax rate of detection boundary for single change point case where $F$ is $p$-dimensional Gaussian distribution with independent entries. Suppose the location shift only occurs in the first $k$ components with the same size of $\rho > 0$, i.e.\
	\[
	\theta = (\underbrace{\rho, \dots, \rho}_{k \; \text{times}}, 0, \dots, 0)^\top.
	\] 
	For sparse regime when $ { k} =|\theta|_0 < \sqrt{p \log \log (8n)}$, let $|\theta_h|_2^2 \approx |\theta|_2^2  =  k \rho^2 $ under local alternative, then their minimax result reads as 
	{
		\[
		{m(n-m) \over n} k \rho^2 \gtrsim \rho^*(p,n,k) \asymp  \left(k \log \{{ep\log \log (8n) \over k^2}\} \vee \log \log (8n) \right).
		\]
		Note that, $m(n-m) = (m \wedge (n-m))((m \vee (n-m)) \asymp (m \wedge (n-m)) n$. Hence, their result indicates that $\rho \gtrsim (m \wedge (n-m))^{-1/2} \sqrt{\log \{{ep\log \log (8n) \over k^2}\} \vee {1 \over k}\log \log (8n) }$. The rate inside square root is up to a logarithm factor through $n,p$ (for example by plugging in $k=1$).
		On the other hand, our (\ref{eqn:signal_rate}) in Theorem~\ref{thm:power_signal_rate} requires the lower bound $\rho  \gtrsim (m \wedge (n-m))^{-1} n^{1/2}$ up to $\log^{1/2} (nd)$.
		If $m \wedge (n-m)$ is bounded away from boundaries, i.e., $m \asymp n-m \asymp n$, then our result is minimax optimal. 
	}
	\qed
\end{rmk}

{
	\begin{rmk}[Extension of the bootstrap test to time series data] \label{rmk:extension_time_series}
		When the noise sequence $\xi_i$ in the location-shift model~\eqref{eqn:location_shift_model} is a stationary time series, we need to modify the bootstrap test statistic to adjust for the temporal dependency because $\E h(X_i, X_j)$ is no longer zero and there is a bias term to be calibrated in the bootstrap test. Nonetheless, if the time series $\xi_i$ is weakly dependent, then the bias term decays to zero when $|i-j|$ increases. This motivates us to consider a {\it trimmed} version of the bootstrap test by removing summands within close indices in $T_n$ (and thus $T_n^\sharp$). Let the integer $0 \le M < m \wedge (n-m)$ be a trimming parameter. We define a generalized $U$-statistic as 
		\begin{equation}
		\label{eqn:Tn_scaled_TS}
		T_n^\natural =  {n}^{1/2}{n \choose 2}^{-1} \sum_{\substack{i<j\\|i-j| > M}} h(X_i, X_j) = {2 \over {n}^{1/2}{(n-1)}} \sum_{i=1}^{n-M-1}\sum_{j=i+M+1}^{n} h(X_i, X_j).
		\end{equation}
		Under $H_0$, we expect $h(X_i, X_j)$ behaves similarly to the i.i.d.\ scenario for large $M$ since the dependency between $X_i$ and $X_j$ is weak. Thus, we have $\E h(X_i, X_j) \approx 0$ for $|i-j|>M$ and $\E T_n^\natural \approx 0$. Under $H_1$, with $\E h(X_i, X_j) \approx \theta_h$ for $i \leq m < j$ and $|i-j|>M$, we have
		\begin{align}
		\label{eqn:Tn_natural_mean}
		\E T_n^\natural &\approx {n}^{1/2}{n \choose 2}^{-1} \left[ \sum_{i=1}^{m-M} \sum_{j=m+1}^{n} + \sum_{i=n-M+1}^{m} \sum_{j=i+M+1}^{n} \right]  \E h(X_i, X_j) \nonumber\\
		& \approx 2 n^{-3/2} \left[ m(n-m) - (M+1)M/2 \right] \theta_h.
		\end{align}
		There is a natural trade-off in choosing the trimming parameter $M$ to control the effective signal strength $\E T_n^\natural$ under $H_0$ and $H_1$. For larger values of $M$, calibration of the distribution of $T_n^\natural$ would be more accurate. However, the compromise of signal strength in~\eqref{eqn:Tn_natural_mean} would also be larger. Thus, it would be harder to detect change point (i.e., to separate $H_0$ from $H_1$) when the temporal dependence of data is stronger.
		Similarly as the i.i.d.\ noise case, we can use the $\ell^\infty$-norm to construct our test statistic 
		\begin{equation}
		\label{eqn:Tn_natural}
		\overline{T}_n^\natural := |T_n^\natural|_\infty = \max_{1\le k \le d} |T_{nk}^\natural|, 
		\end{equation} 
		which separates $H_0$ from $H_1$ when temporal dependence exists. 
		
		Let $e_1, \dots, e_{n-M+1}$ be i.i.d.\ $N(0,1)$ random variables that are independent of $X_1^n$. 
		Define the bootstrapped test statistic
		\begin{equation}
		\label{eqn:Tn_flat}
		T_n^\flat = {n}^{1/2} {n \choose 2}^{-1} \ \sum_{i=1}^{n-M-1} \left\{ \sum_{j=i+M+1}^{n} h(X_i,X_j) \right\} e_i
		\end{equation}
		and $\overline{T}_n^\flat :=  |T_n^\flat|_\infty = \max_{1 \le k \le d} |T_{nk}^\flat|$. We reject $H_0$ if $\overline{T}_n^\natural > q_{\overline{T}_n^\flat \mid X_{1}^{n}} (1-\alpha)$, 
		the $(1-\alpha)$ quantile of the conditional distribution of $\overline{T}_n^\flat$ given $X_1^n$.
		
		When $\xi_i$ is an independent noise sequence, we simply set $M=0$ so that $T_n^\natural$ and $T_n^\flat$ reduce to $T_n$ and $T_n^\sharp$, respectively. In Section~\ref{subsec:sim_extension_time_series}, we shall provide 
		some empirical performance of the trimmed bootstrap test for a vector autoregressive process $\xi_{i}$.
		\qed
	\end{rmk}
}

\section{Extensions to multiple change points scenario}
\label{sec:multiple_extension}

\subsection{Direct extension to multiple change points testing}

Recall  $X_{i} \sim F_{i}, i =1,\dots,n$  as a sequence of independent random vectors taking values in $\R^{p}$. Generally, suppose there are $\nu$ change points $m_0 = 0 < m_1 < \dots < m_\nu < m_{\nu+1} = n$ such that
\begin{equation*}
F_{m_k+1}(x) = \cdots = F_{m_{k+1}}(x) = F(x - \theta^{(k)}) \text{ and } F_{m_k} \neq F_{m_{k+1}} \text{ for } k = 0, \dots, \nu. 
\end{equation*} 
Without loss of generality, we can assume $\theta^{(0)} = 0$. Consider the alternative hypothesis with multiple change points
\begin{equation}
\label{eqn:change_point_mean_test_mcp}
H_1^{'} : \theta^{(k)} \neq \theta^{(k+1)}  \text{ for some } m_k, k=0,\dots,\nu \text{ and } \nu \ge 1.
\end{equation}
Denote $X_i = \xi_i + \theta^{(k)}$ and due to the shift-invariant property (B1) we have
\begin{equation*}
\delta^{(k,k')} = \E h(X_i,X_j) = \E h(\xi_i, \xi_j + (\theta^{(k')} - \theta^{(k)} ) ) \text{ for } m_k < i \le m_{k+1}, m_{k'} < j \le m_{k'+1}.
\end{equation*}
Let $s_i = m_{i+1}-m_{i}$ be the size of data segment that corresponds to the $i$-th location shift. Then,
\begin{equation}
\label{eqn:signal_strength_overall}
\E \left[ \sum_{1 \le i<j \le n} h(X_i,X_j) \right] = \sum_{0 \le k < k' \le \nu} s_{k} s_{k'} \delta^{(k,k')} =: \tilde{\Delta},
\end{equation}
where the standardized signal strength is $|E[T_n]|_\infty = n^{1/2} {n \choose 2}^{-1} |\tilde{\Delta}|_\infty$.
Under the multiple change points alternative, if signal cancellation does not exist, i.e.\ $|\tilde{\Delta}|_\infty$ is away from 0, then we can directly extend the theory as below.

\begin{lem}[Power of the bootstrap test under $H_1^{'}$]
	\label{lem:power_signal_rate_multiple}
	Suppose $H_1^{'}$ is true and (B1)-(B3) hold in addition to (A1)-(A3). Let $\zeta \in (0, e^{-1})$ such that $\log(1/\zeta) \le K \log(\nu^2 nd)$ for some constant $K>0$. Suppose $\nu$ is a constant. If 
	\begin{equation}
	\label{eqn:signal_rate_multiple}
	|\tilde{\Delta}|_\infty  > K_0 \nu^2 D_n n^{3/2} \log^{1/2}(nd/\alpha) + C_1(\ub) n^{3/2} \log^{1/2}(\zeta^{-1}) \log^{1/2}(d) + \phi,
	\end{equation}
	where 
	\begin{align*}
	\phi = & K_0' \left\{	n^{3/4} \log^{3/4}(nd/\alpha) \max_{k<k'} (s_k s_{k'})^{1/4} |\delta^{(k,k')}|_\infty + \right. \\ &\qquad \qquad \left.   n^{1/2} \log^{1/2}(nd/\alpha) \sum_{ k < k'} (s_k s_{k'})^{1/2} |\delta^{(k,k')}|_\infty  \right\},
	\end{align*}
	then 
	$
	\Prob ( \overline{T}_n > q_{\overline{T}_n^\sharp \mid X_{1}^{n}} (1-\alpha) ) \ge 1- \zeta - C_2(\ub) \varpi_n
	$
	for some constants $K_0, K_0'$ and $C_1(\ub),  C_2(\ub)$.
\end{lem}

\begin{rmk}[Explanation on $\phi$ and connection to single change point case]
	\label{rmk:compare_rates_to_single}
	Compared to (\ref{eqn:signal_rate}) in Theorem~\ref{thm:power_signal_rate}, there is an additional term $\phi$ in (\ref{eqn:signal_rate_multiple}). 	It comes from controlling $\Cov(T_n^\sharp \mid X_{1}^{n}) $ under the alternative hypothesis.
	Consider the special case of single change point where $\nu=1$ in (\ref{eqn:change_point_mean_test_mcp}), we may assume $m = s_0 < s_1 = n-m$. Then $\phi \asymp (m^{1/4} n \log^{3/4} (nd)  + m^{1/2} n \log^{1/2} (nd) )  |\delta^{(0,1)}|_\infty \lesssim m(n-m) |\delta^{(0,1)}|_\infty = |\tilde{\Delta}|_\infty$ for $m \gtrsim \log^{5/2}(nd)$, i.e., $\phi$ is dominated by the l.h.s.\ of (\ref{eqn:signal_rate_multiple}). Then our result under $H_1^{'}$ reads the same as (\ref{eqn:signal_rate}).
	\qed
\end{rmk}

The l.h.s.\ of (\ref{eqn:signal_rate_multiple}) is the overall signal strength which does not directly depend on minimum separation of change points $\underline{m} = \min_{0\le k \le \nu} s_k$ or signal strength like $\bar{\delta} = \max_{0\le k < k' \le \nu} |\delta^{(k,k')}|_\infty$ or $\bar{\delta}^{'} = \min_{0\le k <  \nu} |\delta^{(k,k+1)}|_\infty$ that is usually assumed under CUSUM-based approach \cite{cho2016change,chofryzlewicz2015,yuchen2017finite}. Taking (\ref{eqn:cusum_mean_Rp}) for instance, our framework does not screen out any statistic by visiting each location $i = 1, \dots, n-1$. Therefore, we allow the product of $s_k s_{k'} \delta^{(k,k')}$ dominates the overall change $\tilde{\Delta}$ even if $s_k$ or $\delta^{(k,k')}$ is fairly small. However, it is inconvenient that signal cancellation in (\ref{eqn:signal_strength_overall}) cannot be characterized by $\underline{m}$ or $\bar{\delta}$. Another drawback is that $\tilde{\Delta} = 0$ can happen even if $\underline{m} \asymp O(n)$ and $\bar{\delta}$ is large.
This issue will be discussed in the next section.
Before that, we discuss two special cases derived from Lemma~\ref{lem:power_signal_rate_multiple} based on $\underline{m}$ and $\bar{\delta}$ to make the lemma more informative and instructional. Besides, we can avoid $|\delta^{(k,k')}|_\infty$ being on both sides of (\ref{eqn:signal_rate_multiple}). 
\begin{enumerate}
	\item Suppose $\bar{\delta}$ is upper bounded, for example $h$ is the bounded sign kernel. We have $s_k <n$, which leads to $\max_{0\le k < k' \le \nu} (s_k s_{k'})^{1/4} \le n^{1/2}$ and $\sum_{ k < k'}  (s_k s_{k'})^{1/2} \le \nu^2 n$. Since $n\gtrsim \log^7(nd)$, so $\phi \lesssim \nu^2 n^{3/2} \log^{1/2}(nd) \bar{\delta} $, which is nearly the same rate as the first part on the r.h.s.\ of (\ref{eqn:signal_rate_multiple}). Therefore, $\phi$ can be dropped. 
	\item Suppose $\{|\delta^{(k,k')}|_\infty: 0\le k < k' \le \nu\}$ are at the same magnitude and $|\tilde{\Delta}|_\infty$ is dominated by $s_k s_{k'} |\delta^{(k,k')}|_\infty \gtrsim \underline{m}^2 \bar{\delta}$ for some pair of $(k, k')$.
	Then a sufficient condition to control Type II error is to have $\underline{m}^2 \bar{\delta}$ greater than the upper bound of $\phi$, namely $n^{3/2} \log^{1/2}(nd) \bar{\delta}$. So we only need $\underline{m} \gtrsim n^{3/4} \log^{1/4}(nd)$. This is weaker than the condition in \cite[(B1)]{cho2016change} which requires $\underline{m} \gtrsim n^{6/7}$. One example of such assumption is the setup in \cite{jirak2015} where each dimension has at most one change.
\end{enumerate}
In summary, we have the following corollary.
\begin{cor}
	\label{cor:multiple_special_cases}
	Suppose the conditions in Lemma~\ref{lem:power_signal_rate_multiple} are satisfied. \\
	(i) If $\bar{\delta} = \max_{0\le k < k' \le \nu} |\delta^{(k,k')}|_\infty$  is bounded, then $\Prob ( \overline{T}_n > q_{\overline{T}_n^\sharp \mid X_{1}^{n}} (1-\alpha) ) \ge 1- \zeta - C_2(\ub) \varpi_n$ when
	\begin{align*}
	|\tilde{\Delta}|_\infty &= |\sum_{ k < k'} s_{k} s_{k'} \delta^{(k,k')}|_\infty \\
	&> K_0 \nu^2 D_n n^{3/2} \log^{1/2}(nd/\alpha) + C_1(\ub) n^{3/2} \log^{1/2}(\zeta^{-1}) \log^{1/2}(d).
	\end{align*}
	(ii) If all $|\delta^{(k,k')}|_\infty$ are at the same rate and $|\tilde{\Delta}|_\infty > K_1 \underline{m}^2 \bar{\delta}^{'}$, then $\phi$ in \eqref{eqn:signal_rate_multiple} can be dropped when 
	\begin{equation*}
	\underline{m} = \min_{0\le k \le \nu} s_k \ge K_2 n^{3/4} \log^{1/4}(nd/\alpha) .
	\end{equation*}
	Consequently, if signals are almost evenly spread (i.e. $\underline{m} \asymp n$) and $|\delta^{(k,k')}|_\infty$ is upper bounded, then  $\Prob ( \overline{T}_n > q_{\overline{T}_n^\sharp \mid X_{1}^{n}} (1-\alpha) ) \ge 1- \zeta - C_2(\ub) \varpi_n$ when
	\begin{equation*}
	|\sum_{ k < k'} \delta^{(k,k')}|_\infty > K_0 \nu^2 D_n n^{-1/2} \log^{1/2}(nd/\alpha) + C_1(\ub) n^{-1/2} \log^{1/2}(\zeta^{-1}) \log^{1/2}(d).
	\end{equation*}
\end{cor}
In Remark~\ref{rmk:power_boundary}, we have shown that local alternative is detectable when $\underline{m} \gtrsim n^{1/2} \log^{1/2}(nd/\alpha)$. Corollary~\ref{cor:multiple_special_cases} (ii) has a stronger requirement due to extra cost from handling the possible cancellation in analyzing the general case of multiple change points. If there is only one change point, then the interpretation of rates in Lemma~\ref{lem:power_signal_rate_multiple} can be found in Remark~\ref{rmk:compare_rates_to_single}.
A real application for our global test lies in the special case of monotone signals that have order structures $\theta_{1} \le \cdots \le \theta_{\nu}$ \cite{Minami_2020}.

\subsection{Modification to block testing} 
\label{subsec:block_testing}
The direct extension of testing $H_0$ against $H_1^{'}$ depends on $|\tilde{\Delta}|_\infty$, which can be 0 even if each $|\delta^{(k,k')}|_\infty$ are fairly large. The global test will not help under severe signal cancellation. One solution is to localize the test such that the problem can convert to single change point scenario. 

Consider performing a block testing in the following way.
Divide the sample into $B$ blocks of size $L$ ($n=BL$ for brevity) where $L \le 2 \underline{m}$. Then each block contains at most 1 change point.  We can apply the original test to the block-vector data $Z_1, \dots, Z_L \in \R^{Bp}$, where $Z_i = \vec (X_{i}, \dots,  X_{bL+i}, \dots, X_{(B-1)L+i})$. Let $h^Z: \R^{Bp} \times \R^{Bp} \rightarrow \R^{Bd}$ be  the block version extension of $h$: 
$$h^Z (Z_i,Z_j) = ( h(X_{i}, X_{j})^\top, \dots, h(X_{(B-1)L+i}, X_{(B-1)L+j})^\top)^\top.$$
Note that there is no signal cancellation issue. Denote $m^Z_k = (m_k \bmod L)$. Modified theory of power will depend on signal strength as below.

\begin{cor}
	Suppose the conditions in Lemma~\ref{lem:power_signal_rate_multiple} hold. If 
	{\small \begin{equation*}
		\max_{0\le k \le \nu} m^Z_k (L-m^Z_k ) |\delta^{(k,k')}|_\infty > K_0 \nu^2 D_n L^{3/2} \log^{1/2}({nd\over\alpha}) + C(\ub) L^{3/2} \log^{1/2}(\zeta^{-1}) \log^{1/2}(d),
		\end{equation*}}
	then 
	$
	\Prob \left( \overline{T}_n > q_{\overline{T}_n^\sharp \mid X_{1}^{n}} (1-\alpha) \right) \ge 1- \zeta - C_2(\ub) \varpi_n
	$
	for some constants $K_0$ and $C_1(\ub),  C_2(\ub)$.
	
\end{cor}
Note that the rate now depends on $L$ rather than $n$ (except for logarithm factors). The block test sacrifices sample size to gain the single change-point structure. In practice, the block parameter $L$ (or equivalently $B$) needs to be selected carefully since power depends on the relevant locations of $\{m^Z_k\}_{k=0}^\nu$. One solution is to use $L = 2 n^{1/2} \log^{1/2} (nd)$ that is discussed in Remark~\ref{rmk:power_boundary} or $L = 2 n^{3/4} \log^{1/4} (nd)$ that is from Corollary~\ref{cor:multiple_special_cases} (ii).

\subsection{Discussion on binary segmentation}
To deal with multiple change points, binary segmentation (BS) is conceptually straightforward  \cite{cho2016change,chofryzlewicz2015,yuchen2017finite}. The main idea is to recursively estimate change points by screening sub-segments before and after each estimated location. 
However, such process starts from a ``global" detection that may miss change points under unfavorable configuration of signal cancellation.
To improve BS, \cite{fryzlewicz2014} proposed wild binary segmentation (WBS) that randomly draw intervals to localize searching for change points. Recently, it has been widely adopted \cite{wangsamworth2017,wang2019inference} owing to its flexibility and computational efficiency. However, we will not be able to apply BS or WBS based approaches directly because there is no estimator in our framework so far.

One solution is to incorporate an external estimator. For example, consider the $U$-statistics $T(s) = \sum_{i=1}^{s} \sum_{j=s+1}^{n} h(X_i,X_j), s=1,\dots, n-1 $ where $h$ is the anti-symmetric kernel used in (\ref{eqn:Tn_scaled}). It can be shown that for each segment $m_k \le s-1 <s \le m_{k+1}$
\begin{equation*}
\E T(s) - \E T(s-1) = \sum_{j=m_{k+1}+1}^{n} \E h(X_s,X_j) -  \sum_{i=1}^{m_k} \E h(X_i,X_s) = const. 
\end{equation*}
In other word, within each segment $(m_{k}, m_{k+1}]$, $\E T_l(s)$ is monotone ($l=1, \dots, p$). So $\max_{1 \le s \le n-1}|\E T(s)|_\infty$ is always attained at one change point. Therefore, the estimator 
$$\hat{m} = \argmax_{1 \le s \le n-1} |T(s)|_\infty$$
can play a role in BS type approach. Similar ideas are discussed in \cite{Pettitt1979,gombay1995application,gombay2001u, brault2018nonparametric} as applications using $U$-statistics for estimation of change points. 
Though it is fascinating to investigate the consistency of a BS algorithm that combines estimation using $\hat{m}$ and our bootstrapping test using $T_n$, the focus and main contribution of this paper is to perform a test without visiting each point. So we leave this algorithm as an open question for future analysis.

{
	Another solution is to adopt the randomization idea from WBS to conduct inference in the presence of multiple change points. One can independently sample $B_{W}$ intervals that are wider than a pre-specified length $n'$ and obtain a set of (scaled) test statistics on each interval. Denote the set as $\calT_W (X_{1}^{n})$. For a given level $\alpha$, we then perform the proposed bootstrap test on the interval whose corresponding (scaled) test statistic achieves the $(1-\alpha)$-th quantile of $\calT_W (X_{1}^{n})$. If the bootstrap test rejects $H_0$ under level $\alpha$, then it implies a change point in this interval, which in turn concludes $H_1^{'}$ of at least one change point.
	The WBS-type test is summarized in Algorithm~\ref{alg:WBS_Ustat}. 
	
	\begin{algorithm}[h]
		\caption{WBS-type testing ($\alpha, n'$) against multiple change points}\label{alg:WBS_Ustat}
		\begin{algorithmic}[1]    		
			\STATE{Draw $B_{W}$ random intervals $[s_b,e_b], b=1,\dots, B_{W}$, where start- and end-points are taken independently and uniformly from $\{1,\dots, n\}$ such that $e_b - s_b > n'$.}
			\STATE{Denote $\calT_W (X_{1}^{n}) = \{ \max_{1 \le k \le d} |T_{e_b - s_b}(X_{s_b}^{e_b})|_k, b=1,\dots, B_{W}\}$, where $$T_{e_b - s_b} (X_{s_b}^{e_b}) = {(e_b - s_b)}^{1/2}{e_b - s_b \choose 2}^{-1} \sum_{s_b \le i<j \le e_b} h(X_i, X_j) $$ is our $U$-statistic on each interval.}    		
			\STATE{Let $q_{\overline{T}_{W} \mid X_{1}^{n}} (1-\alpha)$ be the $(1-\alpha)$-th quantile of $\calT_W$ and $b'$ be the corresponding index.}
			\STATE{Perform our bootstrap test on $[s_{b'},e_{b'}]$.}
			\IF {our bootstrap test is significant at level $\alpha$}
			\STATE {reject $H_0$.}
			\ELSE
			\STATE {reject $H_1^{'}$.}
			\ENDIF
		\end{algorithmic}
	\end{algorithm}

	Note that the tuning parameter of $n'$ bounds the length of randomly selected intervals from below. If $n'$ is too small, for instance $n'=1$, then Step 4 is likely to end up with a very small interval $[e_{b'},s_{b'}]$. Since approximating $\Cov_{jj} (T_{e_b - s_b})$ on small intervals $\{[s_b,e_b]\}$ will not be consistent, it can lead to the failure of size control under $H_0$. In practice, one may select $n'$ by applying the Algorithm~\ref{alg:WBS_Ustat} on $\{\epsilon_i X_i\}_{i=1}^n$, where the multipliers $\epsilon_i, i=1,\dots,n$ are i.i.d.\ standard Gaussian random variables that are independent of $X_1^n$. Since $\E (\epsilon_iX_i \mid X_1^n) = 0$ and $\Cov (\epsilon_iX_i \mid X_1^n) = X_i X_i^T$, the transformed data $\{\epsilon_i X_i\}_{i=1}^n$ can mimic $H_0$ without any structural assumption. Simulation result for Algorithm~\ref{alg:WBS_Ustat} is presented in Section~\ref{subsec:app_additional_sim_results} in the Appendix.
}

\subsection{Backward detection approach for change points estimation}
As shown in aforementioned forward searching solutions, the drawbacks of BS include cancellation of signals and requirement of change point estimators. 
Instead of repeatedly splitting intervals after each detection of change point, we can reversely merge consecutive segments in a backward detection way \cite[Section 3.2.2]{niu2016multiple}. Then, our test can work as a stopping rule. 

Precisely, denote the initial partition of data segments as $b_0^{(0)} = 0 < b_1^{(0)} < b_2^{(0)} < \cdots < b^{(0)}_{\nu_0-1} < n = b_{\nu_0}^{(0)}$ and the corresponding data blocks as $\calB^{(0)} = \{ B_1^{(0)}, B_2^{(0)}, \dots, B_{\nu_0}^{(0)} \}$, where $B_i^{(0)} = \{ X_{b_{i-1}^{(0)}+1}, \dots, X_{b_{i}^{(0)}}\}$. 
For each pair of consecutive blocks $\{ B_i^{(0)}, B_{i+1}^{(0)} \}, i = 1,\dots,\nu_k-1$, we can compute a \textit{Dissimilarity Index} based on $T_n$ using truncated data sequence, i.e.\ 
\begin{align}
\label{eqn:dissimilarity_index}
DI_i  &= |T_n( B_i^{(0)} \cup B_{i+1}^{(0)})|_\infty \nonumber \\
&= \max_{1 \le k \le d} \left|(b_{i+1}^{(0)} - b_{i-1}^{(0)})^{1/2} {b_{i+1}^{(0)} - b_{i-1}^{(0)} \choose 2}^{-1} \quad \sum_{\mathclap {b_{i-1}^{(0)}+1 \le i < j \le b_{i+1}^{(0)}} } \quad h_k (X_i,X_j)\right|.
\end{align}
Since each component of $T_n$ is the standardized Hodges-Lehmann type estimator of location shift in each dimension, large $DI_i$ indicates strong dissimilarity between $B_i^{(0)}$ and $B_{i+1}^{(0)}$. Therefore, we can pick the pair of data blocks with the smallest $DI$ and perform our bootstrapped test to decide whether to merge them. If the test fails to reject the null hypothesis of no change point, we merge the two  blocks into one. Otherwise, we move on to test the next pair of data blocks with the second smallest $DI$. The process will continue until no blocks can be merged.	
The Backward Detection (BD) algorithm is summarized in Algorithm~\ref{alg:multi_bw}. 

\begin{algorithm}[hpt]
	\caption{Backward Detection: BD($\calB^{(k)}$)}
	\label{alg:multi_bw}
	\begin{algorithmic}[1]
		\STATE {Start from data blocks as $\calB^{(k)} = \{ B_1^{(k)}, B_2^{(k)}, \cdots, B_{\nu_k}^{(k)} \}$}
		\STATE {Compute the Dissimilarity Index $DI_i = T_n(B_i^{(k)}, B_{i+1}^{(k)})$ as in (\ref{eqn:dissimilarity_index}) for $i = 1,\dots,\nu_k-1$}
		\STATE {Let $i^* = \argmin DI_i$.}
		\IF {our bootstrap test rejects the null for the segment $[b_{i^*-1}^{(k)}, b_{i^*+1}^{(k)}]$}
		\STATE {Repeat the test for $i^*$ referring to the next smallest $DI_i$ until all pairs are examined}
		\ELSE
		\STATE {Update $B_i^{(k+1)} = B_i^{(k)}$ for $i < i^*$}
		\STATE {Merge $B_{i^*}^{(k)}, B_{i^*+1}^{(k)}$ into one block $B_{i^*}^{(k+1)} = B_{i^*}^{(k)} \cup B_{i^*+1}^{(k)}$ }
		\STATE {Set $B_i^{(k+1)} = B_{i+1}^{(k)}$ for $i > i^*$}
		\STATE {Perform BD($\calB^{(k+1)}$)}
		\ENDIF
		\RETURN{Estimated blocks $\calB$ and corresponding segmentation $ \hat{m}_1 , \dots , \hat{m}_{\hat\nu}$ }
	\end{algorithmic}
	
\end{algorithm}

Compared to forward detection, BD is able to detect short sequence. Hence, the Backward Detection algorithm will be more powerful compared to the direct extension or the block testing at the beginning of this section. There is no signal cancellation issue for BD. Besides, it can identify change points without introducing new estimators or statistics.
However, there is a risk of Type I error inflation since BD recursively performs testing procedure. Let $b_i^{(0)} = iM, i=1,\dots, \lfloor n/M \rfloor$, where $\lfloor n/M \rfloor$ is the largest integer not exceeding $n/M$. Then small $M$ can cause over rejection, while large $M$ may affect estimation accuracy and bring signal cancellation issue back. We should tune the initial partition size $M$ carefully. To the best of our knowledge, there is no theoretical result on the consistency of backward detection in change point estimation. For testing purpose, we can take $M$ as discussed in Section \ref{subsec:block_testing}. Empirical performance are investigated in simulation and real data application.

\section{Simulation study}
\label{sec:simulation}
In this section, we first report simulation results of our method in size approximation and power performance under single change point model. Independent random vectors are generated according to the location-shift model~\eqref{eqn:location_shift_model}. Comparison with other methods follows. In the end, we evaluate the global test of direct extension and the Backward Detection of estimation for multiple change points.

\subsection{Simulation setup}
We generate i.i.d. $\xi_i$ from the following distributions. 
\begin{enumerate}[leftmargin=1cm,itemindent=.5cm,labelwidth=\itemindent,labelsep=0cm,align=left]
	\item Multivariate Gaussian distribution: $\xi_i \sim  N(0, V)$.
	\item Multivariate elliptical $t$-distribution with degree of freedom $\nu$ ($\nu>2$): $\xi_i \sim {t_\nu}(V)$ with the probability density function \cite[Chapter 1]{muirhead1982}
	\begin{equation*}
	f(x; \nu, V) = {\Gamma(\nu+p)/2 \over \Gamma(\nu/2) (\nu \pi)^{p/2} \det(V)^{1/2}} \left( 1 + {x^\top V^{-1} x \over \nu} \right)^{-(\nu+p)/2}.
	\end{equation*}
	The covariance matrix of $\xi_i$ is $\Sigma = {\nu \over \nu-2} V$. In our simulation, we use $\nu=6$.
	
	\item Contaminated Gaussian distribution (i.e., Gaussian mixture model): $\xi_i \sim \text{ctm-G}(\varepsilon, \nu, V) = (1-\varepsilon) N(0,V) + \varepsilon N(0,\nu^2 V)$ with the probability density function
	\begin{align*}
	f(x; \varepsilon, \nu, V) ={1-\varepsilon \over (2\pi)^{p/2} \det(V)^{1/2}} &\exp\left(-{x^\top V^{-1} x \over2} \right)\\ 
	&+ {\varepsilon \over (2 \pi \nu^2)^{p/2} \det(V)^{1/2}} \exp\left(-{x^\top V^{-1} x \over2 \nu^2} \right).
	\end{align*}
	The covariance matrix of $\xi_i$ is $\Sigma = [ (1-\varepsilon) + \varepsilon \nu^2 ] V$. We set $\varepsilon = 0.2$ and $\nu = 2$.
	
	\item Scale transformation of Cauchy distribution:  $\xi_i = V^{1/2} \eta_{i}$, where $\eta_{i} = (\eta_{i1},\dots,\eta_{ip})^{T}$ and $\eta_{ij}$ are i.i.d.\ standard (univariate) Cauchy distribution. 
\end{enumerate}
For each distribution, we consider three spatial dependence structures of $V$. 
\begin{enumerate}[leftmargin=1cm,itemindent=.4cm,labelwidth=\itemindent,labelsep=0.3cm,align=parleft, label=(\Roman*)]
	\item Independent: $V=\Id_p$, where $\Id_p$ is the $p \times p$ identity matrix.
	\item Strongly dependent: $V=0.8 J + 0.2 \Id_p$, where $J$ is the $p \times p$ matrix of all ones.
	\item Moderately dependent: $V_{ij} = 0.8^{|i-j|},\ i,j=1, \dots, p$. 
\end{enumerate}
Unless explicitly indicated, $B=200$ bootstrap samples are drawn for each testing procedure and all results are averaged on 500 simulations. We fix the sample size $n=500$ and dimension $p=600$ for single change point scenario and focus on the performance of two kernels: the linear kernel $h(x,y) = x-y$ and the sign kernel $h(x,y) = \sign(x-y)$.

\subsection{Size approximation}

Let $\hat{R}(\alpha)$ be the proportion of empirically rejected null hypothesis at significance level $\alpha \in (0, 1)$. There are several observations we can draw from Table~\ref{tab:size_1}, which shows the empirical uniform error-in-size, $\sup_{\alpha \in (0,1)} |\hat{R}(\alpha) - \alpha|$.
First, the dependence structure of $V$ does not  influence the errors remarkably. 
Second, for Gaussian, $t_6$ and contaminated Gaussian (ctm-G) distributions, the two kernels have very similar errors in size. For the Cauchy distribution which is only applicable for the sign kernel, error-in-size is comparable with the other three distribution settings.
Therefore, we conclude that under $H_0$, the sign kernel gains robustness without losing much accuracy. 
Three example curves are displayed additional in Figure~\ref{fig:size_error_Eg} to visualize the size approximation.


\begin{table}[htb]
	\vskip .2cm
	\setlength\tabcolsep{2.5pt}
	\centering
	\caption{Uniform error-in-size under $H_0$.}
	\begin{tabular}{cl|ccc|cccc}
		\hline
		\multicolumn{2}{c|}{\multirow{2}{*}{$\sup_{\alpha \in (0,1)} |\hat{R}(\alpha) - \alpha|$}} & \multicolumn{3}{c|}{linear kernel} & \multicolumn{4}{c}{sign kernel}         \\ \cline{3-9} 
		\multicolumn{2}{c|}{}                                                                      & Gaussian  & $t_6$  & ctm-G & Gaussian & $t_6$ & ctm-G & Cauchy \\ \hline
		I                                  & $V=\Id_p$                                             & 0.034     & 0.086  & 0.040        & 0.026    & 0.066 & 0.032        & 0.028  \\
		II                                 & $V=0.8 J + 0.2 \Id_p$                                 & 0.054     & 0.020  & 0.058        & 0.064    & 0.040 & 0.050        & 0.060  \\
		III                                & $V_{ij} = 0.8^{|i-j|}$                                & 0.026     & 0.048  & 0.040        & 0.040    & 0.036 & 0.060        & 0.058  \\ \hline
	\end{tabular}
	\label{tab:size_1}
\end{table}

\begin{figure}[htb] 
	\centering
	\includegraphics[height=3.95cm]{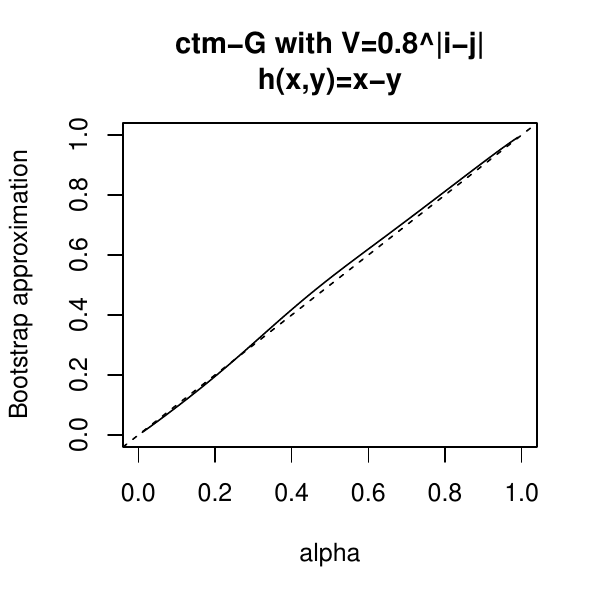}
	\includegraphics[height=3.95cm]{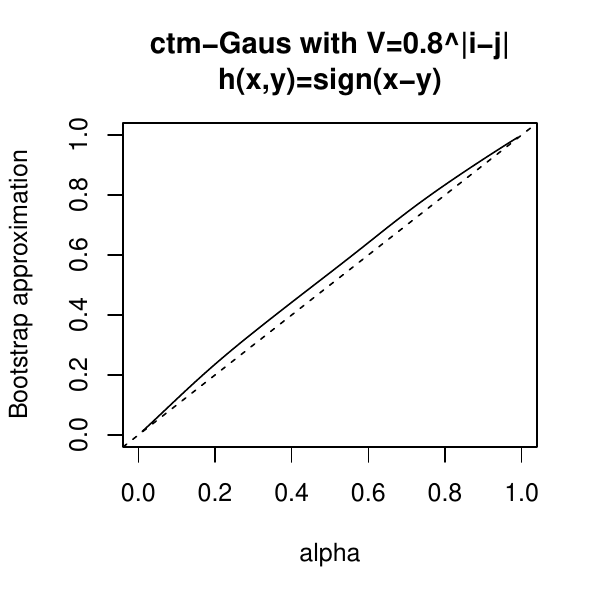}
	\includegraphics[height=3.95cm]{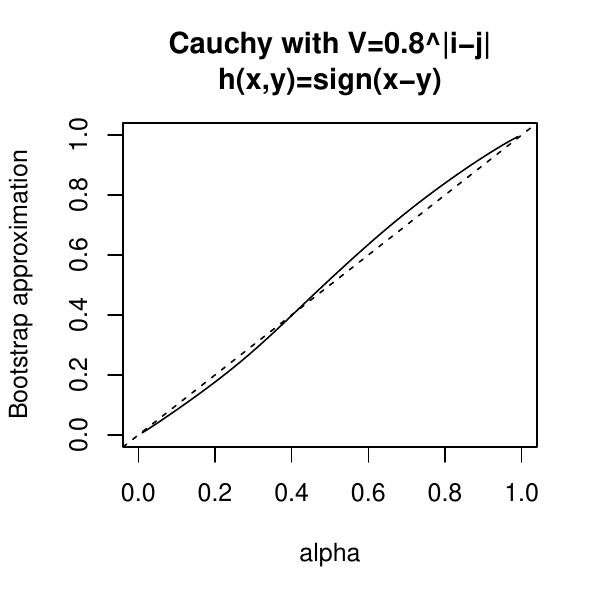}
	\vspace*{-10pt}
	\caption{Selected setups for comparing $\hat{R}(\alpha)$ along with $\alpha$. See headlines for corresponding distribution and kernel.
	}
	\label{fig:size_error_Eg} 
\end{figure}


We also compare our test using the linear kernel to the CUSUM counterpart in \cite[\texttt{BABS}]{yuchen2017finite} under the same setting with the boundary removal parameter as $\us = 40$. 
Table~\ref{tab:size_0.1} displays corresponding simulation results. 
By comparing it to Table~\ref{tab:size_1}, we observe that the CUSUM approach suffers from greater size distortion as it has larger uniform errors in general.
When we focus on the maximum error within the interval $\alpha \in (0,0.1]$ (that are common choices in real applications), our linear kernel based algorithm still outperforms. In addition, our test demands no more computational costs and it enjoys flexibility of no tuning parameter. 

\begin{table}[h]
	\vskip .2cm
	\setlength\tabcolsep{2.5pt}
	\centering
	\caption{Error-in-size $\sup_{\alpha} |\hat{R}(\alpha) - \alpha|$ for $\alpha \in (0,1)$ and $\alpha \in (0,0.1]$ }
	\begin{tabular}{c|ccc||ccc|ccc}
		\hline
		\multirow{3}{*}{} & \multicolumn{3}{c||}{$\sup_{\alpha \in (0,1)} |\hat{R}(\alpha) - \alpha|$} & \multicolumn{6}{c}{$\sup_{\alpha \in (0,0.1]} |\hat{R}(\alpha) - \alpha|$} \\ \cline{2-10} 
		& \multicolumn{3}{c||}{CUSUM approach}                                       & \multicolumn{3}{c|}{CUSUM approach}   & \multicolumn{3}{c}{linear kernel}   \\ 
		& Gaussian                  & $t_6$                 & ctm-G                 & Gaussian     & $t_6$     & ctm-G     & Gaussian     & $t_6$     & ctm-G     \\ \hline
		I                                                    & 0.072          & 0.122 & 0.096                                                  & 0.040          & 0.036 & 0.064        & 0.012         & 0.010 & 0.020        \\
		II                                                   & 0.066          & 0.044 & 0.048                                                  & 0.026          & 0.014 & 0.024        & 0.008         & 0.014 & 0.012        \\
		III                                                  & 0.074          & 0.092 & 0.066                                                  & 0.022          & 0.038 & 0.048        & 0.020         & 0.018 & 0.012          \\ \hline
	\end{tabular}
	\label{tab:size_0.1}
	\vspace*{-10pt}
\end{table}

\subsection{Power of the bootstrap test}
Under $H_{1}$, the signal vector is chosen as $\theta = (\theta_{1},0,\dots,0)^{T}$ such that $\theta_1 = |\theta|_\infty$.  We vary the change point location $m = 50, 150, 250$.
Figure~\ref{fig: PowerEg} displays the power curves for different kernels, change point location $m$ and dependence structure $V$. 
The left panel investigates kernel and location impact. Change point at center $m=n/2 = 250$ (solid curves) is easier to detect than that of $m=n/10=50$ at boundary (dashed curves) regardless of the choice of kernel. For standard Gaussian distribution, the linear kernel has greater power than the sign kernel when the change occurs at boundary point $m=50$, but the relation reverses when $m=250$. 
The middle panel uses linear kernel as an example to illustrate the observation that the dependence structure $V$ does not significantly influence the power, though our $\ell^{\infty}$-type test statistic has advantage in the strong dependence case. 
The right panel displays the power of the sign kernel for Cauchy distributed data to highlight its robustness to location parameter $\theta$ and the impact from change point position $m$.
Regarding to the exact power values, see Table~\ref{tab:power_linear} (linear kernel) and \ref{tab:power_sign} (sign kernel) in Appendix.

\begin{figure}[h] 
	\centering
	\includegraphics[trim=0 0 250 0,clip, height=4cm]{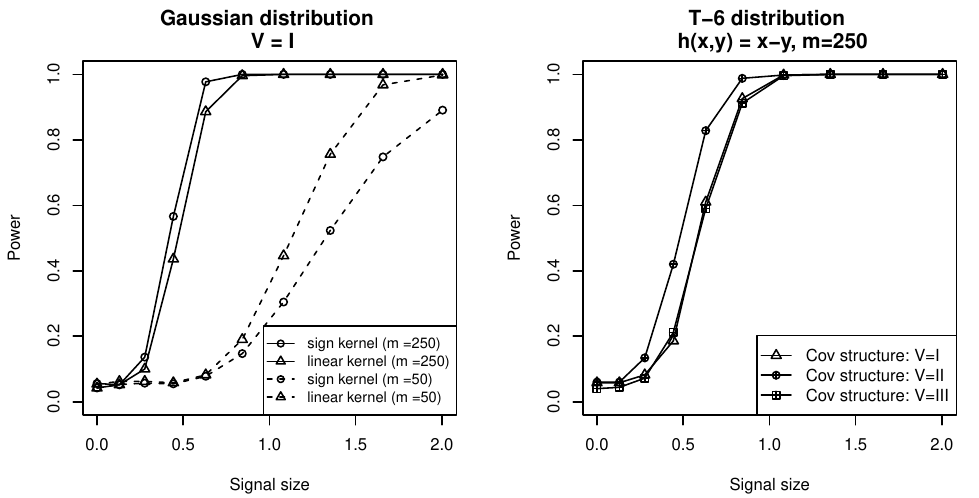} 
	\includegraphics[trim=480 0 0 0,clip, height=4cm]{power_eg.pdf} 
	\caption{Selected setups for comparing power curves. See headlines and legends for corresponding distribution, kernel, covariance structures and change point location $m$.
	}
	\label{fig: PowerEg} 
\end{figure}

\subsection{Comparison with other methods}

We compare our $U$-statistic approach to other competing algorithms in change point literature. 
The linear and sign kernels of our approach are used.  All of the four competitors, namely \cite[\texttt{BABS}]{yuchen2017finite}, \cite[\texttt{Jirak}]{jirak2015}, \cite[\texttt{SBS}]{chofryzlewicz2015} and \cite[\texttt{Inspect}]{wangsamworth2017}, are based on CUSUM statistics. Among them, \texttt{BABS} and \texttt{Jirak} are $\ell^\infty$-type bootstrap test for single change point using different weights on $(s(n-s)/s)$ in (\ref{eqn:cusum_mean_Rp}), the latter of which needs cross-sectional variance estimation on each dimension and it is sensitive to mean shift near the center of data sequence.  The last two competitors target on multiple change point estimation where \texttt{SBS} is thresholded $\ell^1$-type estimator and \texttt{Inspect} is projection based. We adopt their  single change point version function in corresponding R packages and convert them to tests using their default threshold computing functions. 
In our simulation, we set ${n=500, p=600}, \alpha =0.05, m=150$, and set boundary removal as 40 for \texttt{BABS}, \texttt{Jirak} and \texttt{SBS}. 

Table~\ref{tab:comparison_power} compares the power of different tests when the signal $\theta_1$ is growing. It is clear that \texttt{SBS} and \texttt{Inspect} are not suitable in our setting since the location shift parameter is extremely sparse. When the data generating mechanism is not standard multivariate Gaussian (i.e.\ not Gaussian-I in the table), these two algorithms trigger excessive false alarms when $\theta = 0$ and do not return monotone powers as $\theta$ increase. The other two competitors \texttt{BABS} and \texttt{Jirak} behave similarly and return slightly higher powers than ours in general. Note that these two approaches need to pick boundary removal parameter, which can harm powers if it is too large to include true $m$ in the working interval. The contrasts between linear and sign kernel have been discussed in the previous part. 
Therefore, Table~\ref{tab:comparison_power} indicates that our method, which enjoys tuning-free and intermediate-estimation-free properties, is competent in empirical studies.

\begin{table}[!htp]
	\vskip .2cm
	\setlength\tabcolsep{2.5pt}
	\caption{Powers for our method using linear and sign kernels,  \cite[\texttt{BABS}]{yuchen2017finite}, \cite[\texttt{Jirak}]{jirak2015}, \cite[\texttt{SBS}]{chofryzlewicz2015} and \cite[\texttt{Inspect}]{wangsamworth2017}.}
	\label{tab:comparison_power}
	\centering
	\begin{tabular}{c|cccccc|cccccc}
		\hline
		\multirow{2}{*}{$|\theta|_\infty$} & \multicolumn{6}{c|}{Gaussian-I}                 & \multicolumn{6}{c}{Gaussian-II}                \\\cline{2-13}
		& linear & sign  & \texttt{BABS} & \texttt{Jirak} & \texttt{SBS}   & \texttt{Inspect} & linear & sign  & \texttt{BABS} & \texttt{Jirak} & \texttt{SBS}   & \texttt{Inspect} \\\hline
		0                                  & 0.030  & 0.049 & 0.042 & 0.061 & 0.764 & 0.020   & 0.042  & 0.037 & 0.056 & 0.052 & 0.092 & 0.833   \\
		0.28                               & 0.088  & 0.070 & 0.087 & 0.110 & 0.836 & 0.021   & 0.216  & 0.154 & 0.209 & 0.232 & 0.264 & 0.724   \\
		0.44                               & 0.414  & 0.342 & 0.502 & 0.553 & 0.928 & 0.006   & 0.738  & 0.619 & 0.756 & 0.828 & 0.744 & 0.458   \\
		0.63                               & 0.890  & 0.830 & 0.966 & 0.967 & 0.976 & 0.001   & 0.996  & 0.982 & 0.996 & 0.999 & 0.926 & 0.287   \\
		0.84                               & 0.998  & 0.992 & 1     & 1     & 0.966 & 0.003   & 1      & 1     & 1     & 1     & 0.906 & 0.205   \\
		1.08                               & 1      & 1     & 1     & 1     & 0.972 & 0.093   & 1      & 1     & 1     & 1     & 0.898 & 0.183   \\
		1.35                               & 1      & 1     & 1     & 1     & 0.954 & 0.789   & 1      & 1     & 1     & 1     & 0.858 & 0.287   \\
		1.66                               & 1      & 1     & 1     & 1     & 0.938 & 0.999   & 1      & 1     & 1     & 1     & 0.838 & 0.997   \\
		2.00                               & 1      & 1     & 1     & 1     & 0.936 & 1       & 1      & 1     & 1     & 1     & 0.834 & 1       \\\hline\hline
		\multirow{2}{*}{$|\theta|_\infty$}                  & \multicolumn{6}{c|}{ctm-Gaussian-I}             & \multicolumn{6}{c}{$t_6$-II}                   \\\cline{2-13}
		& linear & sign  & \texttt{BABS} & \texttt{Jirak} & \texttt{SBS}   & \texttt{Inspect} & linear & sign  & \texttt{BABS} & \texttt{Jirak} & \texttt{SBS}   & \texttt{Inspect}\\ \hline
		0                                  & 0.030  & 0.051 & 0.020 & 0.067 & 0.592 & 1       & 0.060  & 0.068 & 0.044 & 0.053 & 0.060 & 0.975   \\
		0.28                               & 0.036  & 0.073 & 0.033 & 0.076 & 0.630 & 1       & 0.124  & 0.148 & 0.109 & 0.132 & 0.108 & 0.942   \\
		0.44                               & 0.150  & 0.189 & 0.186 & 0.245 & 0.752 & 1       & 0.418  & 0.451 & 0.477 & 0.537 & 0.418 & 0.791   \\
		0.63                               & 0.524  & 0.593 & 0.675 & 0.750 & 0.904 & 1       & 0.878  & 0.912 & 0.919 & 0.936 & 0.856 & 0.629   \\
		0.84                               & 0.940  & 0.941 & 0.977 & 0.987 & 0.954 & 1       & 0.998  & 1     & 0.997 & 1     & 0.928 & 0.507   \\
		1.08                               & 1      & 1     & 0.999 & 1     & 0.946 & 1       & 1      & 1     & 1     & 1     & 0.898 & 0.453   \\
		1.35                               & 1      & 1     & 1     & 1     & 0.938 & 1       & 1      & 1     & 1     & 1     & 0.878 & 0.609   \\
		1.66                               & 1      & 1     & 1     & 1     & 0.918 & 1       & 1      & 1     & 1     & 1     & 0.846 & 1       \\
		2.00                               & 1      & 1     & 1     & 1     & 0.902 & 1       & 1      & 1     & 1     & 1     & 0.864 & 1      \\\hline
	\end{tabular}
\end{table}

For fair comparison, we do not use Cauchy distribution, since all methods, except for our sign kernel method, will fail when there is no well-defined mean parameter in the heavy tailed distribution.
Unreported results show that \texttt{SBS} and \texttt{Inspect} perform better when the mean change is denser.
We also remark that the Double CUSUM Binary Segmentation \cite[\texttt{DCBS}]{cho2016change} cannot detect any change point under our setting when $|\theta|_\infty \le 2$ because the setup is an extremely sparse case, so the table does not include it. 

{Section~\ref{subsec:app_additional_comparisons} in Appendix presents some further comparison for the size control of \texttt{BABS}, \texttt{Jirak} and our linear kernel approach under $H_0$ with fixed $p$ and boundary removal fraction while varying the sample size $n$.}

\subsection{Multiple change-point detection}

In the multiple change-point scenario, we first let the $k$-th component of $\theta^{(k)}$ to have the same location shift, i.e. $\theta_{1}^{(1)} = \theta_{2}^{(2)} = \dots = \theta_{n,\nu}^{(\nu)} = \delta \ne 0$.
Since change point estimation can be viewed as a special case of clustering, the accuracy can be measured by the adjusted Rand index (ARI) \cite{rand1971objective,hubert1985comparing}. We also report average ARI over all 500 runs. The bootstrap resampling is $200$.

To start with, we consider the direct application of our test using Gaussian distribution and linear kernel as a representative. Let $n = 1000, p=1200$, $ \alpha = 0.05$, and the two change points $(m_1, m_2) = (300, 600)$. The powers are shown in Table \ref{tab:multiple_power}. Our test works well as there is no signal cancellation. 

\begin{table}[hb]
	\centering
	\caption{Powers under multiple change point scenario using linear kernel. Here,  $(m_1, m_2) = (300, 600)$.}
	\label{tab:multiple_power}
	\begin{tabular}{rc|ccccc}
		\hline
		& $\delta$   & 0     & 0.317 & 0.733 & 1.282 & 2.004 \\ \hline
		\multirow{3}{*}{\begin{tabular}[c]{@{}c@{}}Spacial \\ dependent \\ structures\end{tabular}} & I   & 0.052 & 0.278 & 1     & 1     & 1     \\
		& II  & 0.064 & 0.510  & 1     & 1     & 1     \\
		& III & 0.070  & 0.222 & 0.996 & 1     & 1    \\ \hline
	\end{tabular}
\end{table}

Next, we apply the Backward Detection algorithm to estimate change points. We set the initial data blocks as segments of every $M=100$ data points and take the Gaussian distribution with moderate dependence structure (III) for instance. The estimated change points are summarized in Table~\ref{tab:linear_N6} (counts and ARIs) and Figure~\ref{fig:multiple_BD_linear_N6} (estimates). When signal $\delta = 0.317$ is small, BD fails to reject $H_0$ in about half of the time (276 out of 500) and it cannot locate the shifts accurately (small ARIs). However, as signal gets larger, both the number and the locations of change points can be detected consistently (under proper setup of initial data blocks). Meanwhile, ARIs are also increasing to 1, which stands for the perfect estimation. We further add one more change where $(m_1, m_2, m_3) = (300, 600, 800)$. The results in Table~\ref{tab:linear_N6} and Figure~\ref{fig:multiple_BD_linear_N6} are similar to that of two change point case. 

\begin{table}[htb]
	\vskip .2cm
	\setlength\tabcolsep{2.5pt}
	\centering
	\caption{Estimation of multiple change points for $M=100$: counts and ARIs. Here, the data is  Gaussian distributed with dependence structure (III) and the \textit{linear} kernel is used.}
	\label{tab:linear_N6}
	\begin{tabular}{cc|ccccc|ccccc}
		\hline
		&   & \multicolumn{5}{c|}{$(m_1, m_2) = (300, 600)$}                                                   & \multicolumn{5}{c}{$(m_1, m_2, m_3) = (300, 600, 800)$}                                                  \\ \hline
		\multicolumn{2}{c|}{$\delta$}                                                                           & 0                         & 0.317                     & 0.733                     & 1.282                     & 2.004        & 0                         & 0.317                     & 0.733                     & 1.282                     & 2.004        \\ \hline
		\multirow{6}{*}{\begin{tabular}[c]{@{}c@{}}Estimated\\ number\\ of\\ change\\ points\end{tabular}} & 0 & \textbf{497} & 276         & 0            & 0            & 0            & \textbf{494} & 270        & 0            & 0            & 0            \\
		& 1 & 3            & 209         & 0            & 0            & 0            & 6            & 217        & 0            & 0            & 0            \\
		& 2 & 0            & \textbf{15} & \textbf{484} & \textbf{492} & \textbf{483} & 0            & 13         & 32           & 0            & 0            \\
		& 3 & 0            & 0           & 16           & 7            & 17           & 0            & \textbf{0} & \textbf{455} & \textbf{474} & \textbf{483} \\
		& 4 & 0            & 0           & 0            & 1            & 0            & 0            & 0          & 13           & 25           & 17           \\
		& 5 & 0            & 0           & 0            & 0            & 0            & 0            & 0          & 0            & 1            & 0            \\ \hline
		\multicolumn{2}{c|}{Sum}                                                                                & 500          & 500         & 500          & 500          & 500          & 500          & 500        & 500          & 500          & 500          \\ \hline
		\multicolumn{2}{c|}{ARI}                                                                                & 0.994        & 0.195       & 0.933        & 0.998        & 0.996        & 0.988        & 0.152      & 0.920        & 0.995        & 0.997       \\\hline
	\end{tabular}
\end{table}

\begin{figure}[htb]
	\centering
	\includegraphics[trim=190 0 0 0,clip,width = 0.7\textwidth]{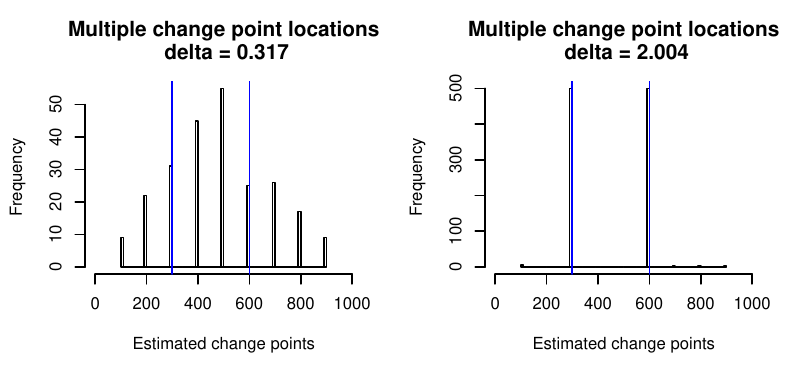}
	\includegraphics[trim=190 0 0 0,clip,width = 0.7\textwidth]{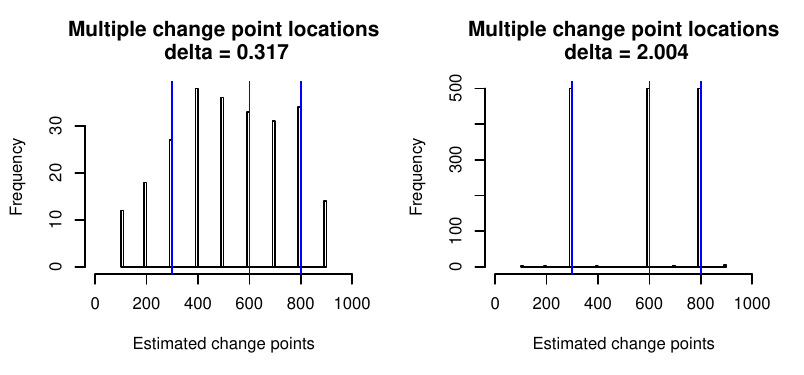}
	\caption{Multiple change point setup using \textit{linear} kernel at signal level $\delta=0.822, 10.023$. Upper: 2 change points $(m_1, m_2) = (300, 600)$. Lower:  3 change points $(m_1, m_2, m_3) = (300, 600, 800)$.}
	\label{fig:multiple_BD_linear_N6}
\end{figure}

Then, we also use the sign kernel to detect location shift for Cauchy distribution with dependence structure (III). Analogously, initial data blocks are segments of every $M=100$ data points in sequence. The cases of 2 change points $(m_1, m_2) = (300, 600)$ and 3 change points $(m_1, m_2, m_3) = (300, 600, 800)$ are implemented and the results are shown in Table~\ref{tab:sign_C6} and Figure~\ref{fig:multiple_BD_sign_C6}. Similar conclusion can be drawn except that stronger signal strength is required as Cauchy distribution has extremely heavy tails.

\begin{table}[htb]
	\vskip .2cm
	\setlength\tabcolsep{2.5pt}
	\centering
	\caption{Estimation of multiple change points for $M=100$. Here, the data is Cauchy distributed with dependence structure (III) and the \textit{sign} kernel is used.}
	\label{tab:sign_C6}
	\begin{tabular}{cc|ccccc|ccccc}
		\hline
		&   & \multicolumn{5}{c|}{$(m_1, m_2) = (300, 600)$}                                                   & \multicolumn{5}{c}{$(m_1, m_2, m_3) = (300, 600, 800)$}                                                  \\ \hline
		\multicolumn{2}{c|}{$\delta$}                                                                           & 0                         & 0.822                     & 2.320                     & 5.050                     & 10.023        & 0                         & 0.822                     & 2.320                     & 5.050                     & 10.023        \\ \hline
		\multirow{6}{*}{\begin{tabular}[c]{@{}c@{}}Estimated\\ number\\ of\\ change\\ points\end{tabular}} & 0 & \textbf{465} & 44         & 0            & 0            & 0            & \textbf{460} & 36        & 0            & 0            & 0            \\
		& 1 & 6            & 257         & 0            & 0            & 0            & 11            & 221        & 0            & 0            & 0            \\
		& 2 & 6            & \textbf{173} & \textbf{365} & \textbf{470} & \textbf{470} & 4            & 172         & 0           & 0            & 0            \\
		& 3 & 6            & 9           & 18           & 12            & 10           & 3            & \textbf{50} & \textbf{401} & \textbf{470} & \textbf{477} \\
		& 4 & 5            & 11           & 21            & 15            & 12            & 8            & 9          & 19           & 11           & 8           \\
		& 5 & 6            & 1           & 59            & 1            & 1            & 5            & 6          & 66            & 1            & 0            \\ 
		& 6 & 6            & 5           & 46            & 2            & 7            & 9            & 6          & 14            & 18            & 15            \\ \hline
		\multicolumn{2}{c|}{Sum}                                                                                & 500          & 500         & 500          & 500          & 500          & 500          & 500        & 500          & 500          & 500          \\ \hline
		\multicolumn{2}{c|}{ARI}                                                                                & 0.930        & 0.557       & 0.888        & 0.986        & 0.983        & 0.920        & 0.495      & 0.951        & 0.986        & 0.989       \\\hline
	\end{tabular}
\end{table}

\begin{figure}[htb]
	\centering
	\includegraphics[trim=190 0 0 0,clip,width = 0.7\textwidth]{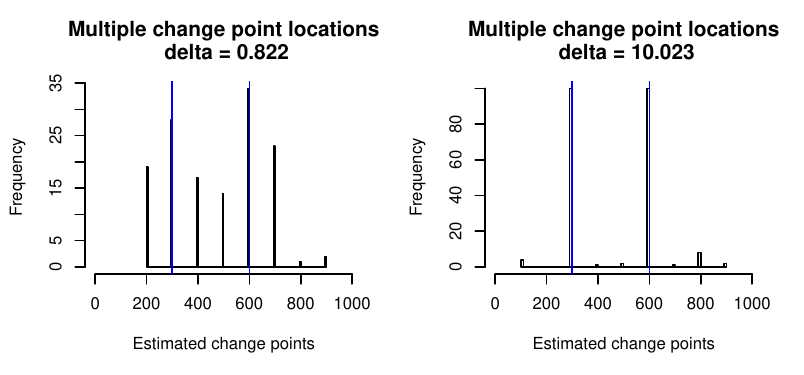}
	\includegraphics[trim=190 0 0 0,clip,width = 0.7\textwidth]{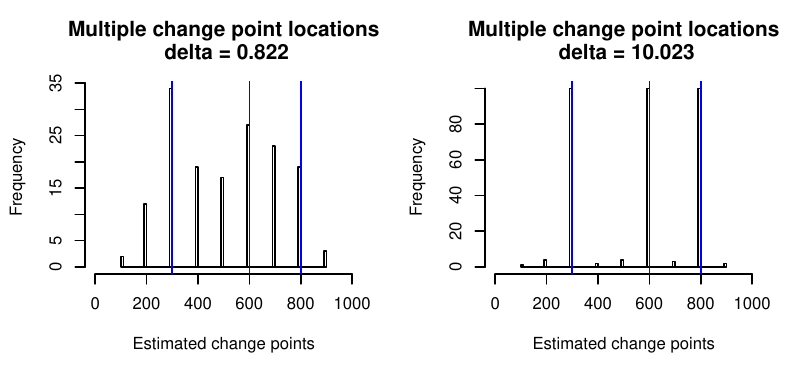}
	\caption{Multiple change point setup using \textit{sign} kernel at signal level $\delta=0.822, 10.023$. Upper: 2 change points $(m_1, m_2) = (300, 600)$. Lower:  3 change points $(m_1, m_2,m_3) = (300, 600, 800)$.}
	\label{fig:multiple_BD_sign_C6}
\end{figure}

Lastly, we set $M=1$ and repeat the experiment using linear kernel and Gaussian distribution with dependence structure (III). The results are summarized in Table~\ref{tab:linear_N6_M1} and Figure~\ref{fig:multiple_BD_sign_C6_M1}. Compared to Table~\ref{tab:linear_N6} and Figure~\ref{fig:multiple_BD_sign_C6} which correspond to the same setting but $M=100$, we can easily observe over rejection issue since more change points are concluded than the truth for both cases. However, when signal is large ($\delta = 2.004$), estimated change points still concentrate around the true $m_i$'s. In practice, a threshold $\underline{m}$ can be introduced to force merging two blocks if the cardinality of their union is small. 

\begin{table}[htb]
	\vskip .2cm
	\setlength\tabcolsep{2.5pt}
	\centering
	\caption{Estimation of multiple change points for $M=1$. Here, the data is  Gaussian distributed with dependence structure (III) and \textit{linear} kernel is used.}
	\label{tab:linear_N6_M1}
	\begin{tabular}{cc|ccccc|ccccc}
		\hline
		&   & \multicolumn{5}{c|}{$(m_1, m_2) = (300, 600)$}                                                   & \multicolumn{5}{c}{$(m_1, m_2, m_3) = (300, 600, 800)$}                                                  \\ \hline
		\multicolumn{2}{c|}{$\delta$}                                                                           & 0                         & 0.317                     & 0.733                     & 1.282                     & 2.004        & 0                         & 0.317                     & 0.733                     & 1.282                     & 2.004        \\ \hline
		\multirow{6}{*}{\begin{tabular}[c]{@{}c@{}}Estimated\\ number\\ of\\ change\\ points\end{tabular}} & 0   & \textbf{475}   & 205   & 0     & 0     & 0     & 477   & 195   & 0     & 0     & 0     \\
		&1   & 21    & 230   & 3     & 0     & 0     & 20    & 224   & 0     & 0     & 0     \\
		&2   & 3     & \textbf{59}    & \textbf{367}   & \textbf{343}   & \textbf{344}   & 3     & 77    & 51    & 0     & 0     \\
		&3   & 1     & 5     & 114   & 135   & 133   & 0     & \textbf{4}     & \textbf{324}   & \textbf{289}   & \textbf{293}   \\
		&4   & 0     & 1     & 16    & 22    & 23    & 0     & 0     & 111   & 167   & 172   \\
		&5   & 0     & 0     & 0     & 0     & 0     & 0     & 0     & 13    & 38    & 32    \\
		&6   & 0     & 0     & 0     & 0     & 0     & 0     & 0     & 1     & 6     & 2     \\
		&8   & 0     & 0     & 0     & 0     & 0     & 0     & 0     & 0     & 0     & 1     \\ \hline
		&Sum & 500   & 500   & 500   & 500   & 500   & 500   & 500   & 500   & 500   & 500   \\ \hline
		&ARI & 0.950 & 0.186 & 0.634 & 0.785 & 0.858 & 0.954 & 0.160 & 0.582 & 0.747 & 0.834   \\\hline 
	\end{tabular}
\end{table}

\begin{figure}[htb]
	\centering
	\includegraphics[trim=190 0 0 0,clip,width = 0.65\textwidth]{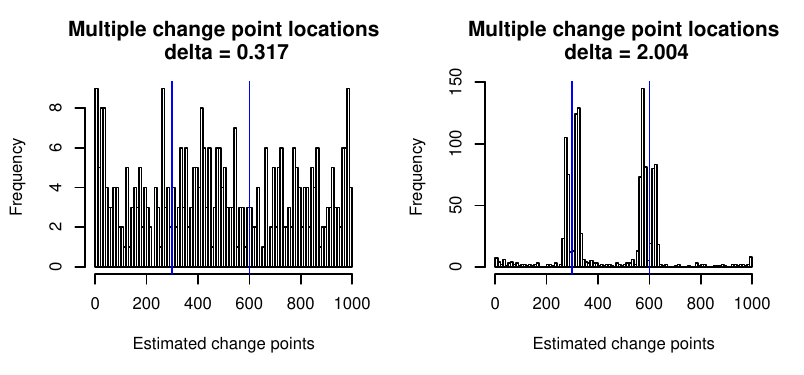}
	\includegraphics[trim=190 0 0 0,clip,width = 0.65\textwidth]{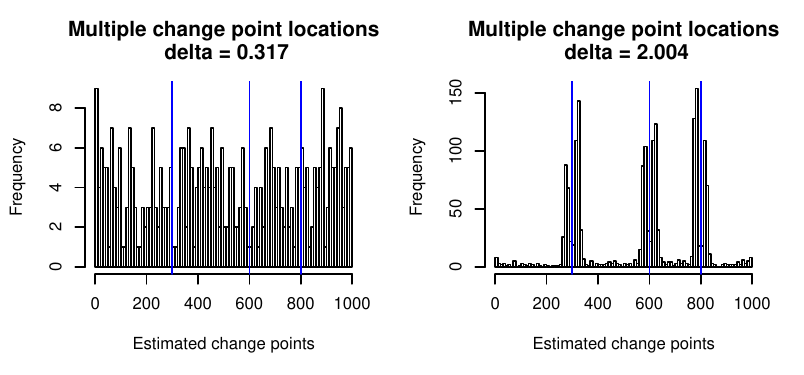}
	\caption{Multiple change point setup using $M=1$ and \textit{linear} kernel at signal level $\delta=0.317, 2.004$. Upper: 2 change points $(m_1, m_2) = (300, 600)$. Lower:  3 change points $(m_1, m_2, m_3) = (300, 600, 800)$.}
	\label{fig:multiple_BD_sign_C6_M1}
\end{figure}

{
	\clearpage
	\subsection{Simulation results for time series data} \label{subsec:sim_extension_time_series}
	We shall study the empirical performance of the bootstrap test for some dependent process $\xi_{i}$.
	In our simulation, we consider the stationary vector autoregression of order 1 (denote as VAR(1)) error process:
	$
	\xi_i = A\xi_{i-1} + \eta_i = \sum_{k=0}^{\infty} A^{k} \eta_{i-k},
	$ 
	where $\{\eta_{i}\}_{i \in \mathbb{Z}}$ is a sequence of i.i.d.\ mean-zero random vectors in $\mathbb{R}^{p}$ and $A$ is a $p \times p$ coefficient matrix, where random matrix $A$ is generated with i.i.d.\ $N(0,1)$ entries. To ensure the stationarity of $\xi_{i}$ process, $A$ is normalized such that $\|A\|_2=1/1.8<1$. 
	
	We first use the linear kernel $h(x,y)=x-y$ and consider different trimming parameters $M=2, 5, 10, 15$. We fix $n=500, p=600, B=200$ and $m=n/2$ under the location-shift model of single change point. Let $\hat{R}(\alpha)$ be the proportion of empirically rejected null hypothesis in 500 simulations. In Table~\ref{tab:H0_size_TS_linear} which provides uniform error-in-size $\sup_{\alpha \in [0,1]} |\hat{R}(\alpha) - \alpha|$, we can observe that larger $M$ needs to be selected if stronger dependence (i.e., the compound symmetry structure II) presents regardless of distribution families. This is the trade-off effect through $M$. We can also find that the best error-in-sizes in each column are comparable to the corresponding values in Table~\ref{tab:size_1}. This indicates the effectiveness of our modified  approach under temporal dependency. Figure~\ref{fig: AlphaEg_TS_linear} displays two examples of $\hat{R}(\alpha)$ under $H_0$ and power under $H_{1}$, where the signal vector is chosen as $\theta = (\theta_{1},0,\dots,0)^{T}$ such that $\theta_1 = |\theta|_\infty$.
	
	Next we use sign kernel $h(x,y)=\sign(x-y)$ and consider the trimming parameters $M=2, 5, 10$. For illustration purpose, we only select the data-generating schemes of Cauchy distribution with Covariance I-III and ctm-Gaussian distribution with Covariance III. The other parameters remain the same as above. The uniform error-in-size for each scenario is give in Table~\ref{tab:H0_size_TS_sign}.
	In general, $M=2$ works the best under each scenario. The non-linear projection by the sign kernel makes the correlation between data pairs weaker. Therefore, it makes the sign kernel more attractive in terms of its robustness against weak temporal dependency. Similarly, two examples are given in Figure~\ref{fig: AlphaEg_TS_sign}.
	
	\begin{table}[htb]
		\caption{\label{tab:H0_size_TS_linear} Uniform error-in-size $\sup_{\alpha \in [0,1]} |\hat{R}(\alpha) - \alpha|$ using linear kernel under $H_0$, where $\xi_i$ are from VAR(1) process. The columns represent distribution families and covariance dependence structures of $\eta_i$ defined in Section 5.1. The rows correspond to different trimming parameter $M$. The smallest errors in each column are highlighted in bold. }
		\setlength\tabcolsep{2.5pt}
		\centering	
		\begin{tabular}{c|ccc|ccc|ccc}
			\hline
			\multirow{2}{*}{$\sup_{\alpha \in [0,1]} |\hat{R}(\alpha) - \alpha|$} & \multicolumn{3}{c|}{Gaussian}                     & \multicolumn{3}{c|}{$t_6$}                        & \multicolumn{3}{c}{ctm-Gaussian}                 \\ \cline{2-10}
			& I              & II             & III            & I              & II             & III            & I              & II             & III            \\ \hline
			M=2                                                  & \textbf{0.058} & 0.092          & \textbf{0.030} & \textbf{0.054} & 0.082          & \textbf{0.028} & \textbf{0.038} & 0.086          & \textbf{0.054} \\
			M=5                                                  & 0.076          & 0.088          & 0.056          & 0.092          & \textbf{0.074} & 0.050          & 0.058          & \textbf{0.082} & 0.086          \\
			M=10                                                 & 0.128          & \textbf{0.064} & 0.102          & 0.134          & 0.080          & 0.080          & 0.106          & 0.084          & 0.136          \\
			M=15                                                 & 0.180          & 0.066          & 0.150          & 0.172          & 0.086          & 0.126          & 0.156          & 0.094          & 0.174         \\ \hline
		\end{tabular}
	\end{table}

	\begin{figure}[htb]
		\begin{subfigure}[t]{\textwidth}
			\centering
			\includegraphics[trim = 0 10 30 45, clip, width=0.85\textwidth]{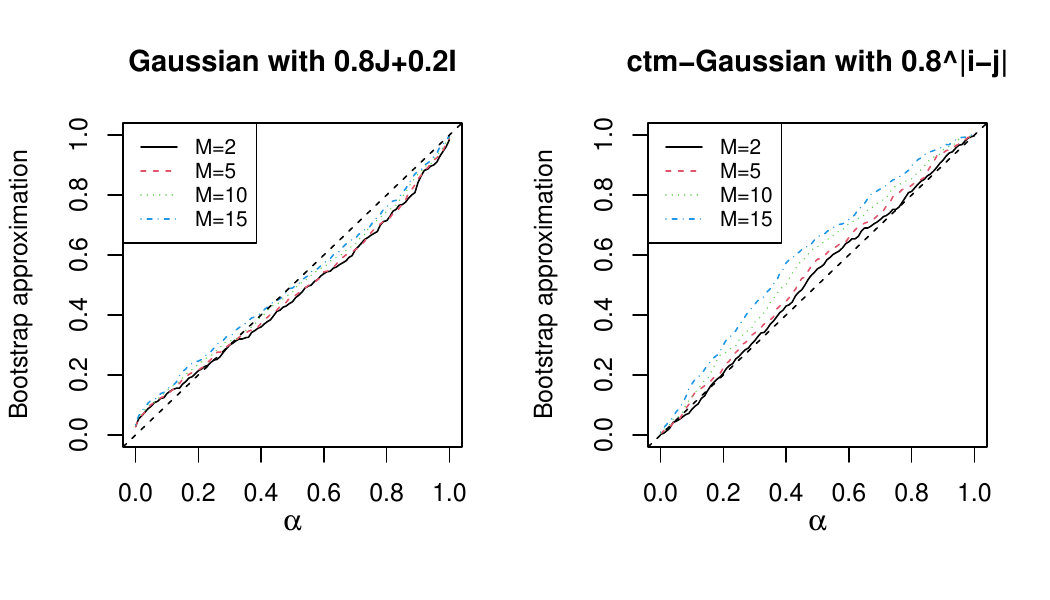}
			\label{fig:TS_linear_subfig:a}
			\vspace{-0.25in}
			\caption{Empirical rejection rate $\hat{R}(\alpha)$ under $H_0$ for $0 \le \alpha \le 1$.}
		\end{subfigure}
		\\
		\begin{subfigure}[t]{\textwidth}
			\centering
			\includegraphics[trim = 0 10 30 45, clip, width=0.85\textwidth]{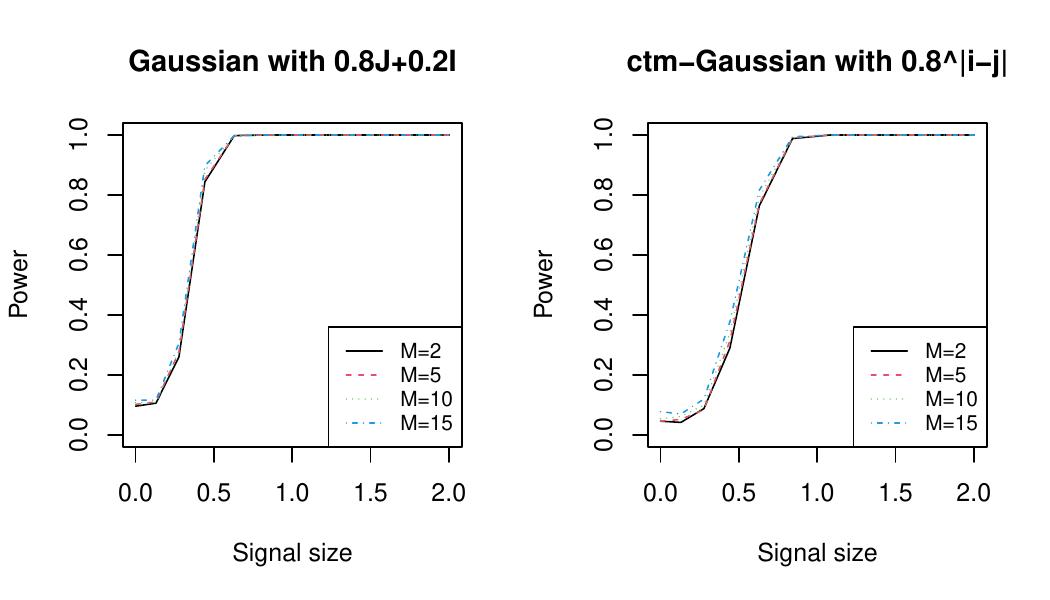}
			\label{fig:TS_linear_subfig:b}
			\vspace{-0.1in}
			\caption{Power under $H_1$ for $0 \le \theta_1 \le 2$.}
		\end{subfigure}
		\caption{Empirical rejection rate $\hat{R}(\alpha)$ under $H_0$ and power under $H_1$ in selected time-series data-generating schemes: (Left) Gaussian distribution with covariance structure II; (Right) ctm-Gaussian distribution with covariance structure III. (Parameters: $n=500$, $p=600$, kernel $h(x,y)=x-y$, and trimming parameters $M = 2,5,10,15$.) } 
		\label{fig: AlphaEg_TS_linear} 
	\end{figure}
	
	\begin{table}[htb]
		\caption{\label{tab:H0_size_TS_sign} Uniform error-in-size $\sup_{\alpha \in [0,1]} |\hat{R}(\alpha) - \alpha|$ using sign kernel under $H_0$, where $\xi_i$ are from VAR(1) process. The columns represent distribution families with covariance structures of $\eta_i$ defined in Section 5.1. The rows correspond to different trimming parameter $M$. The smallest errors in each column are highlighted in bold. }
		\setlength\tabcolsep{2.5pt}
		\centering	
		\begin{tabular}{c|ccc|c}
			\hline
			$\sup_{\alpha \in [0,1]} |\hat{R}(\alpha) - \alpha|$ & Cauchy (I) & Cauchy (II) & Cauchy (III) & ctm-Gaussian (III) \\ \hline
			M=2                                                  & \textbf{0.068}      & \textbf{0.057}       &    \textbf{0.068}           & \textbf{0.060}              \\
			M=5                                                  & 0.094      & 0.062       &     0.096         & 0.088              \\
			M=10                                                 & 0.144      & 0.078       &    0.150          & 0.142     \\\hline        
		\end{tabular}
	\end{table}

	\begin{figure}[htb]
		\begin{subfigure}[t]{\textwidth}
			\centering
			\includegraphics[trim = 0 10 30 45, clip, width=0.85\textwidth]{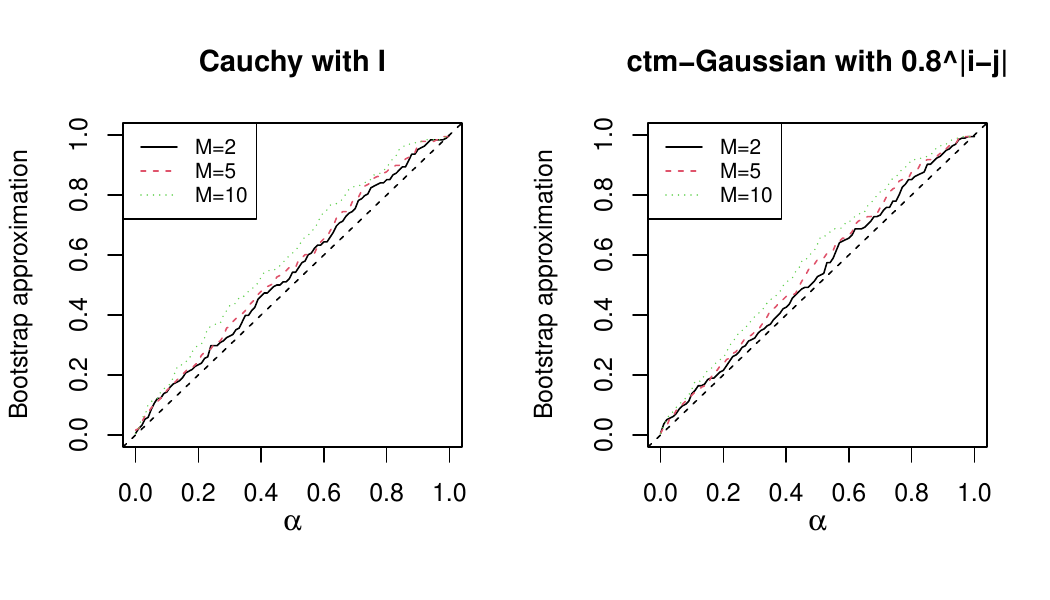}
			\label{fig:TS_sign_subfig:a}
			\vspace{-0.25in}
			\caption{Empirical rejection rate $\hat{R}(\alpha)$ under $H_0$ for $0 \le \alpha \le 1$.}
		\end{subfigure}
		\\
		\begin{subfigure}[t]{\textwidth}
			\centering
			\includegraphics[trim = 0 10 30 45, clip, width=0.85\textwidth]{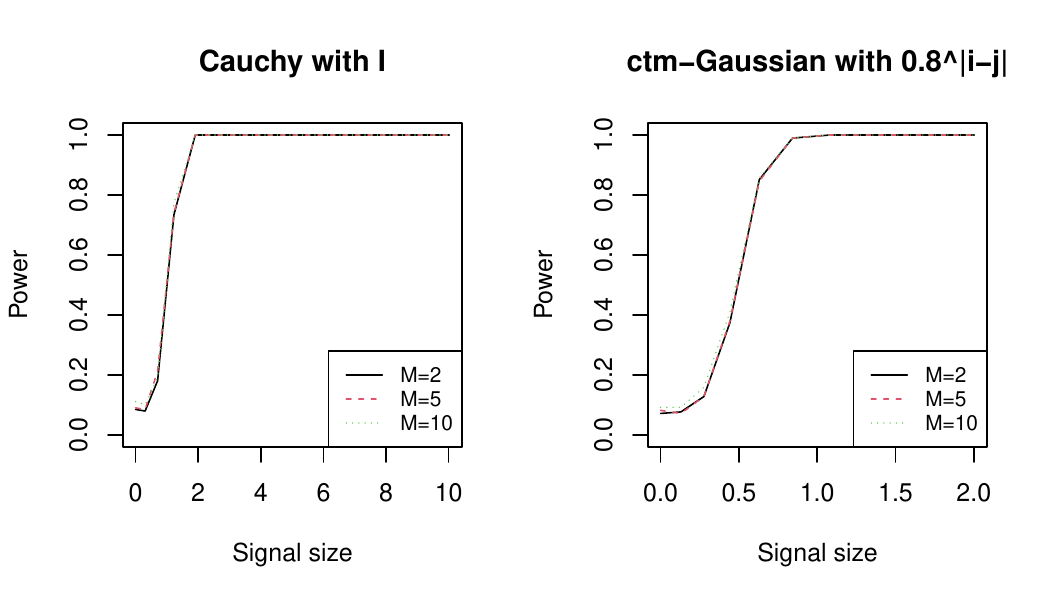}
			\label{fig:TS_sign_subfig:b}
			\vspace{-0.1in}
			\caption{Power under $H_1$ for $0 \le \theta_1 \le 2$.}
		\end{subfigure}
		\caption{Empirical rejection rate $\hat{R}(\alpha)$ under $H_0$ and power under $H_1$ in selected time-series data-generating schemes: (Left) Cauchy distribution with covariance structure II; (Right) ctm-Gaussian distribution with covariance structure III. (Parameters: $n=500$, $p=600$, kernel $h(x,y)=\sign(x-y)$, and trimming parameters $M = 2,5,10$.) } 
		\label{fig: AlphaEg_TS_sign} 
	\end{figure}
	
}
\clearpage
\section{Real Data Applications}
\label{sec:real_data}
\subsection{Single change point: Enron email dataset}
Enron Corporation used to be one of the leading American energy companies. In an accounting scandal, Enron share prices decreased from around \$80 during the summer of 2000 to pennies at the end of 2001. The bankruptcy was filed on 12/02/2001 and it became the largest bankruptcy reorganization in American history at that time.	
The Enron email dataset that contains more than 500,000 messages from about 150 users (mostly senior management) was publicly available during the investigation by the Federal Energy Regulatory Commission in 2002
\footnote{
	The raw data is organized in folders (\url{http://www.cs.cmu.edu/~enron/}) and its tabular format version is available at \url{https://data.world/brianray/enron-email-dataset}. The timeline of major events can be found at \url{http://www.agsm.edu.au/bobm/teaching/BE/Enron/timeline.html}.
} .

We study the collection of messages sent in 2000-2001. To test for the existence of an abrupt changes in email discussions, our analysis is based on the number of emails sent from each user. In order to exclude the yearly trend and temporal dependence, we apply our method to $X_{ij}$  which is the difference of emails sent from user $j$ on the $i$-th day for the two years.  The leap day (02/29/2000) and the users who were inactive during 2000 or 2001 are removed such that the final data matrix $(X_{ij})_{i=1,\dots,n; j=1,\dots,p}$ is of dimension $n=365$ and $p=101$.  We set bootstrap repetition number $B = 2000$.
For the linear kernel, our test statistic has the value $\overline{T}_n = 561.49$ and the 95\% quantile of bootstrapped statistic is 117.17. For the sign kernel, our test statistic has the value $\overline{T}_n = 8.95$ and the 95\% quantile of bootstrapped statistic is 1.44. Both tests reject the null hypothesis of no abrupt change. 
As an illustration of the test results, the aggregated trend of $Y_i = \sum_{j=1}^{101} X_{ij}$ in Figure~\ref{fig: realdata_Enron_Yi} indicates the presence of extensive email communication from the second half of 2000 to the first half of 2001.  Our test confirms that there was abnormal email activity in these two years.

\begin{figure}[htb] 
	\centering
	\includegraphics[width= 3.4in]{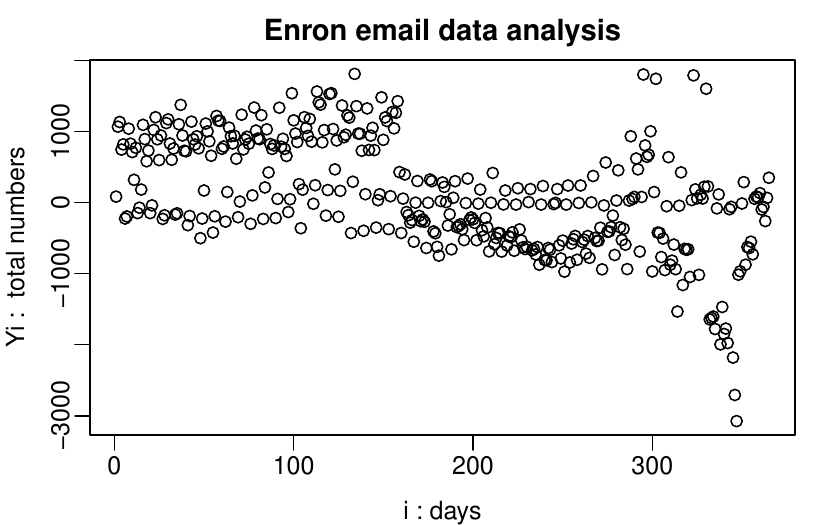} 
	\caption{Trend of $Y_i = \sum_{j=1}^{101} X_{ij}$ for Enron email dataset.}
	\label{fig: realdata_Enron_Yi} 
\end{figure}

\subsection{Multiple change point: Micro-array dataset}
The array comparative genomic hybridization data, \texttt{ACGH} \cite[R package \texttt{ecp}]{james2013ecp}, consists of $p=43$ patients with bladder tumor. We consider to detect change points among their DNA copy number profiles each of which contains $n=2215$ log-intensity-ratio fluorescent measurements.
We apply the BD algorithm using linear kernel and set bootstrap repeats $1000$, significance level $\alpha=0.01$ and initial data block size $M=2$.
The measurements for the first 10 individuals are shown in Figure~\ref{fig:bd_real}.
Our BD algorithm finds 32 change points that marked in red vertical dashed lines. This number is in a reasonable level as indicated in \cite{wangsamworth2017} where the authors only reported 30 most significant ones while their default \texttt{Inspect} algorithm found 254 change points.
The ARI between ours and the bootstrap-assisted binary segmentation \cite[\texttt{BABS}]{yuchen2017finite} which identifies 27 change points is 0.779. As shown in Table~\ref{tab:acgh_compare}, the two methods have overlapped detection that are close loci numbers such as $(73,74), (342,344), (521,528), \dots, (2143, 2142) $.

\begin{figure}[htb]
	\centering
	\includegraphics[width = 0.99\textwidth]{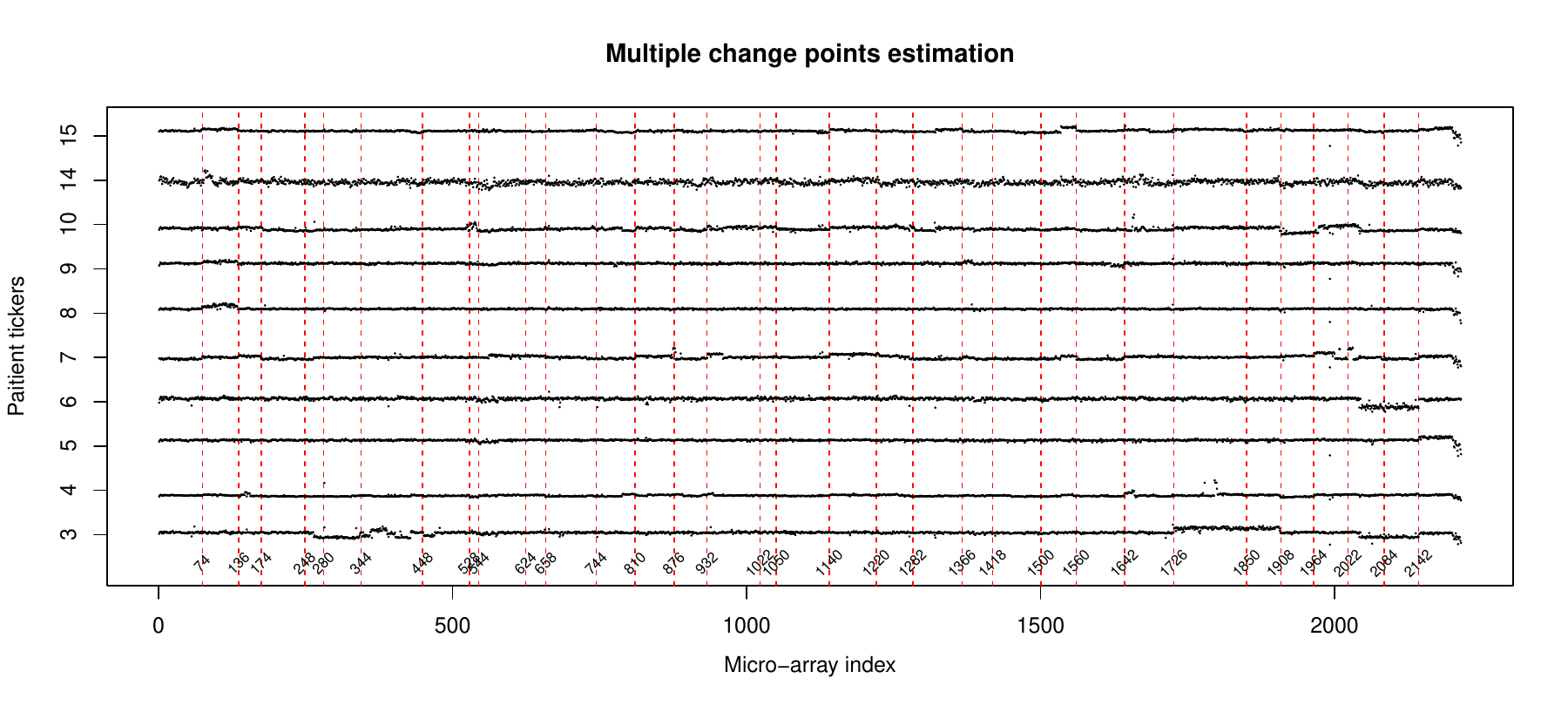}
	\caption{Real data study: aCGH data. Here, we set $B=1000, \alpha=0.01$ and the linear kernel.}
	\label{fig:bd_real}
\end{figure}

\begin{table}[!htp]
	\centering
	\caption{Identified change point locations (loci numbers on genome) in \texttt{ACGH} dataset. }
	\label{tab:acgh_compare}
	\begin{tabular}{p{1cm}|p{10.2cm}}
		\hline
		BABS   & 73, 185, 263, 342, 428, 521, 581, 657, 741, 801, 871, 960, 1051, 1141, 1216, 1276, 1367, 1427, 1503, 1563, 1664, 1724, 1836, 1905, 1965, 2044, 2143.                       \\\hline
		BD     & 74, 136, 174, 248, 280, 344, 448, 528, 544, 624, 658, 744, 810, 876, 932, 1022, 1050, 1140, 1220, 1282, 1366, 1418, 1500, 1560, 1642, 1726, 1850, 1908, 1964, 2022, 2084, 2142. \\\hline
	\end{tabular}
\end{table}

\appendix

\section{Proofs and additional numeric results}
\subsection{Proof of main results}
Throughout the whole proofs, we assume $d \ge 2$, $n \ge 3$ and $n \ge \log^7(nd)$ otherwise the rates will automatically hold. The  $K_i>0, i= 1,2,\dots$ and $C>0$ are large constants that may vary part by part. 

\begin{proof}[Proof of Theorem~\ref{thm:gaussian_approx_rate}]
	Suppose $H_{0}$ is true. Without loss of generality, we may assume $\varpi_{n} \le 1$. 
	
	\underline{\it Step 1. Gaussian approximation to $T_{n}$.} 
	
	Denote $\Gamma = \Cov(g(X_{1}))$. Since the kernel $h$ is anti-symmetric, we have $\E[g(X_{1})] = \vzero$. Thus $\E [L_n] = \vzero$ and 
	\begin{equation*}
	\setlength\abovedisplayskip{0pt}
	\setlength\belowdisplayskip{2pt}
	\Cov(L_n) = n {n \choose 2}^{-2} \sum_{i=1}^n (n+1-2i)^2 \ \Cov(g(X_{i})) = {4(n+1) \over 3(n-1)} \Gamma.
	\end{equation*}
	By Jensen's inequality, we have $\E|g_{j}(X_{i})|^{2+k} \le D_{n}^{k}$ for $k = 1,2$, and $\|g_{j}(X_{i})\|_{\psi_{1}} \le D_{n}$. Then it follows 
	\begin{equation*}
	\setlength\abovedisplayskip{2pt}
	\setlength\belowdisplayskip{2pt}
	{1 \over n} \sum_{i=1}^{n} \left( { 2 \over n-1} \right)^{2+k} |n-2i+1|^{2+k} \E|g_{j}(X_{i})|^{2+k} \lesssim D_{n}^{k} 
	, \quad
	\left\| {2(n-2i+1) \over n-1} g_{j}(X_{i}) \right\|_{\psi_{1}} \lesssim D_{n}. 
	\end{equation*}
	In addition, note that 
	${1 \over n} \sum_{i=1}^{n} 4 \left( {n-2i+1 \over n-1} \right)^{2} \Gamma_{jj} = {n+1 \over n-1} \cdot {4 \over 3} \Gamma_{jj} \ge {4 \over 3} \underline{b} > 0. 
	$
	By Proposition 2.1 in \cite{cck2016a} (applied to the max-hyperrectangles), we have 
	\[
	\rho(\overline{L}_n,\overline{Z}_{n}) \le \left\{ D_n^2 \log^7 (nd) \over n \right\}^{1/6} = \varpi_{n}, 
	\]
	where $\overline{Z}_{n} = \max_{1 \le j \le d} Z_{nj}$ and $Z_{n} \sim N(0, {4(n+1) \over 3(n-1)} \Gamma)$. Let $Z \sim N(0, 4\Gamma/3)$. By the Gaussian comparison inequality (cf. Lemma C.5 in \cite{chen2017randomized}), we have 
	\[
	\setlength\abovedisplayskip{0pt}
	\setlength\belowdisplayskip{2pt}
	\rho(\overline{Z}_n,\overline{Z}) \lesssim  \left( {4 \over 3n} |\Gamma|_{\infty} \log^{2}{d} \right)^{1/3}. 
	\]
	Since $\Gamma_{jj} \le 1 + \E|g_{j}(X_{1})|^{3} \le 1+D_{n} \le 2 D_{n}$,
	it follows from the Cauchy-Schwarz inequality that 
	\[
	\setlength\abovedisplayskip{2pt}
	\setlength\belowdisplayskip{0pt}
	\rho(\overline{Z}_n,\overline{Z}) \lesssim \left( {D_{n} \log^{2}{d} \over n} \right)^{1/3} \lesssim \varpi_{n}. 
	\]
	Then by triangle inequality, we have 
	\begin{equation}
	\label{eqn:gaussian_approx_linear_part}
	\setlength\abovedisplayskip{0pt}
	\setlength\belowdisplayskip{2pt}
	\rho(\overline{L}_n,\overline{Z}) \le \rho(\overline{L}_n,\overline{Z}_{n}) + \rho(\overline{Z}_n,\overline{Z}) \lesssim \varpi_{n}.  
	\end{equation}
	Applying Corollary 5.6 in \cite{chenkato2017a} with $k=2$, we have 
	\begin{equation}
	\label{eqn:gaussian_approx_remainder_part}
	\setlength\abovedisplayskip{0pt}
	\setlength\belowdisplayskip{0pt}
	\E\left( \max_{1 \le j \le d} |R_{nj}| \right) \lesssim {D_{n} {n}^{-1/2} \log{d}  }. 
	\end{equation}
	Then for any $t \in \R$ and $a > 0$, we have 
	\begin{align*}
	\Prob\left( \overline{T}_n \le t \right) &\le \Prob\left( \overline{L}_n \le t+a^{-1} \E[|R_{n}|_{\infty}] \right) + \Prob\left( |R_{n}|_{\infty} > a^{-1} \E[|R_{n}|_{\infty}] \right) \\
	&\le_{(i)} \Prob\left( \overline{L}_n \le t+a^{-1} \E[|R_{n}|_{\infty}] \right) + a \\
	&\le_{(ii)} \Prob\left( \overline{Z} \le t+a^{-1} \E[|R_{n}|_{\infty}] \right) + C \varpi_{n} + a \\
	&\le_{(iii)} \Prob\left( \overline{Z} \le t \right) + C a^{-1} \E[|R_{n}|_{\infty}] {\log^{1/2}{d}} + C \varpi_{n} + a \\
	&\le_{(iv)} \Prob\left( \overline{Z} \le t \right) + C  {D_{n} a^{-1} n^{-1/2} \log^{3/2}{d} } + C \varpi_{n} + a, 
	\end{align*}
	where step $(i)$ follows from Markov's inequality, step $(ii)$ from the Gaussian approximation error bound \eqref{eqn:gaussian_approx_linear_part} for the linear part, step $(iii)$ from Nazarov's inequality (cf. Lemma A.1 in \cite{cck2016a}), and step $(iv)$ from the maximal inequality \eqref{eqn:gaussian_approx_remainder_part} for the degenerate term. Likewise, we can deduce the reverse inequality 
	\[
	\setlength\abovedisplayskip{2pt}
	\setlength\belowdisplayskip{2pt}
	\Prob\left( \overline{T}_n \le t \right) \ge \Prob\left( \overline{Z} \le t \right) - C  {D_{n} a^{-1} n^{-1/2} \log^{3/2}{d} } - C \varpi_{n} - a. 
	\]
	Choosing $a = n^{-1/4} D_{n}^{1/2} \log^{3/4}{d}$, we get $\rho(\overline{T}_n,\overline{Z}) \le C \varpi_{n}$.
	
	\underline{\it Step 2. Bootstrap approximation to $T_{n}$.} 
	Recall the definition of $T_n^\sharp$ in (\ref{eqn:Tn_sharp}), $T_n^\sharp | X_1^n \sim N(\vzero, 4\hat{\Gamma}_n)$ where
	\begin{equation}
	\label{eqn:cov_Tn_sharp_Gamma}
	\setlength\abovedisplayskip{0pt}
	\setlength\belowdisplayskip{2pt}
	\hat{\Gamma}_n = {1 \over n(n-1)^2} \sum_{i=1}^n \sum_{j=i+1}^n \sum_{k=i+1}^n h(X_i, X_j) h(X_i, X_k)^T.
	\end{equation}
	By Lemma \ref{lem:cov_Tn_sharp},
	$\Prob \left(  |\hat{\Gamma}_{n} - \Gamma/3|_\infty  \ge K_3 \left\{ D_n^2 \log (nd) \over n \right\}^{1/2}  \right)
	\le  \gamma.$
	Therefore, \cite[Lemma C.1]{chen2018gaussian} confirms that with probability greater than $1-\gamma$
	\begin{align*}
	\rho(\overline{Z}, \overline{T}_n^\sharp | X_1^n) \lesssim \left[ |4\hat{\Gamma}_{n} - 4 \Gamma/3|_\infty \log^2(nd) \right]^{1/3}  \asymp \left\{ D_n^2 \log^5 (nd) \over n \right\}^{1/6} \lesssim \varpi_n .
	\end{align*}
	In conclusion, $\rho(\overline{T}_n, \overline{T}_n^\sharp | X_1^n) \le \rho(\overline{T}_n, \overline{Z}) + \rho(\overline{Z}, \overline{T}_n^\sharp \mid X_1^n) \le  C(\ub, K)  \varpi_n$.
\end{proof}

{ 
	\begin{proof}[Proof of Theorem~\ref{thm:gaussian_approx_rate_improve}]
		This proof is similar to the proof of Theorem~\ref{thm:gaussian_approx_rate} so that only the key steps are given below. Without loss of generality, we may assume $\varpi'_n \leq 1$.
		
		\underline{\it Step 1. Gaussian approximation to $T_{n}$.} 
		Let $\Gamma = \Cov(g(X_{1}))$, then $\Cov(L_n)  = {4(n+1) \over 3(n-1)} \Gamma$. So the correlation matrix of $L_n$ is the same as the correlation matrix of $g(X_1)$. By (A1), $\sigma_j^2 = \Cov_{jj}(L_n) \ge \ub$. By Jensen's inequality, under (A2), we have $\E|g_{j}(X_{i})|^{2+k} \le D_{n}^{k}$ for $k = 1,2$, whereas under (A3'), we have $\|g_{j}(X_{i})\|_{\psi_{2}} \le D_{n}$. Therefore, 
		\begin{align*}
		&{1\over n} \sum_{i=1}^n \E \left| {2 (n-2i+1) \over (n-1) \sigma_j} g_j(X_i) \right|^4 \le {2 \over n} \sum_{i=1}^n \E \left| g_j(X_i) \right|^4 / \sigma_j^4 \le 2 D_n^2/\ub^2 \le B_n^2,\\
		&\left\Arrowvert {2 (n-2i+1) \over (n-1) \sigma_j} g_j(X_i) \right\Arrowvert_{\psi_2} \le 2	|| g_j(X_i) ||_{\psi_2} / \sigma_j \le 2 D_n/\ub^{1/2} \le B_n,
		\end{align*}
		where $B_n = 2 D_n (\ub^{-1} \vee \ub^{-1/2})$. By \cite[Corollary 2.1]{cck2020}, the Conditions (M) and (E.2) are satisfied so that
		\begin{align}
		\label{eqn:Ln_Z_n_approx_improve}
		\rho(\overline{L}_n,\overline{Z}_{n}) \le K_1 \left( {B_n (\log d)^{3/2} (\log n) \over n^{1/2} \sigma_*^2} \vee {B_n (\log d)^{2} \over  n^{1/2} \sigma_*} \right) \le C_1(\ub)(\sigma_*^{-2} \vee \sigma_*^{-1}) \varpi'_n,
		\end{align}
		where $\overline{Z}_{n} = \max_{1 \le j \le d} Z_{nj}$ and $Z_{n} \sim N(0, {4(n+1) \over 3(n-1)} \Gamma)$. 
		Let $Z \sim N(0, 4\Gamma/3)$. We still have
		\[
		\rho(\overline{Z}_n,\overline{Z}) \le  C_2(\ub) {D_{n}^{1/3} \log^{2/3}{d} \over n^{1/3}} \le C_2(\ub) \varpi'_n. 
		\]
		Hence, by triangle inequality, $\rho(\overline{L}_n,\overline{Z}) \le \rho(\overline{L}_n,\overline{Z}_{n}) + \rho(\overline{Z}_n,\overline{Z}) \le C_3 (\ub, \sigma_*) \varpi'_{n}$, where $C_3 (\ub, \sigma_*) \le  C_1(\ub) (\sigma_*^{-2} \vee \sigma_*^{-1}) +  C_2(\ub)$.
		
		Note that, by choosing $a = n^{-1/4} D_{n}^{1/2} \log^{3/4}{d}$, the following approximation still holds
		\begin{equation*}
		\rho(\overline{T}_n,\overline{Z}) \le C_3 {D_{n} a^{-1} n^{-1/2} \log^{3/2}{d} } + C_3 \varpi'_{n} + a \le (2 C_3+1) \varpi'_n.
		\end{equation*}
		
		\underline{\it Step 2. Bootstrap approximation to $T_{n}$.} 
		Recall $T_n^\sharp | X_1^n \sim N(\vzero, 4\hat{\Gamma}_n)$. 
		By \cite[Lemma 2.1]{cck2020},
		\begin{align*}
		\rho(\overline{Z}, \overline{T}_n^\sharp | X_1^n) \le K_2  {V \log d \over \ub^2 \sigma_*^2} \left( 1 \vee |\log {V \over \ub^2 \sigma_*^2}| \right),
		\end{align*}
		where
		\[
		\Prob \left(  V = |\hat{\Gamma}_{n} - \Gamma/3|_\infty  \le K_3 D_n n^{-1/2}  \log^{1/2} (nd)   \right)	\ge  1-\gamma
		\]
		by Lemma~\ref{lem:cov_Tn_sharp}.
		Without loss of generality, we assume $\varpi'_{n} \le 1$. Then, with probability greater than $1-\gamma$, ${V} \le K_3 n^{-1/4} \log^{-1}n \log^{-1/2} d \le K_3 n^{-1/4}$.
		If $V \ub^{-2}\sigma_*^{-2} \ge 1$, then 
		$$
		|\log {V \over \ub^2\sigma_*^2}| \le C_4(\ub, \sigma_*)\log(n), 
		$$ 
		so that $\rho(\overline{Z}, \overline{T}_n^\sharp | X_1^n) \le C_5(\ub,\sigma_*) V (\log d) (\log n)$ and therefore (\ref{eqn:gaussian_approx_rate_improve}) holds.
		If $V \ub^{-2}\sigma_*^{-2}  < 1$, then observing that the function  $f(x) = x|\log x|\le e^{-1}(1-t)^{-1}x^{t}$ for any $0<t<1$ on $x \in (0, 1)$, we have 
		$$
		\rho(\overline{Z}, \overline{T}_n^\sharp | X_1^n) \le K_2 (\log d) (V \ub^{-2}\sigma_*^{-2})^t.
		$$
		Taking $t=1/2$ and plugging in $V \le K_3 D_n n^{-1/2}  \log^{1/2} (nd)$, we still have (\ref{eqn:gaussian_approx_rate_improve}) holds with probability greater than $1-\gamma$. 
	\end{proof}
}

\begin{proof}[Proof of Theorem~\ref{thm:power_signal_rate}]
	Denote $T_n = T_n(X_1^n) = {n}^{1/2} {n \choose 2}^{-1} \sum_{1 \le i < j \le n} h(X_i, X_j)$ and $T_n^\xi = T_n(\xi_1^n) = {n}^{1/2} {n \choose 2}^{-1} \sum_{1 \le i < j \le n} h(\xi_i, \xi_j)$. Define
	\[
	\setlength\abovedisplayskip{0pt}
	\setlength\belowdisplayskip{1pt}
	\tilde{\Delta} = n^{-1/2} {n \choose 2} \{ T_n(X_1^n) - T_n(\xi_1^n) \} = \sum_{1 \le i < j \le n} h(X_i, X_j) - h(\xi_i, \xi_j).
	\]
	Note that, $\overline{T}_n^\xi = |T_n(\xi_1^n)|_\infty \ge 2 n^{-1/2} (n-1)^{-1}|\tilde{\Delta}|_\infty -	\overline{T}_n$.
	It follows that
	
	\begin{align*}
	\textrm{Type II error} &= \Prob \left( \overline{T}_n \le q_{\overline{T}_n^\sharp \mid X_{1}^{n}} (1-\alpha) \mid H_1 \right) \\
	&\le \Prob \left( \overline{T}_n^\xi \ge 2 n^{-1/2} (n-1)^{-1} |\tilde{\Delta}|_\infty - q_{\overline{T}_n^\sharp \mid X_{1}^{n}} (1-\alpha) \mid H_1 \right) \\
	&\le \Prob \left(\overline{T}_n^\xi \ge q_{\overline{T}_n^\xi} (1-\beta_n)  \mid H_1\right) \\
	& \qquad + \Prob \left ( q_{\overline{T}_n^\sharp \mid X_{1}^{n}} (1-\alpha) + q_{\overline{T}_n^\xi} (1-\beta_n) \ge 2 n^{-1/2} (n-1)^{-1} |\tilde{\Delta}|_\infty  \mid H_1\right) \\
	&\le \beta_n + \Prob \left ( q_{\overline{T}_n^\sharp \mid X_{1}^{n}} (1-\alpha) + q_{\overline{T}_n^\xi} (1-\beta_n) \ge 2 n^{-3/2} |\tilde{\Delta}|_\infty  \mid H_1\right).
	\end{align*}
	Let $\gamma = \zeta/8$. Now denote
	\begin{align*}\hspace{40pt}
	\Delta_1 &=  \gamma^{-1} D_n \log (d) \{m(n-m)\}^{1/2} , \\
	\Delta_2 &=  D_n  \{m(n-m)\}^{1/2} \{m \wedge (n-m)\}^{1/2} \log^{1/2}(nd), \\
	\Delta_3 &=   D_n n^{3/2} \log^{1/2} (nd / \alpha),\\
	\Delta_4 &=  n^{3/2}  \log^{1/2} (\gamma^{-1}) \log^{1/2} (d).
	\end{align*}
	We will quantify  $|\tilde{\Delta}|_\infty$, $q_{\overline{T}_n^\sharp} (1-\alpha)$ and $q_{\overline{T}_n^\xi} (1-\beta_n)$ to conclude that the Type II error is bounded when $|\theta_h|_\infty$ satisfies (\ref{eqn:signal_rate}). 
	
	\emph{(1) Quantify $|\tilde{\Delta}|_\infty$.} 	Without loss of generality, we may assume $n_1 = m \le n- m = n_2$. 
	Recall (\ref{eqn:hoeffding_decomp_two-sample}) where $V_n = V_n(X_1^n)$.
	Denote $V_n(\xi_1^n)$ in similar way. By shift-invariant assumption and the two-sample projection in Section \ref{sec:bootstrap},
	\begin{eqnarray}
	\tilde{\Delta} &=& V_n(X_1^n) - V_n(\xi_1^n) = \sum_{i=1}^{n_1} \sum_{j=1}^{n_2} h(X_i, Y_j) - h(X_i, Y_j - \theta) \nonumber   \\[-2mm]	
	&=& \sum_{i=1}^{n_1} \sum_{j=1}^{n_2} g(Y_j - \theta) -g(Y_j) + \breve{f}(X_i,Y_j) -\breve{f}(X_i,Y_j-\theta) \nonumber \\[-2mm]
	&=& n_1 n_2 \theta_h + n_1  \sum_{j=1}^{n_2} \{-g(Y_j) - \theta_h\} + n_1  \sum_{j=1}^{n_2} g(Y_j - \theta) \nonumber\\
	&& \qquad\quad +  \sum_{i=1}^{n_1} \sum_{j=1}^{n_2} \breve{f}(X_i,Y_j) - \sum_{i=1}^{n_1} \sum_{j=1}^{n_2} \breve{f}(X_i, Y_j-\theta) .
	\end{eqnarray}
	By Lemma \ref{lem:rate_proj_1sample}, with probability smaller than $\gamma$,
	\[
	n_1 | \sum_{j=1}^{n_2}  [-g(Y_j) - \theta_h]  |_\infty \ge K_1 D_n n_1 n_2^{1/2} \log^{1/2}(nd) = K_1  \Delta_2.
	\]
	Similarly, $n_1 | \sum_{j=1}^{n_2}  g(Y_j - \theta) |_\infty \ge K_2   \Delta_2$ with probability smaller than $\gamma$.
	By Lemma \ref{lem:E_2sample_inf_Ustat}, 
	\[
	\E \big| \sum_{i=1}^{n_1} \sum_{j=1}^{n_2} \breve{f}(X_i,Y_j) \big|_\infty \le K_3  \Delta_1 \gamma.
	\]
	From Markov inequality, $\Prob \left( |\sum_{i=1}^{n_1} \sum_{j=1}^{n_2} \breve{f}(X_i,Y_j)|_\infty \ge K_3  \Delta_1  \right) \le \gamma$.\\ 
	Similarly, $| \sum_{i=1}^{n_1} \sum_{j=1}^{n_2} \breve{f} (X_i, Y_j-\theta) |_\infty \ge K_4  \Delta_1 $ with probability smaller than $\gamma$. 
	Therefore, 
	\begin{align*}
	|\tilde{\Delta}|_\infty &\ge n_1 n_2 |\theta_h|_\infty - |n_1  \sum_{j=1}^{n_2} [-g(Y_j) - \theta_h] |_\infty - | n_1  \sum_{j=1}^{n_2} g(Y_j - \theta) |_\infty \\
	&\qquad \qquad \qquad \ - |\sum_{i=1}^{n_1} \sum_{j=1}^{n_2} \breve{f}(X_i,Y_j)|_\infty - |\sum_{i=1}^{n_1} \sum_{j=1}^{n_2} \breve{f} (X_i, Y_j-\theta)|_\infty\\
	& \ge n_1 n_2 |\theta_h|_\infty - (K_1 + K_2)  \Delta_2 - (K_3+ K_4)   \Delta_1
	\end{align*}
	with probability no smaller than $1-4\gamma$. 
	
	\emph{(2) Bound $q_{\overline{T}_n^\sharp} (1-\alpha)$.	 }
	Recall $T_n^\sharp |X_1^n \sim N_d(\vzero, 4 \hat{\Gamma}_n)$, where $\hat{\Gamma}_n$ is defined in (\ref{eqn:cov_Tn_sharp_Gamma}).
	By the Bonferroni inequality, $\Prob \left( \overline{T}_n^\sharp > t |X_1^n \right) \le 2d \left[1-\Phi(t/ 2\overline{\psi}) \right]$, where $\overline{\psi}^2 = \max_{1 \le l \le d} \hat{\Gamma}_{n,ll} $. 	By the Cauchy-Schwarz inequality, for each $l=1,\dots,d$,
	\begin{align*}\left\{ \sum_{i<j,k} h_l(X_i, X_j) h_l(X_i, X_k) \right\}^2 &\le 	\left\{ \sum_{i<j,k} h_l^2(X_i, X_j)  \right\} \left\{ \sum_{i<j,k} h_l^2(X_i, X_k)  \right\}\\
	&= \left\{ \sum_{i<j,k} h_l^2(X_i, X_j)  \right\}^2,
	\end{align*} 
	which implies 
	\[
	\hat{\Gamma}_{n,ll} \le n^{-1} (n-1)^{-2}  \sum_{i=1}^n \sum_{i<j} (n-i) h_l^2(X_i, X_j) \le  (n-1)^{-2}  \sum_{i=1}^n \sum_{i<j} h_l^2(X_i, X_j).
	\]
	By Condition [A2] and [B2], $\E h_l^2(X_i, X_j) \le \E |h_l(X_i, X_j) - \E h_l(X_i, X_j)|^2 +  |\E h_l(X_i, X_j)|^2 \le  D_n +  |\theta_h|^2_\infty \ \vone (1 \le i \le m < j \le n)$ for any $1 \le l \le d$ and $1 \le i < j \le n$.
	From Lemma \ref{lem:cov_Tn_sharp_H1}, it shows that with probability grater than $1-\gamma$, 
	\begin{align*}
	\overline{\psi}^2 &\le (n-1)^{-2} \Big\{ t^\diamond + \max_{1\le l \le d} \sum_{i=1}^n \sum_{i<j} \E h_l^2(X_i, X_j) \Big\} \\[-.5mm]
	& \lesssim D_n^2 + |\theta_h|^2_\infty \underbrace{ n^{-2}  \{ n_1 n_2 + n_1^{1 / 2} n_2 \log^{1 / 2}(nd) + n_2 \log^{3}(nd) \log (\gamma^{-1}) \} }_{\delta_n} .
	\end{align*}
	Therefore, $\overline{\psi} \le K_5 \left[ D_n + |\theta_h|_\infty \delta_n^{1/2} \right] $.
	In addition, for $\Phi^{-1} (1-{\alpha / (2d)}) = t_\alpha >0$ (as $d>1$), Gaussian tail bound (Chernoff method) shows $t_\alpha \le \left[ {2 \log (2d / \alpha)} \right]^{1/2}$. Then, with probability greater than $1-\gamma$,
	\[
	q_{\overline{T}_n^\sharp} (1-\alpha) \le 2 \overline{\psi} \Phi^{-1} (1-{\alpha / (2d)}) \le K_6 n^{-3/2}  \left(\Delta_3 + |\theta_h|_\infty \left\{ n^3 \log({2d / \alpha}) \delta_n \right\}^{1/2} \right).
	\]
	Since $n_2 \ge n/2$ and $n_1 \gtrsim \log^{5/2} (nd)$, the rate of $\left\{ n^3 \log({2d / \alpha}) \delta_n \right\}^{1/2} \lesssim n_1 n_2$ leads to $q_{\overline{T}_n^\sharp | X_1^n} (1-\alpha) \le K_6 n^{-3/2} ( \Delta_3 + n_1 n_2 |\theta_h|_\infty )$.
	For bounded kernel $h$, a simpler bound of $\overline{\psi} \le K_5  D_n$ directly lead to $q_{\overline{T}_n^\sharp | X_1^n} (1-\alpha) \le K_6 n^{-3/2}  \Delta_3$ without assuming $n_1 \gtrsim \log^{5/2} (nd)$.
	
	\emph{(3) Bound $q_{\overline{T}_n^\xi} (1-\beta_n)$.}
	Note that $\overline{T}_n^\xi$ has the same distribution as $\overline{T}_n | H_0$. By the approximation in Theorem \ref{thm:gaussian_approx_rate} Step1, we have $\rho( \overline{T}_n^\xi, \overline{Z}) \le C_1 \varpi_n$ holds for $Z \sim N_d (0, 4\Gamma/3)$ with probability grater than $1-\gamma$.
	Since $||\overline{Z}||_{\psi_2} \le C_2(\ub) \log^{1/2} (d)$ by \cite[Lemma 2.2.2]{vandervaartwellner1996} and $\Prob(\overline{Z} > t) \le 2 \exp \left\{ -({t \over ||\overline{Z}||_{\psi_2}})^2 \right\} \le 2 \exp \left\{ -C_2(\ub) ^{-2} \log^{-1}(d) t^2 \right\}$. 
	Choosing $t = C_3(\ub) \log^{1/2} (\gamma^{-1}) \log^{1/2} (d)$ for large enough $C_3(\ub)$, we have  $\Prob(\overline{Z} > t) \le 2\gamma$. Hence, $\Prob (\overline{T}_n^\xi > t) \le \Prob (\overline{Z} > t) +  C_1 \varpi_n$. Let $\beta_n =  2\gamma +  C_1 \varpi_n$. Then with probability grater than $1-\gamma$, 
	\begin{equation*}
	q_{\overline{T}_n^\xi} (1-\beta_n) \le  C_3(\ub) \log^{1/2} (\gamma^{-1}) \log^{1/2} (d) = C_3(\ub) n^{-3/2} \Delta_4.
	\end{equation*}
	
	Combining Step (1)-(3), when $m(n-m) |\theta_h|_\infty > 2 (K_3+K_4)\Delta_1 + 2(K_1+K_2)\Delta_2 + K_6\Delta_3 + C_3(\ub) \Delta_4$, 
	\[
	|\tilde{\Delta}|_\infty \ge {1 \over 2} n^{3/2}  \left\{q_{\overline{T}_n^\sharp} (1-\alpha)  + q_{\overline{T}_n^\xi} (1-\beta_n) \right\}
	\] 
	with probability no smaller than $1- 6 \gamma$.
	That is, the Type II error is less than $6 \gamma + \beta_n = 8 \gamma +  C_1 \varpi_n$, where we set $\zeta = 8\gamma$.
	As $(\Delta_1 \vee \Delta_2) \lesssim \Delta_3$, the conclusion of Theorem \ref{thm:power_signal_rate} immediately follows for some large enough $K \ge 2 \sum_{i=1}^6 K_i$.
	
\end{proof}

\begin{proof}[Proof of Lemma~\ref{lem:power_signal_rate_multiple}]
	Let 
	\[
	\tilde{\Delta} = \sum_{1 \le i < j \le n} h(X_i, X_j) - h(\xi_i, \xi_j) = \sum_{ k < k'} \tilde{\Delta}^{(k,k')},  
	\]
	where
	\[
	\tilde{\Delta}^{(k,k')} = \sum_{\begin{subarray}{c} m_k < i \le m_{k+1} \\ m_{k'} < j \le m_{k'+1} \end{subarray}} h(X_i, X_j) - h(\xi_i, \xi_j).
	\]
	Similar to the proof of Theorem~\ref{thm:power_signal_rate}, we shall quantify $|\tilde{\Delta}|_\infty$, $q_{\overline{T}_n^\sharp} (1-\alpha)$ and $q_{\overline{T}_n^\xi} (1-\beta_n)$ to conclude that the Type II error is bounded when $|\delta|_\infty$ satisfies (\ref{eqn:signal_rate_multiple}).
	
	\emph{(1) Quantify $|\tilde{\Delta}|_\infty$.}  
	\begin{eqnarray}
	\tilde{\Delta}^{(k,k')} = s_{k} s_{k'} \delta^{(k,k')} +&& s_{k}  \sum_{\mathclap {m_{k'} < j \le m_{k'+1}}}  \{-g(X_j - \theta^{(k)}) - \delta^{(k,k')} \} +  s_{k}  \sum_{\mathclap {m_{k'} < j \le m_{k'+1}}}  g(X_j - (\theta^{(k')} - \theta^{(k)}) ) \nonumber\\
	+&&  \sum_{\mathclap {\begin{subarray}{c} m_k < i \le m_{k+1} \\ m_{k'} < j \le m_{k'+1} \end{subarray}}} \breve{f}(X_i,X_j) - \sum_{\mathclap {\begin{subarray}{c} m_k < i \le m_{k+1} \\ m_{k'} < j \le m_{k'+1} \end{subarray}}} \breve{f}(X_i, X_j- \theta^{(k)}) \nonumber.
	\end{eqnarray}
	Applying the results in Step (1) to $\sum_{ k < k'} \tilde{\Delta}^{(k,k')}$, we have each of the following inequalities satisfied with probability greater than $1 - \gamma$:
	\begin{align*}
	&|\sum_{ k < k'} \quad s_{k}  \sum_{\mathclap {m_{k'} < j \le m_{k'+1}}}  \{-g(X_j - \theta^{(k)}) - \delta^{(k,k')} \}|_\infty \\
	&\quad \le \sum_{ k < k'}  K_1 D_n (s_k s_{k'})^{1/2} n^{1/2} \log^{1/2}(nd)  \le K_1 \nu^2 D_n n^{3/2} \log^{1/2}(nd);\\
	&|\sum_{ k < k'} \quad s_{k}  \sum_{\mathclap {m_{k'} < j \le m_{k'+1}}}  g(X_j - (\theta^{(k')} - \theta^{(k)}) )|_\infty \\
	&\quad \le \sum_{ k < k'} K_2 D_n (s_k s_{k'})^{1/2} n^{1/2} \log^{1/2}(nd) \le K_2 \nu^2 D_n n^{3/2} \log^{1/2}(nd);\\
	&|\sum_{ k < k'} \quad\quad \sum_{\mathclap {\begin{subarray}{c} m_k < i \le m_{k+1} \\ m_{k'} < j \le m_{k'+1} \end{subarray}}} \breve{f}(X_i,X_j)|_\infty + |\sum_{ k < k'} \quad\quad \sum_{\mathclap {\begin{subarray}{c} m_k < i \le m_{k+1} \\ m_{k'} < j \le m_{k'+1} \end{subarray}}} \breve{f}(X_i, X_j- \theta^{(k)})|_\infty  \\
	&\quad \le \sum_{ k < k'} K_3 \gamma^{-1} D_n (s_k s_{k'})^{1/2} \log d \le K_3 \nu^2 D_n n^{3/2} \log^{1/2}(nd).
	\end{align*}
	Combining all pairs of $(k,k')$ for $0\le k < k' \le \nu$, it follows
	\begin{align*}
	|\tilde{\Delta}|_\infty = |\sum_{ k < k'} \tilde{\Delta}^{(k,k')}|_\infty &\ge |\sum_{ k < k'}  s_{k} s_{k'} \delta^{(k,k')}|_\infty  - (K_1 + K_2+K_3)  \nu^2 D_n n^{3/2} \log^{1/2}(nd)
	\end{align*}
	with probability greater than $1 - 3 \gamma$.
	
	\emph{(2) Bound $q_{\overline{T}_n^\sharp} (1-\alpha)$.}
	Under $H_1^{'}$, $T_n^\sharp |X_1^n \sim N_d(\vzero, 4 \hat{\Gamma}_n)$, where $\hat{\Gamma}_n$ is defined the same as in (\ref{eqn:cov_Tn_sharp_Gamma}).
	To control the magnitude of $|\sum_{1 \le i<j \le n} h_l^2(X_i, X_j)|$, note that 
	\[
	\sum_{1 \le i<j \le n} = \sum_{ {\begin{subarray}{c} m_k < i \le m_{k+1} \\ m_{k'} < j \le m_{k'+1}\\0 \le k<k' \le \nu \end{subarray}}} + \sum_{ {\begin{subarray}{c} m_k < i < j \le m_{k+1} \\ 0\le k \le \nu \end{subarray}}}.
	\]
	So we can modify Lemma~\ref{lem:cov_Tn_sharp_H1} from the following two cases.
	For the case of $\calC_{k,k'} = \{ m_k < i \le m_{k+1} \le m_{k'} < j \le m_{k'+1} \}$ where $i,j$ are in different segments, $\E h_l^2 (X_i, X_j) \le D_n + |\delta^{(k,k')}_l|^2 $, based on modified Lemma~\ref{lem:cov_Tn_sharp_H1} we have 
	\begin{align*}
	\Prob\Big( &\max_{1\le l \le d} |\sum_{\calC_{k,k'}}  h_l^2 (X_i, X_j)  -  \E h_l^2 (X_i, X_j)|  \ge \\
	&\qquad\qquad \max_{k<k'} K_4 (D_n^2 + |\delta^{(k,k')}|_\infty^2) (s_k s_{k'})^{1/2} n^{1/2} \log^{1/2} (n d) \Big) \le \gamma.
	\end{align*}		
	For the case of $\calC_{k} = \{ m_k < i  < j \le m_{k+1} \}$ where $i,j$ are in the same segments, $|\E h_l(X_i, X_j)|^2 \le D_n$ and
	\begin{align*}
	\Prob\left( \max_{1\le l \le d} |\sum_{\calC_{k}}  h_l^2 (X_i, X_j)  -  \E h_l^2 (X_i, X_j)| \ge K_5 D_n^2 n^{3/2}  \log^{1/2} (n d) \right) \le \gamma.
	\end{align*}
	Take  $t^\diamond = D_n^2 n^{3/2}  \log^{1/2} (nd) + \max_{k<k'} (s_k s_{k'})^{1/2} |\delta^{(k,k')}|_\infty^2 n^{1/2}  \log^{1/2} (nd) $. Then, adding all $\calC_{k}$ and $\calC_{k,k'}$ together, 
	\begin{align*}
	\overline{\psi}^2 &= \max_{1 \le l \le d} \hat{\Gamma}_{n,ll} \\
	&\le (n-1)^{-2} K_6 \Big\{ t^\diamond + \max_{1\le l \le d} \sum_{i=1}^n \sum_{i<j} \E h_l^2(X_i, X_j) \Big\} \\[-.5mm]
	& \le K_6 \left\{D_n^2 + n^{-3/2} \log^{1/2}(nd) \max_{k<k'} (s_k s_{k'})^{1/2} |\delta^{(k,k')}|_\infty^2 + n^{-2} \sum_{ k < k'} s_k s_{k'} |\delta^{(k,k')}|_\infty^2 \right\} 
	\end{align*}
	holds with probability greater than $1-(\nu+1)( \nu+2)\gamma/2$.
	Therefore, $q_{\overline{T}_n^\sharp} (1-\alpha) \le K_7 \overline{\psi} t_\alpha$, where $t_\alpha = \Phi^{-1} (1-{\alpha / (2d)}) \le 2\log^{1/2} (nd / \alpha)$ and
	\[
	\overline{\psi} \le K_6 \left\{ D_n +  n^{-3/4} \log^{1/4}(nd) \max_{k<k'} (s_k s_{k'})^{1/4} |\delta^{(k,k')}|_\infty +   n^{-1} \sum_{ k < k'} (s_k s_{k'})^{1/2} |\delta^{(k,k')}|_\infty  \right\}.
	\] 
	
	\emph{(3) Bound $q_{\overline{T}_n^\xi} (1-\beta_n)$.}
	Since $\overline{T}_n^\xi$ does not depend on $H_1^{'}$, it obeys the same bound  
	\begin{equation*}
	q_{\overline{T}_n^\xi} (1-\beta_n) \le  C(\ub) \log^{1/2} (\gamma^{-1}) \log^{1/2} (d) = C(\ub) \log^{1/2}(\gamma^{-1}) \log^{1/2}(d)
	\end{equation*}
	with probability grater than $1-\gamma$  for $\beta_n =  2\gamma +  C_1 \varpi_n$.
	
	Combining Step (1)-(3), when 
	\begin{align*}
	&|\sum_{ k < k'}  s_{k} s_{k'} \delta^{(k,k')}|_\infty  > K_0 \nu^2 D_n n^{3/2} \log^{1/2}(nd/\alpha) + C(\ub) n^{3/2} \log^{1/2}(\gamma^{-1}) \log^{1/2}(d)  \\
	&+ K_0' \log^{1/2}(nd/\alpha) \left\{	n^{3/4} \log^{1/4}(nd) \max_{k<k'} (s_k s_{k'})^{1/4} |\delta^{(k,k')}|_\infty +   n^{1/2}  \sum_{ k < k'} (s_k s_{k'})^{1/2} |\delta^{(k,k')}|_\infty  \right\},
	\end{align*}
	the Type II error will be smaller than $\beta_n + \{4+(\nu+1)(\nu+2)/2\}\gamma$ for $\beta_n =  2\gamma +  C_1 \varpi_n$. Substitute $\gamma$ by $\{4+(\nu+1)(\nu+2)/2\}^{-1} \zeta$, we reach the conclusion of theorem.  
\end{proof}

\subsection{Proof of lemmas in theorems}

\begin{lem}[Bounding $|\hat{\Gamma}_n - \Gamma/3|_\infty$ under $H_0$.]
	\label{lem:cov_Tn_sharp}
	Suppose all the conditions in Theorem \ref{thm:gaussian_approx_rate} hold. Let $\Gamma = \Cov(g(X_1))$ and $\hat{\Gamma}_n$ be defined as in (\ref{eqn:cov_Tn_sharp_Gamma}). Then  with probability greater than $1-\gamma$,
	\begin{align*}
	|\hat{\Gamma}_{n} - \Gamma/3|_\infty  \le K_0 \left(D_n^2 \log (nd) \over n \right)^{1/2}.
	\end{align*}
\end{lem}
\begin{proof}[Proof of Lemma \ref{lem:cov_Tn_sharp}]
	Note $\Gamma = \Cov( \E [ h(X,X_1)|X] ) = \E[ h(X_1, X_2) h(X_1,X_3)^T ]$ and let $\Gamma_2 = \E[ h(X_1, X_2) h(X_1,X_2)^T ]$. Then
	\begin{align*}
	\E \hat{\Gamma}_n & = {1 \over n(n-1)^2} \sum_{i=1}^n (n-i)(n-i-1) \Gamma +  {1 \over n(n-1)^2} \sum_{i=1}^n (n-i) \Gamma_2 \\
	&= {n-2 \over 3 (n-1)} \Gamma + {1 \over 2 (n-1)} \Gamma_2 	.
	\end{align*} 
	Note that, the summation in $\hat{\Gamma}_n$ can split into two parts
	\begin{align*}
	\sum_{i=1}^n \sum_{j,k>i}  = \sum_{i=1}^n \sum_{j \neq k >i} + \sum_{i=1}^n \sum_{j = k >i}.
	\end{align*}
	In Steps 1 and 2 below, we will deal with $\hat{\Gamma}_{n1} = {1 \over n(n-1)^2}  \sum_{i=1}^n \sum_{j \neq k >i} h(X_i, X_j) h(X_i, X_k)^T$ and  $\hat{\Gamma}_{n2} = {1 \over n(n-1)^2}  \sum_{i=1}^n \sum_{j = k >i} h(X_i, X_j) h(X_i, X_k)^T$ respectively, where $\hat{\Gamma}_n= \hat{\Gamma}_{n1} + \hat{\Gamma}_{n2}$. 
	Then conclusion will be made in Step 3.
	
	\underline{Step 1: Term $\hat{\Gamma}_{n1} = {1 \over n(n-1)^2}  \sum_{i=1}^n \sum_{j \neq k >i} h(X_i, X_j) h(X_i, X_k)^T$.}
	Define $H( x_1, x_2, x_3) $ to be $ h(x_1, x_2) h(x_1,x_3)^T$. To symmetrize $H$, let $H' ( X_i, X_j, X_k) = \sum_{\pi_3} \tilde{H} ( X_{\pi_3(i)}, X_{\pi_3(j)}, X_{\pi_3(k)})$, where 
	\begin{displaymath}
	\tilde{H}( X_i, X_j, X_k) = \left\{ \begin{array}{ll}
	H( X_i, X_j, X_k), & \textrm{ if $i < j \ne k$,}\\
	\vzero,  & \textrm{ otherwise}
	\end{array} \right. ,
	\end{displaymath}
	and $\pi_3$ is a permutation of $\{i,j,k\}$. Then,
	\begin{align*}
	\hat{\Gamma}_{n1} &= {1 \over n(n-1)^2} \sum_{i < j \neq k} H( X_i, X_j, X_k) = {1 \over n(n-1)^2} \sum_{i \ne j \ne k} \tilde{H}( X_i, X_j, X_k) \\
	&= {1 \over 6n(n-1)^2} \sum_{i \ne j \ne k} H' ( X_i, X_j, X_k) 
	\end{align*}
	is a $U$-statistics of order 3 and $\E \hat{\Gamma}_{n1} = {n-2 \over 3(n-1)}\Gamma$. 
	Let 
	\[
	W_n = {(n-3)! \over n!} \sum_{i \ne j \ne k} H' ( X_i, X_j, X_k) = {6(n-1) \over n-2} \hat{\Gamma}_{n1} .
	\]
	Apply Lemma E.1 in \cite{chen2018gaussian} to $H'$ for $\alpha = 1/2, \eta = 1$ and $\delta = 1/2$, 
	\begin{equation}
	\label{eqn:lem_maxima_centered_U}
	\Prob \left( {n \over 3} |W_n - \E W_n|_\infty \ge 2 \E Z_1 +t \right) 
	\le \exp\left(- {t^2 \over 3 \overline{\zeta}_n^2} \right) + 3 \exp \left[- \left( {t \over K_1 ||M||_{\psi_{1/2}} } \right)^{1/2} \right],
	\end{equation}
	where
	\begin{align*}
	&\E W_n = \E H'(X_1, X_2, X_3) = 2\Gamma,\\
	&Z_1 = \max_{1 \le m_1, m_2 \le d} \left| \sum_{i=0}^{[{n \over 3}] - 1} \left[\overline{H'}_{m_1, m_2} (X_{3i+1}^{3i+3}) - \E \overline{H'}_{m_1, m_2} \right] \right|,\\
	&\overline{\zeta}_n^2 =  \max_{1 \le m_1, m_2 \le d} \sum_{i=0}^{[{n \over 2}] - 1} \E H_{m_1, m_2}^{'2} (X_{3i+1}^{3i+3}), \\
	&M = \max_{1 \le m_1, m_2 \le d} \max_{0 \le i \le [{n \over 3}] - 1} \left| H'_{m_1, m_2} (X_{3i+1}^{3i+3}) \right|.
	\end{align*}	
	and $\overline{H'}_{m_1, m_2} (x_1, x_2, x_3) = H'_{m_1, m_2} (x_1, x_2, x_3) \vone_{ \{\max_{ m_1, m_2} |H'_{m_1, m_2} (x_1, x_2, X_3)| \le \tau \} }$ for $\tau = 8\E M$.
	By Cauchy-Schwarz and Condition (A2), 
	\begin{align*}
	\E H_{m_1, m_2}^{'2} (X_{3i+1}^{3i+3}) &\le 2\E H_{m_1, m_2}^{2} (X_{3i+1}^{3i+3}) \\
	&\le \left( \E h_{m1}^4 (X_{3i+1}, X_{3i+2}) \right)^{1/2} \left( \E h_{m2}^4 (X_{3i+1}, X_{3i+3}) \right)^{1/2} \le D_n^2.
	\end{align*}
	So $\overline{\zeta}_n \le n^{1/2} D_n$. From (i) \cite[Lemma 2.2.2]{vandervaartwellner1996}, (ii) the fact of $||X^2||_{\psi_{1/2}} = ||X||^2_{\psi_1}$ and (iii) Condition (A3), we obtain
	\begin{align*}
	||M||_{\psi_{1/2}} &= || \max_{ 1\le m_1, m_2 \le d} \max_{0 \le i \le {n \over 3}-1} h_{m_1} (X_{3i+1}, X_{3i+2}) h_{m_2} (X_{3i+1}, X_{3i+3}) ||_{\psi_{1/2}} \\
	&\le_{(i)} K_2 \log^2(nd) \max_{ 1\le m_1, m_2 \le d} \max_{0 \le i \le {n \over 3}-1} || h_{m_1} (X_{3i+1}, X_{3i+2}) h_{m_2} (X_{3i+1}, X_{3i+3}) ||_{\psi_{1/2}} \\
	&\le K_2' \log^2(nd) \max_{ 1\le m_1 \le d} \max_{0 \le i \le {n \over 3}-1} || h_{m_1}^2 (X_{3i+1}, X_{3i+2})||_{\psi_{1/2}} \\
	&=_{(ii)} K_2' \log^2(nd) \max_{ 1\le m_1 \le d} \max_{0 \le i \le {n \over 3}-1} || h_{m_1} (X_{3i+1}, X_{3i+2})||_{\psi_1}^2 \\
	&\le_{(iii)} K_2' \log^2(nd) D_n^2.
	\end{align*}
	By  \cite[Lemma 8]{cck2014b},
	\[
	\E Z_1 \le K_3 \left\{ \sqrt{\log d} \ \overline{\zeta}_n + \log d \ ||M||_{\psi_{1/2}} \right\} \le K_4 [ n \log(nd) D_n^2 ]^{1/2}.
	\]
	Therefore, (\ref{eqn:lem_maxima_centered_U}) leads to
	\begin{align*}
	\Prob \big( |\hat{\Gamma}_{n1} - \E \hat{\Gamma}_{n1}|_\infty \ge & 4 K_4  n^{-1/2} D_n \log^{1/2}(nd)+ t \big) 
	\\
	&\le  \exp\left(- {n t^2 \over 3 D_n^2} \right) + 3 \exp \left[- {\sqrt{nt} \over {K_1K_2}^{1/2} \log(nd) D_n }  \right].
	\end{align*}
	Recall $K \log(nd) \ge \log (1/\gamma) \ge 1$ and $n \gtrsim D_n^2 \log^7(nd)$.	Choose 
	\[
	t^* = K_5 \sqrt{D_n^2 \log(nd)  \over n}
	\]
	for some large enough $K_5>0$. Then,
	\[
	\Prob \left( |\hat{\Gamma}_{n1} - \E \hat{\Gamma}_{n1}|_\infty \ge t^* \right) 
	\le  \gamma^{K_5^2 \over 3K} + 3 \gamma^{K_5^{1/2} \over K K_1 K_2^{1/2}} \le \gamma/2.
	\]
	
	\underline{Step 2: Term $\hat{\Gamma}_{n2} = {1 \over n(n-1)^2}  \sum_{i=1}^n \sum_{j = k >i} h(X_i, X_j) h(X_i, X_k)^T$.}
	Let $H( x_1, x_2) $ be defined as $ h(x_1, x_2) h(x_1,x_2)^T$.  Denote $W_n' = {(n-2)! \over n!} \sum_{i \ne j} H(X_i, X_j) = 2(n-1) \hat{\Gamma}_{n2}$.
	By Lemma E.1 in \cite{chen2018gaussian}, 
	\[
	\Prob \left( {n \over 2} |W_n' - \E W_n'|_\infty \ge 2 \E Z'_1 +t \right) 
	\le \exp\left(- {t^2 \over 3 \overline{\zeta'}_n^2} \right) + 3 \exp \left[- \left( {t \over K_6 ||M'||_{\psi_{1/2}} } \right)^{1/2} \right]
	\]
	where
	\begin{align*}
	&\E W_n' = \E [H(X_1,X_2)] =  \Gamma_2,\\
	&Z'_1 = \max_{1 \le m_1, m_2 \le d} \left| \sum_{i=0}^{[{n \over 2}] - 1} \left[\overline{H}_{m_1, m_2} (X_{2i+1}^{2i+2}) - \E \overline{H}_{m_1, m_2} \right] \right|,\\
	&\overline{\zeta'}_n^2 =  \max_{1 \le m_1, m_2 \le d} \sum_{i=0}^{[{n \over 2}] - 1} \E H_{m_1, m_2}^2 (X_{2i+1}^{2i+2}), \\
	&M' = \max_{1 \le m_1, m_2 \le d} \max_{0 \le i \le [{n \over 2}] - 1} \left| H_{m_1, m_2} (X_{2i+1}^{2i+2}) \right|.
	\end{align*}	
	and $\overline{H}_{m_1, m_2} (x_1, x_2) = H_{m_1, m_2} (x_1, x_2) \vone_{ \{\max_{ m_1, m_2} |H_{m_1, m_2} (x_1, x_2)| \le \tau \} }$ for $\tau = 8\E M'$.
	Similarly, 
	\[
	\E H_{m_1, m_2}^2 (X_{2i+1}^{2i+2}) \le \left( \E h_{m1}^4 (X_{2i+1}^{2i+2}) \right)^{1/2} \left( \E h_{m2}^4 (X_{2i+1}^{2i+2}) \right)^{1/2} \le D_n^2.
	\]
	So $\overline{\zeta'}_n \le n^{1/2} D_n$. In addition,
	\begin{align*}
	||M'||_{\psi_{1/2}} &= || \max_{ 1\le m_1, m_2 \le d} \max_{0 \le i \le {n \over 2}-1} h_{m_1} (X_{2i+1}^{2i+2}) h_{m_2} (X_{2i+1}^{2i+2}) ||_{\psi_{1/2}} \\
	&\le K_7 \log^2(nd) \max_{ 1\le m_1 \le d} \max_{0 \le i \le {n \over 2}-1} || h_{m_1} (X_{2i+1}, X_{2i+2})||_{\psi_1}^2 \\
	&\le K_7 \log^2(nd) D_n^2.
	\end{align*}
	Then by \cite[Lemma 8]{cck2014b}, we have $\E Z'_1 \le K_8 [ n \log(nd) D_n^2 ]^{1/2}$.
	Similar to Step 1, taking $t'^* = K_9 \sqrt{D_n^2 \log(nd) \over n}$ for some large enough $K_9>0$, we end up with 
	\[
	\Prob \left( |W_n' - \E W_n'|_\infty \ge t'^* \right)  \le \gamma/2,
	\]
	i.e. $\Prob \left( |\hat{\Gamma}_{n2} - \Gamma_2|_\infty \ge (n-1)^{-1} \cdot t'^* \right) \le \gamma/2$.
	
	\underline{Step 3: Approximating $\hat{\Gamma}_{n}$ to $\Gamma/3$.}	By Cauchy-Schwarz inequality and Condition (A2),
	\begin{align*}
	|\Gamma|_\infty &= \max_{1 \le m_1, m_2 \le d} | \E  h_{m1} (X_1, X_2) \E  h_{m2} (X_1, X_3)| \\
	& \le \max_{ 1\le m_1 \le d} | \E  h_{m1}^2 (X_1, X_2)| \le \max_{ 1\le m_1 \le d} | \E  h_{m1}^4 (X_1, X_2)|^{1/2} \le D_n,\\
	|\Gamma_2|_\infty &= \max_{1 \le m_1, m_2 \le d} | \E  h_{m1} (X_1, X_2) \E  h_{m2} (X_1, X_2)|\\
	&\le  \max_{ 1 \le m_1 \le d} | \E  h_{m1}^2 (X_1, X_2)| \le D_n.
	\end{align*}
	Notice that 
	\[
	|\hat{\Gamma}_{n} - \Gamma/3|_\infty \le |\hat{\Gamma}_{n} - \E \hat{\Gamma}_{n}|_\infty + |\E\hat{\Gamma}_{n} - \Gamma/3|_\infty, 
	\]
	where
	\[
	|\E\hat{\Gamma}_{n} - \Gamma/3|_\infty \le {1 \over 3(n-1)}	|\Gamma|_\infty + {1 \over 2 (n-1)} |\Gamma_2|_\infty \le n^{-1} D_n \le K_{10} \sqrt{D_n^2 \log(nd) \over n}.
	\]
	Combine Step 1 and 2 and take $t_0 = K_0 \sqrt{D_n^2 \log(nd) \over n}$ for some $K_0>K_{10}+K_9+K_5$ large enough, we have
	\[
	\Prob \left(  |\hat{\Gamma}_{n} - \Gamma/3|_\infty \ge t_0 \right) \le \gamma.
	\]
	
\end{proof}

\begin{lem}[Bounding $\max_{1\le l \le d} |\sum_{i=1}^n \sum_{i<j} h_l^2(X_i, X_j) - \E h_l^2(X_i, X_j) |$ under $H_1$.]
	\label{lem:cov_Tn_sharp_H1}
	Suppose all the conditions in Theorem \ref{thm:gaussian_approx_rate} and Theorem \ref{thm:power_signal_rate} hold.
	Let  $\gamma \in (0, e^{-1})$ such that $\log (\gamma^{-1}) \le K \log (nd)$ and suppose $n_1 = m \le n-m=n_2$. Then the following holds with probability greater than $1-\gamma$ for some large enough constant $K^\diamond$
	\[
	\max_{1\le l \le d} |\sum_{i=1}^n \sum_{i<j} h_l^2(X_i, X_j) - \E h_l^2(X_i, X_j) |\le K^\diamond t^\diamond,
	\]
	where $t^\diamond =  D_n^2 n^{3 \over 2} \log^{1 \over 2}(nd) + |\theta_h|^2_\infty [ n_1^{1 \over 2} n_2 \log^{1 \over 2}(nd) + n_2 \log^{3}(nd) \log (\gamma^{-1}) ]$.
\end{lem}
\begin{proof}[Proof of Lemma \ref{lem:cov_Tn_sharp_H1}]
	Note that $h_l^2(x,y) = h_l^2(y,x)$ and the summation breaks down to 
	\[
	\sum_{i=1}^n \sum_{i<j} = \sum_{i=1}^m \sum_{j=i+1}^m + \sum_{i=1}^m \sum_{j=m+1}^n + \sum_{i=m+1}^n \sum_{j=i+1}^n.
	\]
	Apply \cite[Lemma E.1]{chen2018gaussian} to $\hat{\Gamma}_{1} = {1 \over n_1(n_1-1)} \sum_{1 \le i < j \le n_1} h(X_i, X_j) h(X_i, X_j)^T$, calculation (similar to Lemma \ref{lem:cov_Tn_sharp} Step 2) shows 
	\begin{align*}
	\Prob \Big( |\hat{\Gamma}_{1} - \E \hat{\Gamma}_{1}|_\infty \ge &K_1 [ D_n n_1^{-1/2} \log^{1/2} (d) +  D_n^2 n_1^{-1} \log^3{(n_1 d)} ]+ t \Big) \\
	& \qquad \le  \exp\left(- {n_1 t^2 \over 3 D_n^2} \right) + 3 \exp \left[- \left( {\sqrt{n_1 t} \over K_2 D_n \log(n_1 d) } \right)\right].
	\end{align*}
	Take $t_1 = K_3 [ D_n n_1^{-1/2} \log^{1/2} (nd) \vee D_n^2 n_1^{-1} \log^3 (n d) \log (\gamma^{-1}) ]$. It follows that 
	\[
	{n_1 {t_1}^2 \over D_n^2} \gtrsim D_n^2 \log(nd) \gtrsim \log (\gamma^{-1})  \text{ and }  {\sqrt{n_1 t_1} \over D_n \log(n_1 d)} \gtrsim \left( {\log^3 (n d) \log (\gamma^{-1}) \over \log^2 (n_1 d) } \right)^{1/2} \gtrsim \log (\gamma^{-1}).
	\]
	So $\Prob \left( |\hat{\Gamma}_{1} - \E \hat{\Gamma}_{1}|_\infty \ge t_1 \right)	\le  \gamma/3$ for some large enough $K_3$. Therefore, the diagonal part obeys the same bound such that the first term $\sum_{i=1}^m \sum_{j=i+1}^m h_l^2(X_i, X_j)$ has a tail bound
	\[
	\Prob \left( {m \choose 2}^{-1} \max_{1\le l \le d} |\sum_{i=1}^m \sum_{j=i+1}^m h_l^2(X_i, X_j)- \E h_l^2(X_i, X_j) |_\infty \ge t_1 \right) 	\le  \gamma/3.
	\]
	Next, apply the two-sample tail bound Lemma \ref{lem:rate_2sample_Ustat} to the middle term. Thus, 
	\[
	\Prob \left( {1 \over m(n-m)} \max_{1\le l \le d} |\sum_{i=1}^m \sum_{j=m+1}^n h_l^2(X_i, X_j)- \E h_l^2(X_i, X_j) |_\infty \ge t_2 \right) 	\le  \gamma/3
	\]
	holds for $t_2 = K_4 B_n^2 [  {n_1}^{-1/2} \log^{1/2}(n d) \vee {n_1}^{-1} \log^{3}(n d) \log(1/\gamma) ]$, where $B_n = D_n + |\theta_h|_\infty$.
	At last, apply \cite[Lemma E.1]{chen2018gaussian} to $\hat{\Gamma}_{2} = {1 \over n_2(n_2-1)} \sum_{1 \le i < j \le n_2} h(Y_i, Y_j) h(Y_i, Y_j)^T$ for the third term, we have
	\begin{align*}
	\Prob \Big( |\hat{\Gamma}_{2} - \E \hat{\Gamma}_{2}|_\infty \ge & K_5 (D_n^2 n_2^{-1} \log(n_2 d))^{1/2} + t \Big) \\
	&\qquad \le  \exp\left(- {n_2 t^2 \over 3 D_n^2} \right) + 3 \exp \left[- \left( {\sqrt{n_2 t} \over K_6 D_n \log(n_2 d)  } \right)\right].
	\end{align*}
	Since $n_2 = n-m \ge n/2$ and $n \gtrsim D_n^2 \log^7 (nd)$, it suffices to take $t_3 = K_7 D_n n^{-1/2} \log^{1/2} (nd)$ such that
	\[
	{n_2 {t_3}^2 \over D_n^2} \gtrsim \log(n d) \quad \text{and} \quad {\sqrt{n_2 t_3} \over D_n \log(n_2 d)} \gtrsim D_n^{-1/2} n^{1/4} \log^{-3/4}(n d)  \gtrsim  \log (\gamma^{-1} ).
	\]
	Then, the third term has a tail bound
	\[
	\Prob \left( {n-m \choose 2}^{-1} \max_{1\le l \le d} |\sum_{i=m+1}^n \sum_{j=i+1}^n h_l^2(X_i, X_j)- \E h_l^2(X_i, X_j) |_\infty \ge t_3 \right) 	\le  \gamma/3.
	\]
	Since there exists a large enough constant $K^\diamond$ such that
	\begin{align*}
	&(n_1^2 t_1) \lor (n_1 n_2 t_2) \lor (n_2^2 t_3)\\ 
	\le & K^\diamond \left\{ D_n^2 n^{3 \over 2} \log^{1 \over 2}(nd) + |\theta_h|^2_\infty [ n_1^{1 \over 2} n_2 \log^{1 \over 2}(nd) + n_2 \log^{3}(nd) \log (\gamma^{-1}) ]\right\} =: t^\diamond ,
	\end{align*}
	we conclude $\Prob \left( \max_{1\le l \le d} |\sum_{i=1}^n \sum_{i<j} h_l^2(X_i, X_j) - \E h_l^2(X_i, X_j) | \ge 3 t^\diamond \right) \le  \gamma$.
\end{proof}

\subsection{Lemma for tail probability of the maximum of two-sample $U$-statistics}
Let $X_1^{n_1}$ and $Y_1^{n_2}$ be two random samples taking values in a measurable space $(S, \mathcal{S})$. Suppose $X_i \sim F$ are independent with $Y_j \sim G$. Let $h: S^2 \rightarrow \R^d$ be a measurable function and
$$
T_n= {1 \over n_1 n_2} \sum_{i=1}^{n_1} \sum_{j=1}^{n_2} h(X_i, Y_j)
$$
be the two-sample $U$-statistics. WLOG, we may first assume $n_1 \le n_2$. Consider a permutation $\pi_{n_2}$ on $Y_1^{n_2}$ and the sum of first $n_1$ pairs $\sum_{i=1}^{n_1} h(X_i, Y_{\pi_{n_2}(i)})$
\[
\begin{array}{ccc;{2pt/2pt}l}
X_1 & \cdots & X_{n_1}& \nonumber \\
\downarrow   &  & \downarrow  & \nonumber \\
Y_{\pi_{n_2}(1)}& \cdots & Y_{\pi_{n_2}(n_1)}&Y_{\pi_{n_2}(n_1+1)} \;\cdots\; Y_{\pi_{n_2}(n_2)} \nonumber
\end{array}	\]
The symmetry leads to $\sum_{\pi_{n_2}} \sum_{i=1}^{n_1} h(X_i, Y_{\pi_{n_2}(i)}) = (n_2-1)! \sum_{i=1}^{n_1} \sum_{j=1}^{n_2} h(X_i, Y_j)$, i.e.
$$
{1 \over n_2!} \sum_{\pi_{n_2}} \sum_{i=1}^{n_1} h(X_i, Y_{\pi_{n_2}(i)}) = {1 \over n_2} \sum_{i=1}^{n_1} \sum_{j=1}^{n_2} h(X_i, Y_j).
$$
This representation reduce the bounds on $Z = n_1 |T_n - \theta_h|_\infty$ to those of $|V|_\infty = |\sum_{i=1}^{n_1} h(X_i, Y_i) -\theta_h |_\infty$, where $\theta_h = \E h(X_1, Y_1)$. Define
\begin{eqnarray}
\overline{h}(x,y) &=& h(x,y) \vone\{ \max_{1\le k \le d} |h_k (x,y)| \le \tau\}, \tau > 0 \nonumber\\
Z_1 &=& \max_{1\le k \le d} \left| \sum_{i=1}^{n_1} \overline{h}_k (X_i, Y_i) - \E\bar{h}_k \right| \nonumber\\
M &=& \max_{1\le k \le d} \max_{1 \le i \le n_1} |h_k (X_i, Y_i)|  \nonumber\\
\overline{\zeta}_{n_1}^2 &=& \max_{1\le k \le d}  \sum_{i=1}^{n_1} \E h_k^2(X_i, Y_i) \nonumber
\end{eqnarray}
By similar argument of Lemma E.1 in \cite{chen2018gaussian}, we have the following result.
\begin{lem}[Sub-exponential inequality for the maxima of centered two-sample $U$-statistics]
	\label{lem:rate_2sample_general}
	Let $X_1, \cdots X_{n_1}$ and $Y_1, \cdots Y_{n_2}$ be two independent sets of iid random vectors from $F$ and $G$, respectively. Suppose $n_1 \le n_2$ and $||h_k(X_1,Y_1)||_{\psi_\alpha} < \infty$ for $\alpha \in (0,1]$ and all $k = 1, \cdots, d$. Let $\tau = 8\E[M]$, then for any $0 < \eta \le 1$ and $\delta>0$, there exists a constant $C(\alpha,\eta,\delta) > 0$ such that
	\begin{equation}
	\Prob ( Z \ge (1+\eta) \E Z_1 + t) \le \exp \left(-{t^2 \over 2(1+\delta) \overline{\zeta}_{n_1}^2} \right) + 3 \exp \left[ - \left( {t \over C(\alpha,\eta,\delta) ||M||_{\psi_\alpha}} \right)^\alpha \right]
	\end{equation}
	holds for all $t>0$.
\end{lem}
\begin{proof}[Proof of Lemma \ref{lem:rate_2sample_general}]
	See Lemma E.1 in \cite{chen2018gaussian}.
\end{proof}

By Lemma \ref{lem:rate_2sample_general}, we can have the following result.
\begin{lem}[Tail bound of the maxima of two-sample $U$-statistics in second order]
	\label{lem:rate_2sample_Ustat}
	Let $X_1, \cdots X_{n_1}$ and $Y_1, \cdots Y_{n_2}$ be two independent sets of iid random vectors from $F$ and $G$, respectively. Let $\underline{n} = \min\{n_1, n_2\}$, $\overline{n} = \max\{n_1, n_2\}$ and $\zeta \in (0,1)$ be a constant s.t. $\log (\zeta^{-1}) \le K \log (\overline{n} d)$. Suppose $||h_k(X_1,Y_1) - \E h_k (X_1,Y_1)||_{\psi_1}\le D_n$ and $\E |h_k(X_1,Y_1) - \E h_k(X_1,Y_1) |^{2+\ell} \le D_n^\ell$ for all $k = 1, \cdots, d$ and $\ell =1,2$. Denote $B_n = D_n + |\theta_h|_\infty$, where $\theta_h = \E h(X_1,Y_1)$.
	Then,
	\begin{equation}
	\Prob ( \max_{1 \le k \le d} |{1 \over n_1 n_2} \sum_{i=1}^{n_1} \sum_{j=1}^{n_2} h_k^2(X_i, Y_j)- \E h_k^2 (X_i, Y_j)| \ge t^*) \le \zeta
	\end{equation}
	holds for $t^* = K_0  B_n^2 \{  \underline{n}^{-1/2} \log^{1/2}(\overline{n} d) + \underline{n}^{-1} \log^{3}(\overline{n} d) \log(1/\zeta) \}$.
\end{lem}

\begin{proof}[Proof of Lemma \ref{lem:rate_2sample_Ustat}]
	Without loss of generality, we may assume $D_n \ge 1$.
	Let $H_k(x,y) = h_k^2 (x,y), k = 1, \dots, d$, and define $Z$, $Z_1$, $M$ and $\overline{\zeta}_{n_1}^2$ for $H$ accordingly. 
	Apply Lemma \ref{lem:rate_2sample_general} to $H(x,y)$ and follow the fact $||M||_2 \lesssim ||M||_{\psi_{1/2}} = ||\sqrt{M}||^2_{\psi_1}$, we have
	\[
	\Prob (Z \ge 2\E Z_1 + t) \le \exp \left(- {t^2 \over 3\bar{\zeta}_{n_1}^2}\right) + 3\exp  \left[ - \left( {\sqrt{t} \over K_1 ||\sqrt{M}||_{\psi_1}} \right) \right].
	\]
	Note that $||h_k(X_1,Y_1) ||_{\psi_1} \le ||h_k(X_1,Y_1) -\E h_k (X_1,Y_1) ||_{\psi_1} + || \E h_k (X_1,Y_1) ||_{\psi_1} \le D_n + ||\theta_{h,k}||_{\psi_1} = B_n$ and $\E h_k^4 (X_1,Y_1) \lesssim \E |h_k (X_1,Y_1) - \theta_{h,k} |^4 + |\theta_{h,k}|^4 \le D_n^2 + |\theta_{h}|^4_\infty \lesssim B_n^4$. 
	By Lemma 2.2.2 in \cite{vandervaartwellner1996},
	\begin{align*}
	||\sqrt{M}||^2_{\psi_1} &= || \max_{1\le k \le d} \max_{1 \le i \le n_1} |h_k (X_i, Y_i)|  ||^2_{\psi_1} \\
	& \le K_3 (\log(n_1 d) \max_{k,i}  ||h_k(X_i,Y_i) ||_{\psi_1})^2 \\
	& = K_3 \log^2(n_1 d) B_n^2.
	\end{align*}
	Since $\overline{\zeta}_{n_1}^2 = \max_{1\le k \le d}  \sum_{i=1}^{n_1} \E h_k^4(X_i, Y_i) \le n_1 B_n^4$, by Lemma 8 in \cite{cck2014b} and Jensen inequality,
	\[
	\E Z_1 \le K_4 [ \log^{1/2} (d) \overline{\zeta}_{n_1} + \log (d) ||M||_2] \le  K_5 (B_n^2 n_1^{1/2}  \log^{1/2} (n_1 d) +  B_n^2 \log^{3} (n_1 d)).
	\]
	Therefore, 
	\begin{eqnarray}
	\Prob \left( \max_{1 \le k \le d} |{1 \over n_1 n_2} \sum_{i=1}^{n_1} \sum_{j=1}^{n_2} h_k^2(X_i, Y_j)- \E h_k^2| \ge K_5 B_n^2 [ n_1^{-1/2}  \log^{1/2} (d) + n_1^{-1} \log^{3} (n_1 d) ] + t \right)    \nonumber \\
	\le \exp \left(- {n_1 t^2 \over 3 B_n^4 }\right) + 3\exp  \left[ - \left( {\sqrt{n_1 t} \over K_1  K_3 B_n \log(n_1 d)} \right) \right] \nonumber
	\end{eqnarray}
	Recall $\underline{n} = n_1$ and $\overline{n} = n_2$.\\
	(i) If $\underline{n} \ge K_6  \log^5 (\overline{n} d) \log^2 (1/\zeta)$, then take $t_1^* = K B_n^2 \underline{n}^{-1/2} \log^{1/2}(\overline{n} d)$ such that 
	\[
	{n_1 {t_1^*}^2 \over B_n^4} = \log (\overline{n} d) \gtrsim  \log(1/\zeta) \text{ and } { \sqrt{n_1 t_1^*} \over B_n \log(n_1 d)} \ge  \underline{n}^{1/4} \log^{-3/4} (\overline{n} d) \gtrsim \log(1/\zeta).
	\]
	(ii) If $\underline{n} \le K_6  \log^5 (\overline{n} d) \log^2 (1/\zeta)$, then take $t_2^* = K B_n^2 \underline{n}^{-1} \log^{3}(\overline{n} d) \log(1/\zeta)$ such that
	\begin{align*}
	&{n_1 {t_2^*}^2 \over B_n^4} \ge \underline{n}^{-1} \log^6 (\overline{n} d) \log^2 (1/\zeta) \gtrsim \log (1/\zeta)  \qquad \text{ and } \\
	&{\sqrt{n_1 t_2^*} \over B_n \log(n_1 d)} = \log^{1/2} (\overline{n} d) \log^{1/2} (1/\zeta) \gtrsim \log(1/\zeta).
	\end{align*}
	Observing $B_n^2 [ n_1^{-1/2}  \log^{1/2} (d) +  n_1^{-1} \log^{3} (n_1 d) ] \lesssim t_1^* + t_2^* =: t^*$. Hence, 
	\[
	\Prob ( \max_{1 \le k \le d} |{1 \over n_1 n_2} \sum_{i=1}^{n_1} \sum_{j=1}^{n_2} h_k^2(X_i, Y_j)- \E h_k^2| \ge t^*) \le \zeta.
	\]
	
\end{proof}

\subsection{Lemma for two-sample Hoeffding decomposition}

\begin{lem}[Tail bound of the maxima of the first order projection]
	\label{lem:rate_proj_1sample}
	Let $X_1, \dots, X_{n}$ be i.i.d.\ random vectors from $F$ and $Y$ is independently draw from $G$. Suppose  $\theta_h = \E h(X_1, Y)$, $||h_k(X_1,Y) - \theta_{hk}||_{\psi_1}\le D_n$ and $\E |h_k(X_1,Y)  - \theta_{hk}|^{2+\ell} \le D_n^\ell$ for all $k = 1, \dots, d$ and $\ell = 1,2$. Let $\zeta \in (0,1)$ be a constant s.t. $\log (\zeta^{-1}) \le K \log (nd)$. Define the projection $Gh(x) = \E h(x, Y) - \theta_{h}$. Then,
	\begin{equation*}
	\Prob \left( | \sum_{i=1}^{n}  Gh(X_i) |_\infty \ge K D_n \{ n^{1/2} \log^{1/2}(nd) \vee \log^2(nd)\} \right) \le \zeta.
	\end{equation*}
	Therefore when $n \gtrsim \log^3 (nd)$,
	\begin{equation*}
	\Prob \left( | \sum_{i=1}^{n}  Gh(X_i) |_\infty \ge K D_n n^{1/2} \log^{1/2}(nd) \right) \le \zeta.
	\end{equation*}
\end{lem}

\begin{proof}[Proof of Lemma \ref{lem:rate_proj_1sample}]
	Let $Z = \max_{1 \le k \le d} |\sum_{i=1}^n [ Gh_k (X_i)]|$, $\sigma^2 = \max_{1 \le k \le d} \sum_{i=1}^n \E [ Gh_k (X_i)]^2$ and $M = \max_{1 \le i \le n} \max_{1 \le k \le d} | Gh_k  (X_i)|$. By \cite[Theorem 4]{adamczak2008}, 
	\begin{equation*}
	\Prob \left( Z \ge 2 \E Z + t \right) \le \exp {(-{t^2 \over 3 \sigma^2})} + 3 \exp {(-{t \over K_1||M||_{\psi_1}})}. 
	\end{equation*}
	By Jensen inequality,  $\E |Gh_k(X_i)|^{2} = \E | \E [h_k (X_i, Y) - \theta_{hk}|X_i]|^{2} \le\E |h_k(X_i, Y) - \theta_{hk}|^{2} \le D_n$  and  $||Gh_k(X_i)||_{\psi_1}  \le ||h_k(X_i, Y) - \theta_{hk}||_{\psi_1} \le D_n$. So  $\sigma^2 \le n D_n$. By  \cite[Lemma 2.2.2]{adamczak2008} and \cite[Lemma 8]{cck2014b},
	\begin{eqnarray}
	&& ||M||_{\psi_1} \le K_2 \log(nd) \max_{i,k} ||Gh_k(X_i)||_{\psi_1} \le K_2 D_n \log(nd) \quad \text{ and}  \nonumber\\
	&& \E Z \le K_3 \{\sigma \sqrt{\log d} + ||M||_{\psi_1} \log d \} \le K_4 \{ \sqrt{n \log(d) D_n} + \log(nd) \log (d) D_n\}. \nonumber
	\end{eqnarray}
	Take $t^* = K_5 D_n \{ n^{1/2} \log^{1/2}(nd) \vee \log^2(nd)\}$, simple calculation shows $\Prob (Z \ge t^*) \le \zeta$. 
\end{proof}

\begin{lem}[Maximal inequality for canonical two-sample $U$-statistics]
	\label{lem:E_2sample_inf_Ustat}
	Let $X_1, \dots, X_{n_1}$ and $Y_1, \dots, Y_{n_2}$ be two independent sets of iid random vectors from $F$ and $G$, respectively. Let $\theta_h = \E h(X_1,Y_1)$, $n_1 \le n_2$ and $d \ge 2$. Suppose $||h_m(X_1,Y_1) - \theta_{h,m}||_{\psi_1}\le D_n$ and $\E |h_m(X_1,Y_1)- \theta_{h,m}|^{2+\ell} \le D_n^\ell$ for all $m = 1, \dots, d$ and $\ell =1,2$. We have
	\begin{align*}
	&\E | \sum_{i=1}^{n_1} \sum_{j=1}^{n_2} \breve{f}(X_i, Y_j) |_\infty \\
	\le & K D_n \log (d) \left\{ \log(d) \log (n_2 d) + (n_1n_2)^{1/2} + [n_2 \log(d) \log^2 (n_2 d) ]^{1/2} + [n_1 n_2^2 \log (d) ]^{1/4} \right\}.
	\end{align*}
\end{lem}
\begin{proof}[Proof of Lemma \ref{lem:E_2sample_inf_Ustat}]
	The structure of this proof is similar to the one-sample version in \cite[Thm 5.1]{chen2018gaussian}.
	By constructing randomization from iid Rademacher random variables (i.e. $\Prob(\epsilon_i = \pm 1) ={1 \over 2}$ for all $\epsilon_i$ and $\epsilon_j'$,  $i = 1, \dots, n_1, j=1, \dots, n_2$ ), \cite[Thm 3.5.3]{delaPenaGine1999} shows
	\begin{equation*}
	\E | \sum_{i=1}^{n_1} \sum_{j=1}^{n_2} \breve{f}(X_i,Y_j)|_\infty \le K_1 \E| \sum_{i=1}^{n_1} \sum_{j=1}^{n_2} \breve{f}(X_i,Y_j) \epsilon_i \epsilon_j'|_\infty
	\end{equation*}
	Fix an $m=1, \dots, d$. Let $\Lambda^m$ be a $(n_1+n_2)$-by-$(n_1+n_2)$ matrix with zero diagonal blocks, where $\Lambda^m_{ij} = \breve{f}_m(X_i,Y_{j-n_1})$ if $1 \le i \le n_1, n_1+1\le j \le n_1+n_2$ and $\Lambda^m_{ij} = 0, \ otherwise$. Apply Hanson-Wright inequality \cite[Thm 1]{rudelson2013hanson} conditioning on $X_1^{n_1}$ and $Y_1^{n_2}$,
	\begin{equation*}
	\Prob\left( \epsilon^T \Lambda^m \epsilon | X_1^{n_1} Y_1^{n_2}\right) \le 2 \exp [ -K_2 \min \{ {t^2 \over |\Lambda^m|_F^2}, {t \over ||\Lambda^m||_2 } \} ],
	\end{equation*}
	where $\epsilon^T = (\epsilon_1, \dots, \epsilon_{n_1}, \epsilon'_1, \dots, \epsilon'_{n_2})$ and $t>0$. 
	Denote $V_1 = \max_{1 \le m \le d} |\Lambda^m|_F$ and $V_2 = \max_{1 \le m \le d} ||\Lambda^m||_2$. Let 
	\begin{equation*}
	t^* = \max \{ V_1 \sqrt{\log d \over K_2}, V_2 {\log d \over K_2} \},
	\end{equation*}
	such that 
	\begin{align*}
	\E [ \max_{1 \le m \le d} |\epsilon^T \Lambda^m \epsilon| | X_1^{n_1}, Y_1^{n_2} ] &= \int_{0}^{\infty} \Prob\left( \max_{1 \le m \le d} |\epsilon^T \Lambda^m \epsilon| \ge t | X_1^{n_1}, Y_1^{n_2}\right) dt\\ &\le t^* + 2d \int_{t^*}^{\infty} \max \{ \exp {(-{K_2 t^2 \over V_1^2})}, \exp {(- {K_2 t \over V_2})} \}.
	\end{align*}
	Apply the tail bound of standard Gaussian random variables $1-\Phi(x) \le \phi(x)/x$ for $x>0$, and note that $d \ge 2$, we have 
	\begin{equation*}
	2d \int_{t^*}^{\infty} \exp {(-{K_2 t^2 \over V_1^2})} dt \le {V_1 \over \sqrt{2K_2}} \int_{\sqrt{2 \log d}}^{\infty} \exp {(-{s^2 \over 2})} ds \le {V_1 \over \sqrt{K_2 \log d}} \le K_2 V_1.
	\end{equation*}
	Similarly,
	\begin{equation*}
	2d \int_{t^*}^{\infty} \exp {(- {K_2 t \over V_2})} dt \le 2V_2/K_2.
	\end{equation*}
	By Jensen's inequality and the fact $V_2 \le V_1$, we have
	\begin{align}
	\label{eqn:ramdomization_f}
	\E | \sum_{i=1}^{n_1} \sum_{j=1}^{n_2} \breve{f}(X_i,Y_j) \epsilon_i \epsilon_j'|_\infty &\le K_1 \E [t^* + K_2 V_1 + 2V_2/K_2] \le K_3 (\log d ) \E V_1 \nonumber \\ 
	&\le K_3 (\log d ) (\E \max_{1 \le m \le d} |\Lambda^m|_F^2)^{1/2}  .
	\end{align}
	Our last task is to bound $I \overset{def}{=}\E \max_{1 \le m \le d} |\Lambda^m|_F^2 = \E [\max_{1 \le m \le d} \sum_{i=1}^{n_1} \sum_{j=1}^{n_2} \breve{f}_m^2(X_i,Y_j) ] $. Consider Hoeffding decomposition of $\breve{f}_m^2$, 
	\begin{equation*}
	\breve{f}_0^m (x_1,y_1) = \breve{f}_m^2 (x_1,y_1) - \breve{f}_1^m(x_1) - \breve{f}_2^m (y_1) - \E \breve{f}_m^2,
	\end{equation*}
	where $\breve{f}_1^m(x_1) = \E \breve{f}_m^2 (x_1,Y) - \E \breve{f}_m^2$ and $\breve{f}_2^m(y_1) = \E \breve{f}_m^2 (X,y_1) - \E \breve{f}_m^2$ for $X \sim F \indep Y \sim G$ are two random vectors independent from $X_1^{n_1}, Y_1^{n_2}$, and all $x_1, y_1$ from the measurable space of $F$ and $G$, respectively. Then,
	\begin{align}
	& \E [\max_{1 \le m \le d} \sum_{i=1}^{n_1} \sum_{j=1}^{n_2} \breve{f}_m^2(X_i,Y_j) ] \nonumber\\
	=& \E [\max_{1 \le m \le d} \sum_{i=1}^{n_1} \sum_{j=1}^{n_2} \breve{f}_0^m (X_i,Y_j) + \breve{f}_1^m(X_i) + \breve{f}_2^m (Y_j) + \E \breve{f}_m^2 ] \quad\quad \nonumber\\
	\le&  \E [| \sum_{i=1}^{n_1} \sum_{j=1}^{n_2} \breve{f}_0^m(X_i,Y_j) |_\infty] + n_2 \E [| \sum_{i=1}^{n_1} \breve{f}_1^m(X_i) |_\infty]   + n_1 \E [| \sum_{j=1}^{n_2} \breve{f}_2^m (Y_j) |_\infty] + n_1 n_2 \max_{1 \le m \le d} \E \breve{f}_m^2.  
	\label{eqn:bound_I}
	\end{align}
	Note that, conditioning on $X_1^{n_1}$, Hoeffding inequality shows for $t>0$
	\begin{equation*}
	\Prob \left( |\sum_{i=1}^{n_1} \breve{f}_1^m(X_i) \epsilon_i| >t |X_1^{n_1} \right) \le 2 \exp {(-{t^2 \over 2\sum_{i=1}^{n_1} \breve{f}_1^m(X_i)^2})}.
	\end{equation*}
	Denote $M = \max_{i, j, m} |\breve{f}_m (X_i, Y_j)|$. Following arguments in beginning and the symmetrization inequality \cite[Lemma 2.3.1]{vandervaartwellner1996}, we have
	\begin{align}
	&\E | \sum_{i=1}^{n_1} \breve{f}_1(X_i) |_\infty \le \sqrt{\log d} \ \E \sqrt{\max_m \sum_{i=1}^{n_1} \breve{f}_1^m(X_i)^2} \le K_4 \sqrt{\log d} \sqrt{n_1 \max_m \E \breve{f}_m^4 + \log d ||M||_4^4}, \label{eqn:a1}\\
	&\E | \sum_{j=1}^{n_2} \breve{f}_2 (Y_j) |_\infty \le  \sqrt{\log d} \ \E \sqrt{\max_m \sum_{j=1}^{n_2} \breve{f}_2^m (Y_j)^2} \le K_5 \sqrt{\log d} \sqrt{n_2 \max_m \E \breve{f}_m^4 + \log d ||M||_4^4},  \label{eqn:a2}\\
	&\E | \sum_{i=1}^{n_1} \sum_{j=1}^{n_2} \breve{f}_0(X_i,Y_j) |_\infty \le \log d \ \E \sqrt{\max_m \sum_{i=1}^{n_1} \sum_{j=1}^{n_2} \breve{f}_0^m(X_i,Y_j)^2}  \le K_6 \log d \sqrt{I} ||M||_2.  \label{eqn:a3}
	\end{align}
	The last step of (\ref{eqn:a1}) comes from  \cite[Equation (58)]{chen2018gaussian}. The (\ref{eqn:a2}) follows the same procedure. And the first step of (\ref{eqn:a3}) is dealt the same way as (\ref{eqn:ramdomization_f}) with
	\begin{align}
	\E \sqrt{\max_m \sum_{i=1}^{n_1} \sum_{j=1}^{n_2} \breve{f}_0^m(X_i,Y_j)^2}  & \le 2 \Big[\E \sqrt{\max_m \sum_{i,j} \breve{f}_m^4(X_i,Y_j)} + \E \sqrt{\max_m \sum_{i,j} (\E [\breve{f}_m^2(X_i,Y_j')|X_1^{n_1}] )^2}  \nonumber \\
	+  \E & \sqrt{\max_m \sum_{i,j} (\E [\breve{f}_m^2(X_i',Y_j)|Y_1^{n_2}] )^2} + \E \sqrt{\max_m \sum_{i,j} (\E \breve{f}_m^2(X_i,Y_j) )^2}  \Big] \nonumber\\
	\le&  K_6 \sqrt{I} \sqrt{\E M^2}. \nonumber
	\end{align} 
	Since $||h_m(X_1,Y_1) - \theta_{h,m}||_{\psi_1}\le D_n$ and $\E |h_m(X_1,Y_1)- \theta_{h,m}|^{2+\ell} \le D_n^\ell$, we know $\max_m \E \breve{f}_m^4 \le D_n^2$ and $||M||_4 \lesssim ||M||_{\psi_{1}} \le K_7 D_n \log(n_1 n_2 d)  \le 2 K_7 D_n \log (n_2 d)$. Besides, we have $D_q = \max_m [\E |\breve{f}_m(X,Y)|^q]^{1/q} \lesssim D_n$.
	Plug (\ref{eqn:a1})-(\ref{eqn:a3}) in (\ref{eqn:bound_I}) and the solution of quadratic inequality for $I$ gives   
	\begin{align*}
	I \le K_8 \Big\{ ||M||_2^2 \log^2 d + n_1n_2 D_2 + n_2 \sqrt{\log d} \sqrt{n_1D_4 + \log d ||M||_4^4} \qquad\qquad\\
	+ n_1 \sqrt{\log d} \sqrt{n_2D_4 + \log d ||M||_4^4} \Big\}.
	\end{align*}
	Therefore, the square-root of $I$ is less than the square-root of each term on RHS. Plug the result in \ref{eqn:ramdomization_f}. A simplified result is obtained in the statement of Lemma \ref{lem:E_2sample_inf_Ustat}. 
\end{proof}

\subsection{Additional simulation and tables}
\label{subsec:app_additional_sim_results}

\begin{table}[!htp]
	\vskip .2cm
	\setlength\tabcolsep{4.5pt}
	\caption{Powers report of our method using \textit{linear} kernel. Here, $n = 500, p = 600, \alpha = 0.05$ and change point locations are $t_m = m/n = 5/10, 3/10, 1/10$.}
	\label{tab:power_linear}
	\begin{tabular}{c|ccc|ccc|ccc}
		\hline
		& \multicolumn{3}{c|}{Gaussian} & \multicolumn{3}{c|}{$t_6$} & \multicolumn{3}{c}{ctm-Gaussian} \\\cline{2-10}
		$|\theta|_\infty$ & I        & II      & III     & I       & II     & III    & I         & II        & III      \\ \hline
		\multicolumn{10}{c}{$t_m=5/10$}                                                                                          \\\hline
		0                          & 0.042    & 0.050   & 0.032   & 0.058   & 0.060  & 0.040  & 0.052     & 0.050     & 0.048    \\
		0.28                       & 0.100    & 0.178   & 0.082   & 0.082   & 0.134  & 0.072  & 0.066     & 0.102     & 0.070    \\
		0.44                       & 0.436    & 0.628   & 0.390   & 0.186   & 0.420  & 0.212  & 0.154     & 0.356     & 0.200    \\
		0.63                       & 0.886    & 0.970   & 0.896   & 0.610   & 0.828  & 0.590  & 0.554     & 0.810     & 0.578    \\
		0.84                       & 0.996    & 1       & 0.996   & 0.926   & 0.988  & 0.912  & 0.918     & 0.990     & 0.910    \\
		\hline
		\multicolumn{10}{c}{$t_m=3/10$}                                                                                          \\\hline
		0                          & 0.030    & 0.042   & 0.066   & 0.038   & 0.060  & 0.026  & 0.030     & 0.072     & 0.060    \\
		0.28                       & 0.088    & 0.216   & 0.108   & 0.068   & 0.124  & 0.036  & 0.036     & 0.156     & 0.082    \\
		0.44                       & 0.414    & 0.738   & 0.384   & 0.222   & 0.418  & 0.178  & 0.150     & 0.440     & 0.200    \\
		0.63                       & 0.890    & 0.996   & 0.908   & 0.594   & 0.878  & 0.634  & 0.524     & 0.846     & 0.570    \\
		0.84                       & 0.998    & 1       & 0.998   & 0.930   & 0.998  & 0.960  & 0.940     & 0.996     & 0.940    \\
		\hline
		\multicolumn{10}{c}{$t_m=1/10$}                                                                                          \\\hline
		0                          & 0.054    & 0.060   & 0.050   & 0.064   & 0.058  & 0.060  & 0.054     & 0.054     & 0.064    \\
		0.63                       & 0.082    & 0.210   & 0.086   & 0.078   & 0.126  & 0.082  & 0.058     & 0.118     & 0.086    \\
		0.84                       & 0.190    & 0.472   & 0.224   & 0.144   & 0.278  & 0.120  & 0.116     & 0.240     & 0.120    \\
		1.08                       & 0.446    & 0.768   & 0.446   & 0.268   & 0.492  & 0.252  & 0.208     & 0.470     & 0.230    \\
		1.35                       & 0.756    & 0.966   & 0.770   & 0.486   & 0.762  & 0.516  & 0.444     & 0.760     & 0.462    \\
		2.00                       & 0.998    & 1.000   & 0.998   & 0.954   & 0.996  & 0.960  & 0.962     & 0.994     & 0.956   \\\hline
	\end{tabular}
\end{table}

\begin{table}[!htp]
	\caption{Powers report of our method using \textit{sign} kernel. Here, $n = 500, p = 600, \alpha = 0.05$ and change point locations are $t_m = m/n = 5/10, 3/10, 1/10$.}
	\label{tab:power_sign}
	\vskip .2cm
	\setlength\tabcolsep{2.5pt}
	\begin{tabular}{c|ccc|ccc|ccc||c|ccc}
		\hline
		& \multicolumn{3}{c}{Gaussian} & \multicolumn{3}{c}{$t_6$} & \multicolumn{3}{c||}{ctm-Gaussian} &                            & \multicolumn{3}{c}{Cauchy} \\\cline{2-10}\cline{12-14}
		$|\theta|_\infty$ & I        & II      & III     & I       & II     & III    & I         & II        & III      & $|\theta|_\infty$ & I       & II      & III    \\\hline
		\multicolumn{14}{c}{$t_m=5/10$}                                          \\\hline
		0                          & 0.056    & 0.043   & 0.048   & 0.066   & 0.062  & 0.066  & 0.067     & 0.032     & 0.055    & 0                          & 0.054   & 0.062   & 0.039  \\
		0.28                      & 0.136    & 0.289   & 0.147   & 0.110   & 0.229  & 0.099  & 0.105     & 0.204     & 0.083    & 0.71                      & 0.403   & 0.651   & 0.432  \\
		0.44                      & 0.566    & 0.870   & 0.624   & 0.452   & 0.738  & 0.479  & 0.364     & 0.674     & 0.397    & 1.23                      & 0.971   & 1       & 0.981  \\
		0.63                      & 0.977    & 1       & 0.971   & 0.915   & 0.996  & 0.913  & 0.854     & 0.980     & 0.872    & 1.91                      & 1       & 1       & 1      \\
		0.84                      & 1        & 1       & 1       & 0.998   & 1      & 1      & 0.988     & 1         & 0.998    & 2.79                      & 1       & 1       & 1      \\\hline
		\multicolumn{14}{c}{$t_m=3/10$}                                          \\\hline
		0                          & 0.049    & 0.037   & 0.047   & 0.039   & 0.068  & 0.056  & 0.051     & 0.049     & 0.055    & 0                          & 0.055   & 0.035   & 0.065  \\
		0.28                      & 0.070    & 0.154   & 0.068   & 0.058   & 0.148  & 0.078  & 0.073     & 0.104     & 0.083    & 0.71                      & 0.257   & 0.386   & 0.280  \\
		0.44                      & 0.342    & 0.619   & 0.342   & 0.218   & 0.451  & 0.230  & 0.189     & 0.427     & 0.240    & 1.23                      & 0.829   & 0.969   & 0.876  \\
		0.63                      & 0.830    & 0.982   & 0.848   & 0.663   & 0.912  & 0.706  & 0.593     & 0.872     & 0.628    & 1.91                      & 1       & 1       & 1      \\
		0.84                      & 0.992    & 1       & 0.996   & 0.975   & 1      & 0.973  & 0.941     & 0.994     & 0.945    & 2.79                      & 1       & 1       & 1      \\\hline
		\multicolumn{14}{c}{$t_m=1/10$}                                       \\\hline
		0                          & 0.042    & 0.046   & 0.065   & 0.053   & 0.046  & 0.046  & 0.050     & 0.048     & 0.050    & 0                          & 0.057   & 0.059   & 0.080  \\
		0.63                      & 0.078    & 0.139   & 0.082   & 0.063   & 0.107  & 0.078  & 0.060     & 0.110     & 0.075    & 1.91                      & 0.216   & 0.394   & 0.243  \\
		0.84                      & 0.147    & 0.309   & 0.155   & 0.097   & 0.231  & 0.132  & 0.104     & 0.218     & 0.110    & 2.79                      & 0.410   & 0.680   & 0.433  \\
		1.08                     & 0.305    & 0.580   & 0.336   & 0.214   & 0.458  & 0.248  & 0.183     & 0.423     & 0.222    & 3.95                      & 0.627   & 0.873   & 0.647  \\
		1.35                      & 0.523    & 0.796   & 0.588   & 0.405   & 0.706  & 0.439  & 0.367     & 0.660     & 0.351    & 5.47                      & 0.806   & 0.931   & 0.806  \\
		2.00                      & 0.891    & 0.992   & 0.931   & 0.794   & 0.964  & 0.834  & 0.815     & 0.950     & 0.828    & 10.02                     & 0.937   & 0.980   & 0.933 \\\hline
	\end{tabular}
\end{table}

{
	
	In this section, we test the performance of the WBS-type procedure (Algorithm~\ref{alg:WBS_Ustat}). Let $n=500, p=600, B=200, B_W = 200, n'= 0.2n=100$ and data be i.i.d.\ Gaussian distributed with covariance structure III. The two change points are $(m_1, m_2) = (150,300)$ and only the $k$-th component of the $k$-th change point has signal $\theta_1^{(1)} = \theta_2^{(2)} = \delta \neq 0$.
	The powers along $\delta$ for each $\alpha = 0.01, 0.05, 0.1$ are shown in the rows of Table~\ref{tab:WBS_power}.
	We find that when $\delta=0$, the power is close to the nominal levels, respectively. Besides, the power grows as $\delta$ increases. 

	\begin{table}[htb]
		\caption{Power of $U$-statistics based WBS-type testing.}
		\label{tab:WBS_power}
		\centering
		\begin{tabular}{l|lllll}
			\hline
			\multirow{2}{*}{Power} & \multicolumn{5}{c}{$\delta$}          \\ \cline{2-6}
			& 0     & 0.317 & 0.733 & 1.282 & 2.004 \\ \hline
			$\alpha=0.01$ & 0.012 & 0.046 & 0.806 & 0.900 & 0.908 \\
			$\alpha=0.05$ & 0.032 & 0.100 & 0.818 & 0.898 & 0.926 \\
			$\alpha=0.1$  & 0.088 & 0.198 & 0.882 & 0.926 & 0.938 \\ \hline
		\end{tabular}
	\end{table}
}

\subsection{Additional comparisons with \texttt{BABS} and \texttt{Jirak}}
\label{subsec:app_additional_comparisons}

{
	
	We further compare the size control of \cite[\texttt{BABS}]{yuchen2017finite}, \cite[\texttt{Jirak}]{jirak2015} and our linear kernel approach under $H_0$. As suggested, we fix $p=100$ and vary $n$ from 50 to 300. The bootstrap repeat is $B=200$, $\xi_i$ are i.i.d.\ Gaussian with dependence structure III, and each simulation repeats 500 times. The boundary removal parameters in \texttt{BABS} and \texttt{Jirak} are both $0.1n$.
	
	From Figure~\ref{fig:SmallNFixedP_subfig:a}, we can find that all three methods have a decreasing trend when $n$ grows, but our $U$-statistic approach has the lowest uniform error-in-size $\sup_{\alpha \in [0,1]} |\hat{R}(\alpha) - \alpha|$ under each choice of $n$. This confirms that our $U$-statistic test performs better than the others for small $n$. 
	From Figure~\ref{fig:SmallNFixedP_subfig:b} where empirical rejection rates at $\alpha=0.05, 0.1$ are provided, we may observe that the difference among three methods diminishes for $n=300$. However, our approach is closer to the corresponding nominal significance level except for $n=50$. 
	Therefore, the simulation indicates that the no-boundary-removal property in our proposed test is beneficial to size control under small sample size.

	\begin{figure}[htb]
		\centering
		\begin{subfigure}[b]{.6\textwidth}
			\includegraphics[trim = 5 10 30 10, clip,width=\textwidth]{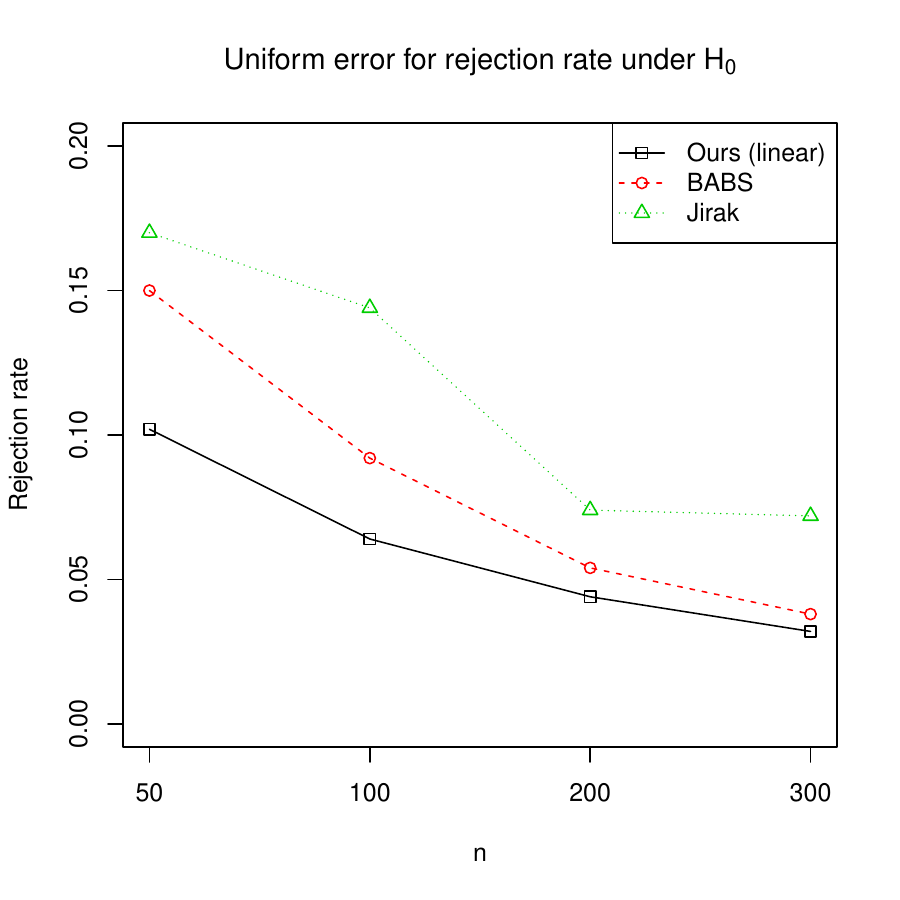}
			\caption{Uniform error-in-size, $\sup_{\alpha \in [0,1]} |\hat{R}(\alpha) - \alpha|$ under $H_0$.}
			\label{fig:SmallNFixedP_subfig:a}
		\end{subfigure}\qquad
		\begin{subfigure}[b]{.35\textwidth}
			\includegraphics[trim = 5 10 30 10, clip,width=\textwidth]{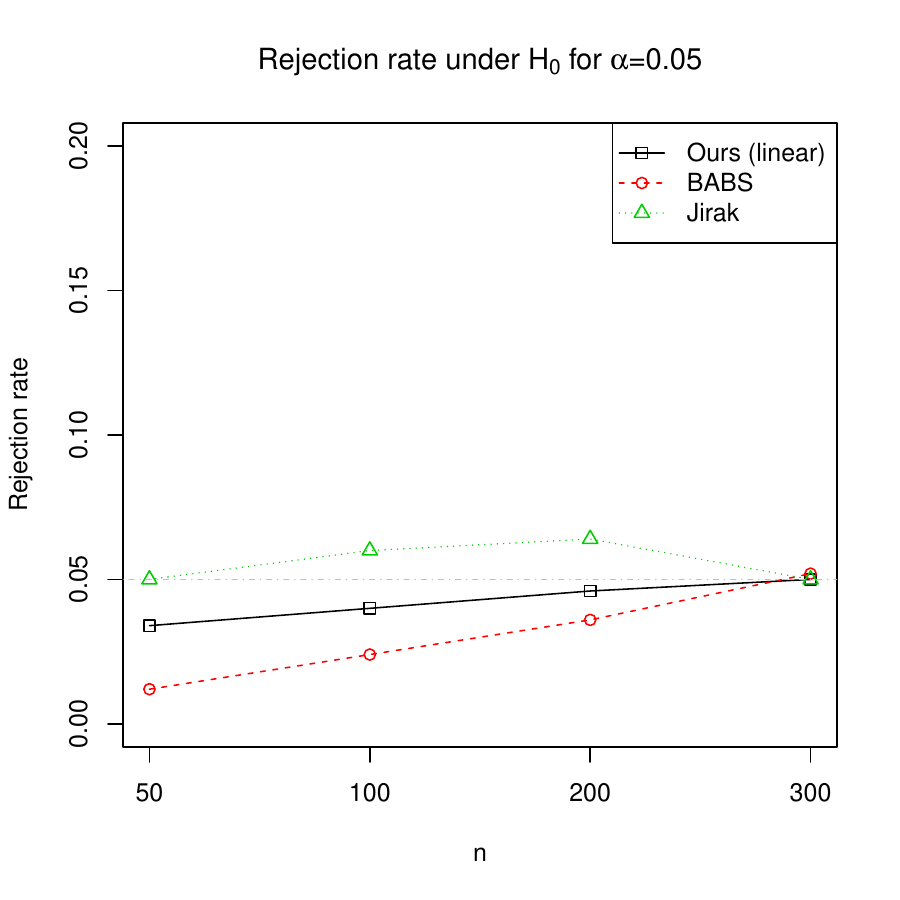}
			\includegraphics[trim = 5 10 30 10, clip,width=\textwidth]{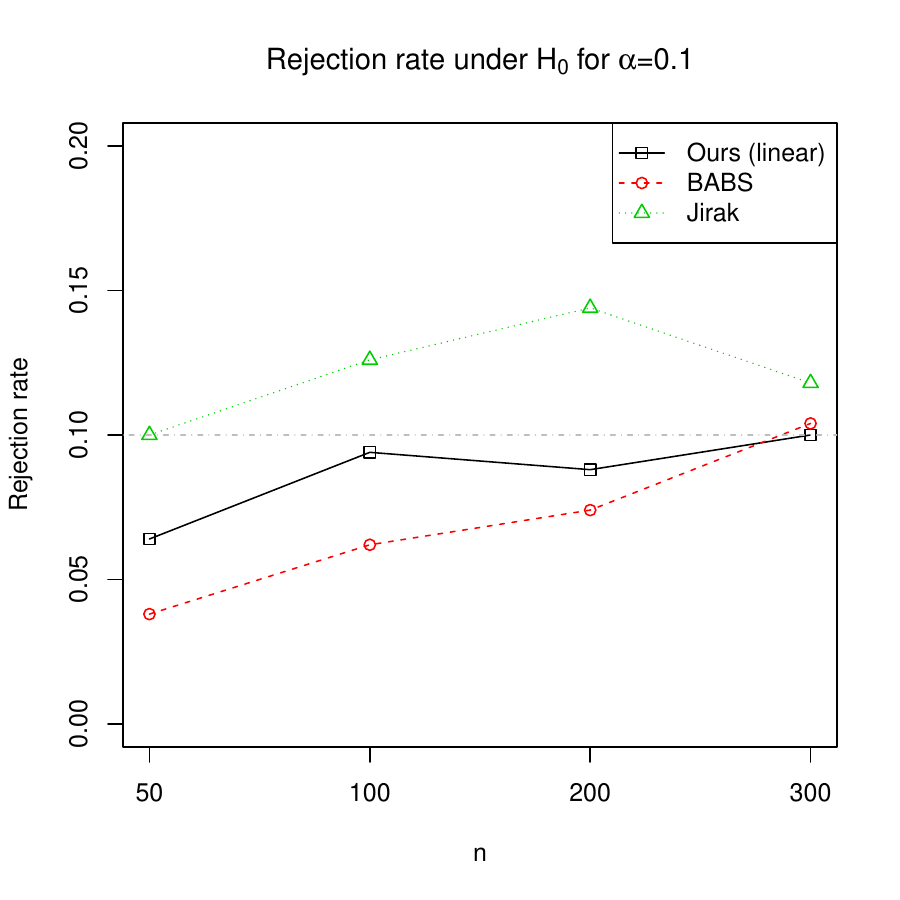}
			\caption{Empirical $\hat{R}(\alpha)$ at $\alpha=0.05, 0.1$.}
			\label{fig:SmallNFixedP_subfig:b}
		\end{subfigure}
		\caption{Comparison of size control among \texttt{BABS}, \texttt{Jirak} and our method using linear kernel. }
		\label{fig: SmallNFixedP}
	\end{figure}
}

\FloatBarrier

 \begin{acks}[Acknowledgments]
	Research partially supported by NSF DMS-1404891, NSF CAREER Award DMS-1752614, and University of Illinois at Urbana-Champaign (UIUC) Research Board Awards (RB17092,  RB18099). This work is completed in part with the high-performance computing resource provided by the Illinois Campus Cluster Program at UIUC.
The authors are grateful to the editor, associate editor, and referee for their insightful comments.
 \end{acks}


\clearpage
\bibliographystyle{imsart-number} 
\bibliography{hdcp_u2021}       


\end{document}